\documentclass[a4paper, 10pt, reqno]{amsart}
\usepackage[T1]{fontenc}
\usepackage{lmodern}
\usepackage[utf8]{inputenc}

\usepackage{amsmath,amssymb,graphicx,mathtools} 
\usepackage{amsthm}
\usepackage[foot]{amsaddr}
\usepackage[a4paper,left=2.6cm,right=2.6cm,top=3cm,bottom=3.5cm]{geometry}
\usepackage{enumerate}
\usepackage{enumitem}
\usepackage{tikz} 
\usetikzlibrary{arrows,decorations.markings}
\usepackage[colorlinks, linkcolor=red, urlcolor=blue, citecolor=blue]{hyperref}
\usepackage{url}
\usepackage{comment}

\usepackage[
    giveninits=true,
    url=false,
    doi=false,
    isbn=false,
    eprint=true,
    datamodel=mrnumber,
    sorting=nyt, 
    sortcites=false, 
    maxbibnames=99,
    maxcitenames=4,
    backref=false,
    block=space,
    backend=biber,
    style=phys,
    biblabel=brackets
]{biblatex}
\AtEveryBibitem{\clearfield{month}}
\AtEveryCitekey{\clearfield{month}} 
\renewbibmacro{in:}{}
\DeclareFieldFormat[article]{title}{\emph{#1}} 
\DeclareFieldFormat{mrnumber}{\ifhyperref{\href{http://www.ams.org/mathscinet-getitem?mr=#1}{\nolinkurl{MR#1}}}{\nolinkurl{#1}}}
\DeclareFieldFormat{pmid}{\ifhyperref{\href{https://www.ncbi.nlm.nih.gov/pubmed/#1}{\nolinkurl{PMID#1}}}{\nolinkurl{#1}}}
\DeclareFieldFormat{eprint}{\ifhyperref{\href{https://arxiv.org/abs/#1}{\nolinkurl{arXiv:#1}}}{\nolinkurl{#1}}}
\renewbibmacro*{doi+eprint+url}{%
  \iftoggle{bbx:doi}{\printfield{doi}}{}
  \newunit\newblock%
  \printfield{mrnumber}%
  \newunit\newblock%
  \printfield{pmid}%
  \newunit\newblock%
  \printfield{eprint}%
  \iftoggle{bbx:url}{\usebibmacro{url+urldate}}{}}
\bibliography{refs}
\usepackage{caption}
\usepackage{subcaption}

\DeclareMathOperator{\Tr}{Tr}

\DeclareMathOperator*{\argmax}{arg\,max}
\DeclareMathOperator*{\argmin}{arg\,min}

\newcommand{\norm}[1]{\lVert#1\rVert}
\newcommand{\R}{\mathbb{R}} 
\newcommand{\E}{\mathbb{E}}
\newcommand{\N}{\mathbb{N}} 

\theoremstyle{plain}
\newtheorem{thm}{Theorem}[section]

\newtheorem{lem}[thm]{Lemma}
\newtheorem{prop}[thm]{Proposition}

\theoremstyle{definition}
\newtheorem{defn}[thm]{Definition}

\theoremstyle{remark}
\newtheorem{rmk}[thm]{Remark}

\numberwithin{equation}{section}

\makeatletter
\renewcommand\subsection{\@startsection{subsection}{2}%
  \z@{-.5\linespacing\@plus-.7\linespacing}{.5\linespacing}%
  {\normalfont\scshape}}
\renewcommand\subsubsection{\@startsection{subsubsection}{3}%
  \z@{.5\linespacing\@plus.7\linespacing}{-.5em}%
  {\normalfont\scshape}}
\makeatother

\makeatletter
\@namedef{subjclassname@2020}{%
  \textup{2020} Mathematics Subject Classification}
\makeatother

\usepackage[format=plain,
            labelfont=sc]{caption}
\usepackage[toc,page]{appendix} 
\usepackage{scalerel,stackengine}
\stackMath%
\newcommand\reallywidehat[1]{%
\savestack{\tmpbox}{\stretchto{%
  \scaleto{%
    \scalerel*[\widthof{\ensuremath{#1}}]{\kern.1pt\mathchar"0362\kern.1pt}%
    {\rule{0ex}{\textheight}}
  }{\textheight}%
}{2.4ex}}%
\stackon[-6.9pt]{#1}{\tmpbox}%
}

\makeatletter
\newsavebox\myboxA
\newsavebox\myboxB
\newlength\mylenA
\newcommand*\xoverline[2][0.75]{%
    \sbox{\myboxA}{$\m@th#2$}%
    \setbox\myboxB\null
    \ht\myboxB=\ht\myboxA%
    \dp\myboxB=\dp\myboxA%
    \wd\myboxB=#1\wd\myboxA
    \sbox\myboxB{$\m@th\overline{\copy\myboxB}$}
    \setlength\mylenA{\the\wd\myboxA}
    \addtolength\mylenA{-\the\wd\myboxB}%
    \ifdim\wd\myboxB<\wd\myboxA%
       \rlap{\hskip 0.5\mylenA\usebox\myboxB}{\usebox\myboxA}%
    \else
        \hskip -0.5\mylenA\rlap{\usebox\myboxA}{\hskip 0.5\mylenA\usebox\myboxB}%
    \fi}
\makeatother    

\usepackage{fancyhdr}

\newcommand\shortitle{Permutation recovery of spikes in noisy high-dimensional tensor estimation}
\newcommand\name{gérard ben arous, cédric gerbelot, and vanessa piccolo}

\begin{document}

\title{Permutation recovery of spikes in noisy high-dimensional tensor estimation}
\author{Gérard Ben Arous\(^1\)}
\address{\(^1\)Courant Institute of Mathematical Sciences, New York University}
\email{benarous@cims.nyu.edu}
\author{Cédric Gerbelot\(^{1,2}\)}
\author{Vanessa Piccolo\(^2\)}
\address{\(^2\)Unité de Mathématiques Pures et Appliquées (UMPA), ENS Lyon}
\email{cedric.gerbelot-barrillon@ens-lyon.fr}
\email{vanessa.piccolo@ens-lyon.fr}

\subjclass[2020]{68Q87, 62F30, 60H30} 
\keywords{High-dimensional optimization, Multi-spiked tensor PCA, Gradient flow dynamics, Permutation recovery}
\date{\today}

\begin{abstract}
We study the dynamics of gradient flow in high dimensions for the multi-spiked tensor problem, where the goal is to estimate \(r\) unknown signal vectors (spikes) from noisy Gaussian tensor observations. We analyze the maximum likelihood estimator, which corresponds to optimizing a high-dimensional, nonconvex random objective. Our main results determine the sample complexity and runtime required for gradient flow to efficiently recover all spikes, up to a permutation. We show that, during recovery, correlations between the estimators and true spikes increase sequentially, in an order depending on their initial value and on the associated signal-to-noise ratios (SNRs). This ordering determines the permutation under which the spikes are recovered. This work builds on our companion paper~\cite{benarous2024langevin}, which analyzes Langevin dynamics and establishes the sample complexity and SNR conditions required for exact recovery, where the recovered permutation matches the identity. 
\end{abstract}

\maketitle	
{
\hypersetup{linkcolor=black}
\tableofcontents
}

\section{Introduction} \label{section:introduction}

Motivated by recent advances in data science, where gradient-based methods are used routinely to efficiently optimize high-dimensional, nonconvex functions, we study \emph{gradient flow dynamics} in the context of a noisy tensor estimation problem: the \emph{spiked tensor model}. The goal is to recover a hidden vector on the unit sphere from the noisy tensor observations. This problem reduces to optimizing a highly nonconvex random function arising from the maximum likelihood estimation (MLE) method. We generalize previous results for the single-spike case to the \emph{multi-spike setting}, focusing on the sample complexity and runtime required to efficiently recover the \(r\) orthogonal signal vectors from random initialization. The spiked tensor model, introduced by Richard and Montanari~\cite{MontanariRichard} for the single-spike case, has since been widely studied, particularly regarding the optimization dynamics of gradient-based methods. In particular, the algorithmic thresholds for these methods in the single-spike case were analyzed by the first author in collaboration with Gheissari and Jagannath~\cite{arous2020algorithmic, arous2021online}. Our analysis builds on results for Langevin dynamics presented in our companion paper~\cite{benarous2024langevin}, where Langevin dynamics recovers gradient flow in the zero-temperature limit. In that work, the sample complexity threshold was studied under a separation condition on the signal-to-noise ratios (SNRs). In contrast, in this paper, we show that for gradient flow, no such condition is required to fully characterize the optimization dynamics of the MLE objective function. This relaxation introduces the notion of \emph{recovery up to a permutation of the spikes}, which we define below. The core of our analysis lies in a quantitative reduction of the random high-dimensional dynamics to a deterministic low-dimensional dynamical system, where the initial condition determines the permutation of the spikes recovered by the algorithm. 

Our present work, along with~\cite{benarous2024langevin} and a third companion paper~\cite{benarous2024} on the (discrete-time) online stochastic gradient descent (SGD) algorithm, is part of an ongoing research effort to understand the remarkable efficiency of gradient-based methods in high-dimensional, nonconvex optimization problems. The emergence of preferred directions in the trajectories of high-dimensional optimization algorithms has been observed repeatedly, particularly in the context of deep neural networks--for instance, in the work of Papyan, Han, and Donoho~\cite{papyan2020}--and lies at the heart of recent theoretical advances in modern machine learning. In particular, the first author, together with Gheissari and Jagannath~\cite{benarousneurips,ben2022effective}, proposed a general framework for reducing the high-dimensional trajectories of online stochastic gradient descent (SGD) methods to selected projections, called \emph{summary statistics}. In this context, the present work shows that when multiple summary statistics are identified, the resulting low-dimensional dynamical system can exhibit complex and unexpected behavior. At a technical level, however, controlling the noisy part of the dynamics for gradient flow is more challenging than for online SGD, where the noise can be handled uniformly using martingale inequalities, see e.g.~\cite{arous2021online, tan2023online}. Here, the noisy part of the dynamics does not verify convenient martingale properties, necessitating tools to control the resulting correlations across the entire trajectory. In particular, our proof method builds on advances in the analysis of dynamics in spin glass models, developed by the first author jointly with Gheissari and Jagannath~\cite{ben2020bounding,arous2020algorithmic}. These results allow us to overcome the limitations of standard techniques from statistical physics, see e.g.~\cite{sompolinsky1981dynamic, crisanti1993spherical, cugliandolo1993analytical}, and probability theory, see e.g.~\cite{grunwald1996sanov,arous2001aging,ben2006cugliandolo}, to analyze gradient flow trajectories on random landscapes. Further details on related works, relevant to both probability theory and machine learning theory, can be found in the literature sections of our companion papers~\cite{benarous2024langevin, benarous2024}.

\subsection{Model}

The multi-spiked tensor model is defined as follows. Let \(p \ge 3\) and \(r \ge 1\) be fixed integers. Suppose that we are given \(M\) i.i.d.\ observations \(\boldsymbol{Y}^\ell\) of a rank-\(r\) \(p\)-tensor on \(\R^N\) of the form
\begin{equation}\label{eq: spiked tensor model}
\boldsymbol{Y}^\ell = \sum_{i=1}^r \lambda_i \sqrt{N} \left (\frac{\boldsymbol{v}_i}{\sqrt{N}}\right)^{\otimes p} + \boldsymbol{W}^\ell ,
\end{equation}
where \((\boldsymbol{W}^\ell)_\ell\) are i.i.d.\ samples of a \(p\)-tensor with i.i.d.\ entries \(W^\ell_{i_1, \ldots, i_p} \sim \mathcal{N}(0,1)\), \(\lambda_1 \geq \cdots \geq \lambda_r \geq 0\) are the signal-to-noise ratios (SNRs), and \(\boldsymbol{v}_1, \ldots,\boldsymbol{v}_r\) are \emph{unknown, orthogonal vectors} lying on the \(N\)-dimensional sphere of radius \(\sqrt{N}\), denoted by \(\mathbb{S}^{N-1}(\sqrt{N})\). The orthogonality assumption simplifies the analysis slightly. The scaling in~\eqref{eq: spiked tensor model} is chosen such that the signal and the typical fluctuations of the noise are of the same order of magnitude $\sqrt{N}$.

The goal is to estimate the unknown signal vectors \(\boldsymbol{v}_1, \ldots, \boldsymbol{v}_r\) via empirical risk minimization:
\begin{equation} \label{eq: ERM}
[\hat{\boldsymbol{v}}_1, \ldots, \hat{\boldsymbol{v}}_r] = \argmin_{\boldsymbol{X} \colon \boldsymbol{X}^\top \boldsymbol{X} = N \boldsymbol{I}_r} \hat{\mathcal{R}}_{N,r}(\boldsymbol{X}),
\end{equation}
where the empirical risk \(\hat{\mathcal{R}}_{N,r}\) is defined as
\[
\hat{\mathcal{R}}_{N,r}(\boldsymbol{X}) = \frac{1}{M} \sum_{\ell=1}^M \mathcal{L}_{N,r}(\boldsymbol{X};\boldsymbol{Y}^\ell).
\]
The constraint set \(\{\boldsymbol{X} \in \R^{N \times r} \colon \boldsymbol{X}^\top \boldsymbol{X} = N \boldsymbol{I}_r \}\) consists of \(N \times r\) matrices with orthogonal columns of norm \(\sqrt{N}\), referred to as the \emph{normalized Stiefel manifold}: 
\begin{equation} \label{def: normalized stiefel manifold}
\mathcal{S}_{N,r} = \left \{ \boldsymbol{X} = [\boldsymbol{x}_1, \ldots, \boldsymbol{x}_r] \in \R^{N \times r} \colon \langle \boldsymbol{x}_i, \boldsymbol{x}_j \rangle = N \delta_{ij}\right \}.
\end{equation}
We solve the optimization problem~\eqref{eq: ERM} using maximum likelihood estimation (MLE), where the loss function \(\mathcal{L}_{N,r} \colon \mathcal{S}_{N,r}\times (\R^N)^{\otimes p} \to \R\) is given by
\[
\mathcal{L}_{N,r}(\boldsymbol{X}; \boldsymbol{Y}^\ell) = - \sum_{i=1}^r \lambda_i \sqrt{N} \left \langle\boldsymbol{Y}^\ell, \left(\frac{\boldsymbol{x}_i}{\sqrt{N}}\right)^{\otimes p} \right \rangle.
\]
Substituting the tensor model~\eqref{eq: spiked tensor model} into this expression, the loss function results in
\[
\mathcal{L}_{N,r}(\boldsymbol{X}; \boldsymbol{Y}^\ell)  = - \frac{1}{N^{\frac{p-1}{2}}} \sum_{i=1}^r \lambda_i \langle \boldsymbol{W}^\ell, \boldsymbol{x}_i^{\otimes p} \rangle -  \sum_{1 \le i,j \le r}  N \lambda_i \lambda_j \left ( \frac{\langle \boldsymbol{v}_i, \boldsymbol{x}_j \rangle}{N}\right )^p.
\]
Given the Gaussian assumption on $\boldsymbol{W}^{\ell}$, optimizing the empirical risk \(\hat{\mathcal{R}}_{N,r}\) is equivalent, in distribution, to minimizing 
\begin{equation} \label{eq: empirical risk}
\mathcal{R} (\boldsymbol{X}) = \frac{1}{\sqrt{M}} H_0(\boldsymbol{X}) -  \sum_{1 \leq i,j \leq r} N \lambda_i \lambda_j \left( m_{ij}^{(N)}(\boldsymbol{X}) \right)^p,
\end{equation}
where \(m_{ij}^{(N)}(\boldsymbol{X}) \coloneqq N^{-1} \langle \boldsymbol{v}_i, \boldsymbol{x}_j \rangle\) denotes the \emph{correlation} of \(\boldsymbol{v}_i\) with  \(\boldsymbol{x}_j\). Here, the Hamiltonian \(H_0 \colon \mathcal{S}_{N,r} \to \R\) is given by
\begin{equation} \label{eq: Hamiltonian H0}
H_0(\boldsymbol{X}) = \frac{1}{N^{\frac{p-1}{2}}} \sum_{i=1}^r \lambda_i  \langle \boldsymbol{W}, \boldsymbol{x}_i^{\otimes p} \rangle.
\end{equation}
We note that \(H_0\) is a centered Gaussian process with covariance of the form 
\[
\E \left[H_0 (\boldsymbol{X})H_0(\boldsymbol{Y})\right] = N \sum_{1\leq i,j\leq r} \lambda_i \lambda_j \left(\frac{\langle \boldsymbol{x}_i, \boldsymbol{y}_{j} \rangle}{N}\right)^p.
\]

\subsection{Gradient flow dynamics}  \label{subsection: algorithm}

The gradient flow can be interpreted as the limiting dynamics of gradient descent with infinitesimal step size. Given an initial condition \(\boldsymbol{X}_0 \in \mathcal{S}_{N,r}\), which is possibly random, we let \(\boldsymbol{X}_t \in \mathcal{S}_{N,r}\) solve the following ordinary differential equation:
\begin{equation} \label{eq: GF}
\frac{\textnormal{d} \boldsymbol{X}_t}{\textnormal{d} t} = - \nabla \mathcal{R}(\boldsymbol{X}_t),
\end{equation}
where \(\nabla\) denotes the Riemannian gradient on the manifold \(\mathcal{S}_{N,r}\). Specifically, for any function \(f \colon \mathcal{S}_{N,r} \to \R\), the Riemannian gradient is given by
\begin{equation}  \label{eq: riemannian gradient on normalized stiefel}
\nabla f(\boldsymbol{X}) = \hat{\nabla} f(\boldsymbol{X}) - \frac{1}{2N} \boldsymbol{X} \left (\boldsymbol{X}^\top \hat{\nabla} f(\boldsymbol{X}) + \hat{\nabla} f(\boldsymbol{X})^\top \boldsymbol{X} \right ),
\end{equation}
where \(\hat{\nabla}\) denotes the Euclidean gradient. The Lie derivative operator associated with the deterministic gradient flow is given by
\begin{equation} \label{eq: generator GF}
L = - \langle \nabla  \mathcal{R}, \hat{\nabla} \cdot \rangle,
\end{equation}
where the inner product \(\langle \boldsymbol{A}, \boldsymbol{B} \rangle\) denotes the trace inner product between matrices, i.e., \(\langle \boldsymbol{A}, \boldsymbol{B} \rangle = \Tr(\boldsymbol{A}^\top \boldsymbol{B})\). This operator \(L\) describes the infinitesimal evolution of smooth functions along the gradient flow vector field \(- \nabla \mathcal{R}\), and can be interpreted as the Lie derivative along \(- \nabla \mathcal{R}\). Similarly, we define the operator \(L_0\) as the infinitesimal evolution induced by the gradient flow associated with the noise Hamiltonian \(H_0\):
\begin{equation} \label{eq: generator GF noise}
L_0 = - \frac{1}{\sqrt{M}} \langle \nabla H_0, \hat{\nabla} \cdot \rangle.
\end{equation}

\subsection{Main results} \label{subsection: main results}

Our goal is to determine the \emph{sample complexity} (i.e., the number of observations required) and the \emph{computational runtime} (i.e., the time horizon of the gradient flow) needed to recover the unknown orthogonal vectors \(\boldsymbol{v}_1, \ldots, \boldsymbol{v}_r\) via the gradient flow~\eqref{eq: GF}. From this point onward, we consider the process \((\boldsymbol{X}_t)_{t \ge 0}\) defined by~\eqref{eq: GF}, initialized randomly with \(\boldsymbol{X}_0\) drawn from the uniform distribution \(\mu_{N \times r}\) on \(\mathcal{S}_{N,r}\). The measure \(\mu_{N \times r}\) is the unique probability distribution on \(\mathcal{S}_{N,r}\) that is invariant under both the left and right orthogonal transformations. We consider the probability space \((\Omega,\mathcal{F},\mathbb{P})\) on which the \(p\)-tensors \((\boldsymbol{W}^\ell)_\ell\) are defined. We denote by \(\mathbb{P}_{\boldsymbol{X}_0}\) the law of the process \((\boldsymbol{X}_t)_{t \ge 0}\) initiated at \(\boldsymbol{X}_0 \sim \mu_{N \times r}\). More precisely, following the convention of~\cite[Chapter 6]{le2013mouvement}, we have
\[
\mathbb{P}_{\boldsymbol{X}_0}(A) = \int_{\mathcal{S}_{N,r}} \mathbb{P}_{\boldsymbol{X}}(A) \textnormal{d} \mu_{N \times r}(\boldsymbol{X}),
\]
for any measurable set $A$ in the \(\sigma\)-algebra generated by the coordinate mappings from $\R^+$ to $\mathcal{S}_{N,r}$. We also define \(\mathbb{P}_{\boldsymbol{X}_0^+}\) as the law of the process initiated at \(\boldsymbol{X}_0 \sim \mu_{N \times r}\), subject to the condition \(m_{ij}(\boldsymbol{X}_0)>0\) for all \(1 \le i,j \le r\). 

\paragraph{\emph{Notations}.} 
For a positive integer \(n \in \N\), we denote \([n] \coloneqq \{1, \ldots, n\}\). For two sequences \(x_N\) and \(y_N\), we write \(x_N \ll y_N\) to indicate that \(x_N / y_N \to 0\) as \(N \to \infty\).
\medskip

We are now ready to present our main results. Throughout this section, we assume that the SNRs \(\lambda_1 \ge \cdots \ge \lambda_r \ge 0\) are of order \(1\). While the statements below are presented in asymptotic form, we provide stronger non-asymptotic formulations—including explicit constants and convergence rates---in Section~\ref{section: main results}. 
 
\begin{thm}[Recovery up to a permutation] \label{thm1}
If the number of samples satisfies \(M \gg N^{p-1}\), then there exists a permutation \(\sigma^\ast \in S_r\) and a time \(T_0 \gg  N^{\frac{p-2}{2}}\) such that for every \(\varepsilon >0\) and every \(T \ge T_0\),
\[
\lim_{N \to \infty} \mathbb{P}_{\boldsymbol{X}_0^+} \left( \inf_{t \in [T_0,T]} \, m_{\sigma^\ast (i) i}^{(N)}(\boldsymbol{X}_t)  \geq 1 - \varepsilon \right) = 1.
\]
\end{thm}

Theorem~\ref{thm1} establishes that, under a positive initialization of the correlations, gradient flow successfully recovers all signal directions up to a permutation, provided the number of samples \(M\) scale as \(N^{p-1}\). In our companion work~\cite{benarous2024langevin}, we show that Langevin dynamics achieves exact recovery (i.e., with \(\sigma^\ast\) being the identity permutation), provided the SNRs are separated by large constants independent of \(N\). Since gradient flow corresponds to the zero-temperature limit of Langevin dynamics, Theorems 1.4 and 1.5 of~\cite{benarous2024langevin} extend naturally to the gradient flow setting (see Remark~\ref{rmk: langevin SNRS} for more details). In the remainder of this section, we refine Theorem~\ref{thm1} in two key directions: we remove the assumption that the initial correlations are strictly positive and characterize the permutation \(\sigma^\ast\) governing the recovered spikes. This permutation can be explicitly determined via the following procedure.

\begin{defn}[Greedy maximum selection] \label{def: greedy operation}
Let \(\boldsymbol{A} \in \R^{r \times r}\) be a matrix whose nonzero entries are all distinct. We define a sequence of index pairs \((i_k^\ast, j_k^\ast) \in [r]^2\) recursively as follows:
\begin{itemize}
\item[1.] Set \(\boldsymbol{A}^{(0)} \coloneqq \boldsymbol{A}\).
\item [2.] For \(k = 1, 2, \ldots\), define
\[
(i_k^\ast, j_k^\ast) \coloneqq \argmax_{1 \le i,j \le r-(k-1)} | \boldsymbol{A}^{(k-1)} |_{ij},
\]
where \(\boldsymbol{A}^{(k-1)} \in \R^{(r-(k-1)) \times (r-(k-1))}\) is obtained from \(\boldsymbol{A}\) by removing the rows \(i_1^\ast, \ldots, i_{k-1}^\ast\) and the columns \(j_1^\ast, \ldots, j_{k-1}^\ast\), and \(|\boldsymbol{A}^{(k-1)}| \) denotes the absolute value of the entries in \(\boldsymbol{A}^{(k-1)}\). 
\item[3.] If at some step \(r_\textnormal{c} \in [r]\) we have
\[
\max_{ij} |\boldsymbol{A}^{(r_\textnormal{c})}|_{ij} = 0 ,
\]
the procedure terminates. 
\end{itemize}
The resulting sequence \((i_1^\ast, j_1^\ast), \ldots, (i_{r_\textnormal{c}}^\ast, j_{r_\textnormal{c}}^\ast)\) is called the \emph{greedy maximum selection} of \(\boldsymbol{A}\).
\end{defn}

The permutation \(\sigma^\ast\) in Theorem~\ref{thm1} is determined by the greedy maximum selection applied to the following initialization matrix:
\begin{equation} \label{eq: initialization matrix}
\boldsymbol{I}_0 = \left (\lambda_i \lambda_j \left( m_{ij}^{(N)}(\boldsymbol{X}_0)\right)^{p-2} \mathbf{1}_{\left(m_{ij}^{(N)}(\boldsymbol{X}_0)\right)^{p-2} \geq 0} \right)_{1 \le i, j \le r}. 
\end{equation}
From this procedure we obtain a sequence of index pairs \((i_k^\ast, j_k^\ast)\), which specifies the correspondence between recovered and true spikes, i.e., \((\sigma^\ast(i), i) = (i_k^\ast, j_k^\ast)\). The matrix \(\boldsymbol{I}_0 \in \R^{r \times r}\) is random. Although in principle its nonzero entries may coincide due to randomness, Lemma~\ref{lem: separation initial data GF} ensures that they are distinct with probability \(1-o(1)\), making the greedy maximum selection of \(\boldsymbol{I}_0\) well-defined with high probability. We now present a more precise formulation of Theorem~\ref{thm1}.

\begin{thm} \label{thm: strong recovery GF p>2 asymptotic}
\leavevmode
\begin{itemize}
\item[(a)] If \(M = N^\alpha\) for \(\alpha > p-2\), then there exists a time \(T_0 \gg N^{\frac{p-2}{2}}\) such that for every \(\varepsilon >0\) and every \(T \ge T_0\),
\[
\lim_{N \to \infty} \mathbb{P}_{\boldsymbol{X}_0} \left( \inf_{t \in [T_0,T]} \, \lvert m_{i^\ast_1 j^\ast_1}^{(N)}(\boldsymbol{X}_t) \rvert \geq 1 - \varepsilon \right) = 1,
\]
where \((i_1^\ast, j_1^\ast)\) denotes the first index pair obtained via the greedy maximum selection applied to the initialization matrix \(\boldsymbol{I}_0\).
\item[(b)] If \(M \gg N^{p-1}\), then there exists a time \(T_0 \gg  N^{\frac{p-2}{2}}\) such that for every \(\varepsilon >0\), every \(T \ge T_0\), and every \(k\in [r_\textnormal{c}]\),
\[
\lim_{N \to \infty} \mathbb{P}_{\boldsymbol{X}_0} \left( \inf_{t \in [T_0,T]}  \, \lvert m_{i^\ast_k j^\ast_k}^{(N)}(\boldsymbol{X}_t) \rvert \geq 1 - \varepsilon \right) = 1,
\]
where \((i_k^\ast, j_k^\ast)\) denotes the \(k\)th index pair obtained via the greedy maximum selection applied to the initialization matrix \(\boldsymbol{I}_0\).
\end{itemize}
\end{thm}

\begin{rmk}
The index pairs \((i_k^\ast, j_k^\ast)\) depend on the random initialization \(\boldsymbol{X}_0\) through the greedy maximum selection applied to the matrix \(\boldsymbol{I}_0\). Consequently, they are random variables measurable with respect to \(\boldsymbol{X}_0\). The probability \(\mathbb{P}_{\boldsymbol{X}_0}\) naturally accounts for this randomness.
\end{rmk}

From statement (b) of Theorem~\ref{thm: strong recovery GF p>2 asymptotic} and the definition of the matrix \(\boldsymbol{I}_0\), we observe that if all correlations are positive at initialization or if $p$ is even, then \(r_\text{c} = r\), ensuring that all spikes are recovered up to a permutation. However, if we do not impose positivity constraints on the initialization or if \(p\) is odd, Theorem~\ref{thm: strong recovery GF p>2 asymptotic} guarantees recovery of a subset of the spikes, with cardinality $r_\text{c} \le r$. A key subtlety compared to Theorem~\ref{thm1} is that recovering the first spike requires a lower sample complexity than recovering all spikes. Specifically, item (a) states that recovery of the first spike requires \(M\) to scale as \(N^\alpha\) for \(\alpha > p-2\), matching the threshold obtained in the single-spike setting~\cite{arous2020algorithmic}, while recovering a subset of the spikes requires an order \(N^{p-1}\) samples. This difference in sample complexity arises from our proof method. During the initial phase of recovery, the noise term \(L_0 m_{ij}^{(N)}(\boldsymbol{X})\) is absorbed by the initial correlation \(m_{ij}(\boldsymbol{X}_0)\). This absorption reduces the noise scaling from order \(1\) to \(N^{-\frac{1}{2}}\),  thereby lowering the sample complexity required for recovering the first spike from \(N^{p-1}\) to \(N^{p-2}\). However, once the first spike has been recovered, this beneficial scaling no longer applies, and the noise is bounded by a constant of order \(1\). As a result, recovering the subsequent spikes requires \(N^{p-1}\) samples. In our companion paper~\cite{benarous2024}, we show that using the online SGD algorithm, the sample complexity threshold for permutation recovery matches the sharp threshold \(N^{p-2}\) obtained for \(r=1\)~\cite{arous2020algorithmic, arous2021online}. The difference in sample complexity between gradient flow and online SGD arises from the sample usage: online SGD uses independent samples at each iteration, allowing the sharp \(N^{p-2}\) scaling even for subsequent spikes.

The phenomenology underlying Theorem~\ref{thm: strong recovery GF p>2 asymptotic} is richer than the one presented by Theorem~\ref{thm: strong recovery GF p>2 asymptotic} itself. Indeed, based on the values of the entries of \(\boldsymbol{I}_0\), the correlations \(\{m_{i^\ast_k j^\ast_k}^{(N)}\}_{k=1}^{r_\text{c}}\) reach a macroscopic threshold one by one, sequentially eliminating the correlations that share a row or column index to allow the next correlation to grow to macroscopic. This phenomenon is referred to as \emph{sequential elimination} with ordering determined by the greedy maximum selection and is illustrated by Figures~\ref{fig1} and~\ref{fig2}.

\begin{defn}[Sequential elimination] \label{def: sequential elimination}
Let \(S = \{ (i_1, j_1), \ldots, (i_m,j_m)\}\) be a set with distinct \(i_1, \ldots, i_m \in [r]\) and distinct \(j_1, \ldots, j_m \in [r]\), where \(m \le r\). We say that the correlations \(\{m_{ij}(t)\}_{1 \leq i,j \leq m}\) follow a \emph{sequential elimination with ordering \(S\)} if for every \(\varepsilon, \varepsilon' > 0\), there exist \(m\) stopping times \(T_1 \leq \cdots \leq T_m\) such that for every \(k \in [m]\) and every \(T \ge T_k\),
\[
| m_{i_k j_k}(\boldsymbol{X}_T)| \geq 1 - \varepsilon  \quad \textnormal{and} \quad |m_{i_k j}(\boldsymbol{X}_T)| \leq \varepsilon', |m_{i j_k}(\boldsymbol{X}_T)| \leq \varepsilon' \enspace \textnormal{for} \: i\neq i_k, j \neq j_k.
\]
\end{defn}

Based on Definition~\ref{def: sequential elimination}, we have the following result, which serves as a foundation for Theorem~\ref{thm: strong recovery GF p>2 asymptotic}.

\begin{thm} \label{thm: strong recovery online p>2 asymptotic stronger}
If \(M \gg N^{p-1}\), then the correlations \(\{m_{ij}^{(N)}\}_{1 \le i, j \le r}\) follow a sequential elimination with ordering \(\{(i_k^\ast, j_k^\ast)\}_{k=1}^{r_\text{c}}\) and stopping times of order \(N^{\frac{p-2}{2}}\), with \(\mathbb{P}\)-probability \(1\) in the large-\(N\) limit.
\end{thm}

\begin{figure}
\centering
\includegraphics[scale=0.3]{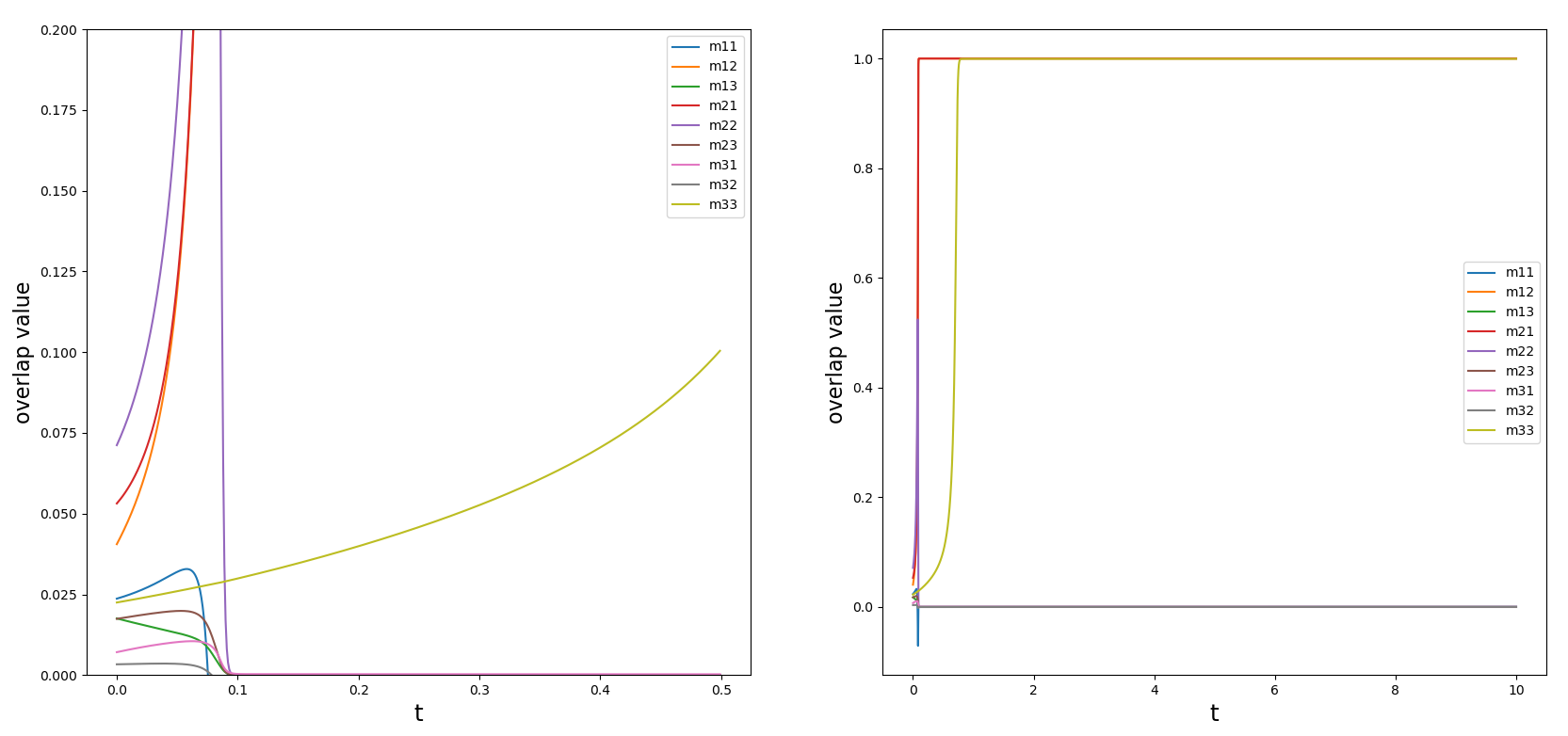}
\caption{Evolution of the correlations \(m_{ij}\) under gradient flow for the case where \(p=3, r=3\), with SNRs \(\lambda_1 = 3, \lambda_2 =2, \lambda_3 = 1\). The simulation is performed with \(M = 1000\) samples and a dimension of \(N=1000\). The simulation shows recovery of a permutation of the spikes.}
\label{fig1}
\end{figure}

\begin{figure}
\centering
\includegraphics[scale=0.3]{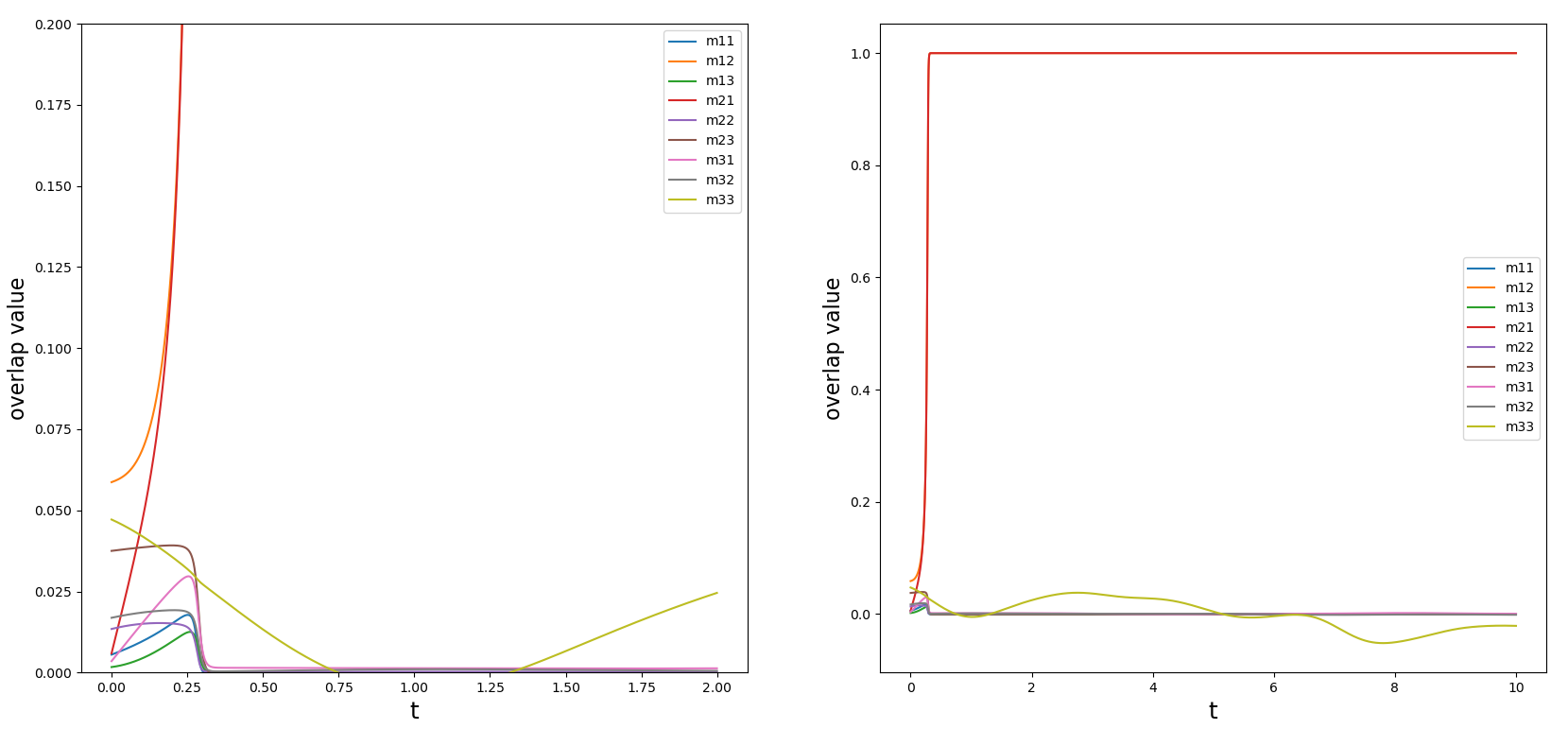}
\caption{Evolution of the correlations \(m_{ij}\) under gradient flow for the case where \(p=3, r=3\), with SNRs \(\lambda_1 = 2, \lambda_2 =1, \lambda_3 = 0.1\). The simulation is performed with \(M = 1000\) samples and a dimension of \(N=1000\). The first two directions are successfully recovered, while the third direction, associated with the lowest SNR, is lost in the noise and remains unrecovered.}
\label{fig2}
\end{figure}

\begin{rmk} \label{rmk: sign initialization}
It is important to note that in the above results, the behavior of gradient flow depends on the parity of integer \(p\). When \(p\) is odd, then each estimator \(\boldsymbol{x}_{j_k^\ast}\) recovers the spike \(\boldsymbol{v}_{i_k^\ast}\) with \(\mathbb{P}\)-probability \(1-o(1)\), since the correlations that are negative at initialization get trapped at the equator. Conversely, when \(p\) is even, we have that each estimator \(\boldsymbol{x}_{j_k^\ast}\) recovers \(\textnormal{sgn}(m_{i_k^\ast j_k^\ast}(\boldsymbol{X}_0)) \boldsymbol{v}_{i_k^\ast j_k^\ast}\) with probability \(1-o(1)\). This means that if the correlation at initialization is positive, then \(\boldsymbol{x}_{j_k^\ast}\) recovers \(\boldsymbol{v}_{i_k^\ast}\); otherwise, \(\boldsymbol{x}_{j_k^\ast}\) recovers \(-\boldsymbol{v}_{i_k^\ast}\).
\end{rmk}

\begin{rmk} 
In our companion paper~\cite{benarous2024langevin}, we also analyze Langevin dynamics in the matrix case ($p=2$), distinguishing between two scenarios: when the SNRs are separated by order-$1$ constants and when the SNRs are all equal. In the former case, we establish exact recovery of all spikes, whereas in the latter, we recover the subspace spanned by all spikes. For details, we refer readers to~\cite[Theorems 1.10, 1.11, and 1.12]{benarous2024langevin}. Since gradient flow is a special case of Langevin dynamics, these results naturally extend to the gradient flow setting and are therefore not presented in this article.
\end{rmk}
\subsection{Related works} \label{subsection: related work}

The tensor PCA problem~\eqref{eq: spiked tensor model}, originally introduced for matrices by Johnstone~\cite{Johnstone} and later extended to tensors by Richard and Montanari~\cite{MontanariRichard}, provides a fundamental framework for analyzing optimization in high-dimensional, nonconvex landscapes using gradient-based methods. The case \(r=1\) has been extensively studied, with particular focus on various threshold phenomena. In particular, the \emph{information-theoretic threshold} for signal detection has been the subject of significant research, with notable contributions including~\cite{lesieur2017, perry2018optimality, bandeira2020, chen2019, jagannath2020, dominik2024}. The \emph{statistical threshold}, which validates the maximum likelihood estimator (MLE) as a reliable statistical method, has been analyzed in~\cite{benarouscomplexity2019, ros2019, jagannath2020}. From a \emph{computational} perspective, spectral methods and sum-of-squares algorithms have been shown to achieve the sharp sample complexity threshold \(N^\frac{p-2}{2}\)~\cite{Hopkins15, Hopkins16, Bandeira2017, Weinkikuchi, BenArousUnfolding}. In contrast, gradient-based methods~\cite{arous2020algorithmic, arous2021online} and tensor power iteration~\cite{HuangPCA, Wu24} reach the computational threshold of \(N^{p-2}\). In particular, the latter work~\cite{Wu24} provides the state-of-the-art threshold, showing that the required number of samples scales as \(N^{p-2} \log(N)^{-C}\), where \(C\) is a constant depending on the tensor order \(p\). For the multi-rank tensor PCA model, both detection and recovery thresholds have been studied. On the information-theoretic side, it has been shown that for \(p=2\)~\cite{lelarge2019fundamental} and for \(p \ge 3\)~\cite{chen2021}, there is an order-\(1\) critical threshold for the SNRs, above which it is possible to detect the unseen low-rank signal tensor \(\sqrt{N} \sum_{i=1}^r \lambda_i \boldsymbol{v}_i^{\otimes p}\). On the algorithmic side,~\cite{HuangPCA} analyzed the power iteration algorithm and identified the local threshold for efficiently recovering the finite-rank signal components. In our companion paper~\cite{benarous2024}, we analyze the discretization of gradient flow in the form of online SGD and show that it achieves the same algorithmic threshold of \(N^{p-2}\) as in the single-spike case~\cite{arous2021online}. 

The multi-spiked tensor PCA problem serves as both a paradigmatic example of high-dimensional, nonconvex optimization and a key illustration of \emph{statistical-to-computational gaps}. While various techniques from the statistical physics of spin glasses and statistical learning theroy have been applied to study gradient flows in disordered systems, these methods prove insufficient for the current problem. In particular, they fail to capture sharp sample complexity thresholds and do not precisely characterize the minimizers reached by gradient flow. Further discussion of these limitations can be found in the related works section of our companion paper on Langevin dynamics~\cite{benarous2024langevin}. Additionally, the relevance of this problem to machine learning theory is explored in Subsection 1.3 of our companion paper on online SGD~\cite{benarous2024}.
\subsection{Outline of proofs} \label{subsection: outline proofs}

We now outline the proof of our main results. A similar explanation is presented in our companion paper~\cite{benarous2024langevin}, which focuses on Langevin dynamics, a broader framework within which gradient flow serves as a special case. To prove our main results, we analyze the evolution of the correlations \(\{m_{ij}^{(N)}\}_{i=1}^r\) under gradient flow~\eqref{eq: GF}. We assume an initial random start with a completely uninformative prior, specifically the invariant distribution on \(\mathcal{S}_{N,r}\). As a consequence, all correlations \(m_{ij}^{(N)}\) have the typical scale of order \(N^{-\frac{1}{2}}\) at initialization. For simplicity, we assume that all correlations are positive at initialization. Additionally, to streamline notation, we write \(m_{ij}\) instead of \(m_{ij}^{(N)}\) in the following discussion.

According to~\eqref{eq: GF}, the evolution equation for the correlations \(m_{ij}(\boldsymbol{X}_t)\) under gradient flow dynamics is governed by
\[
\frac{\text{d}m_{ij} (\boldsymbol{X}_t)}{\text{d} t} = - \frac{1}{N} \left \langle \boldsymbol{v}_i, \left (\nabla \mathcal{R}(\boldsymbol{X}_t )\right )_j \right \rangle,
\] 
where \( \left (\nabla \mathcal{R}(\boldsymbol{X}_t )\right )_j\) denotes the \(j\)th column of the Riemannian gradient \(\nabla \mathcal{R}\), and \(\mathcal{R}\) is the the empirical risk defined in~\eqref{eq: empirical risk}. Using the definition of the generator \(L\) from~\eqref{eq: generator GF}, the gradient flow dynamics can be rewritten as
\[
\frac{\text{d}m_{ij} (\boldsymbol{X}_t)}{\text{d} t}  = - \left \langle \nabla \mathcal{R}, \hat{\nabla} m_{ij} \right \rangle = L m_{ij}.
\]
A direct computation of the Riemannian gradient \(\nabla \mathcal{R}\) yields the following decomposition:
\[
L m_{ij} = L_0 m_{ij} + p \lambda_i \lambda_j m_{ij}^{p-1}  -  \frac{p}{2} \sum_{1 \leq k, \ell \leq r} \lambda_k m_{kj} m_{k \ell} m_{i \ell}\left(\lambda_j m_{kj}^{p-2} + \lambda_\ell m_{k \ell}^{p-2} \right) ,
\]
where \(L_0\) is the noise generator defined in~\eqref{eq: generator GF noise}. The second term, \(p \lambda_i \lambda_j m_{ij}^{p-1}\), corresponds to the primary drift and dominates the dynamics, particularly near initialization. The third term represents a correction arising from the orthogonality constraint on the estimator \(\boldsymbol{X}\) being on the normalized Stiefel manifold, and becomes increasingly relevant as the dynamics evolve and the correlations escape their initial scale. 

The main challenge lies in balancing the signal and noise contributions to the dynamics. At early times---such as near initialization---the population drift predominates over the correction term, allowing the approximation
\[
Lm_{ij} \approx L_0 m_{ij} + p \lambda_i \lambda_j m_{ij}^{p-1}.
\]
For the correlations \(m_{ij}\) to grow, the drift term \( p \lambda_i \lambda_j m_{ij}^{p-1}\) must exceed the noise term \(L_0 m_{ij}\). Since \(m_{ij}\) typically scales as \(N^{-\frac{1}{2}}\) at initialization, it follows that \(m_{ij}^{p-1}\) is of order \(N^{\frac{p-1}{2}}\). Meanwhile, the noise term \(L_0 m_{ij}\) is of order \(N^{-\frac{1}{2}}\), implying that a sample complexity \(M = \Theta (N^{p-2})\) suffices for the drift to dominate. Under this sample complexity, the dynamics in this first phase is well approximated by the simple ordinary differential equation (ODE):
\begin{equation} \label{eq: simple ODE}
\dot{m}_{ij} \approx p \lambda_i \lambda_j m^{p-1}_{ij}.
\end{equation}
To ensure sustained signal growth, it is crucial that the drift term continues to outweigh the noise \(L_0 m_{ij}\) over a sufficiently long time horizon. This allows \(m_{ij}\) to \emph{escape mediocrity}, that is, to reach a macroscopic threshold. Bounding flows~\cite{ben2020bounding, arous2020algorithmic} address this by providing time-dependent upper bounds on the noise term, using Sobolev-type norm estimates of \(H_0(\boldsymbol{X})\) to control the evolution of correlations throughout this early phase.\\

We now focus on the population dynamics. The solution to~\eqref{eq: simple ODE} shows that, near initialization, the correlations \(m_{ij}\) are approximately given by
\begin{equation} \label{eq: simple sol p>3}
m_{ij}(t) \approx m_{ij}(0)\left(1-\lambda_i \lambda_j p(p-2) m_{ij}(0)^{p-2} t\right)^{-\frac{1}{p-2}},
\end{equation}
where \(m_{ij}(0) = \frac{\gamma_{ij}}{\sqrt{N}}\) for some constants \(\gamma_{ij}\) of order \(1\). From this expression, we see that the time it takes for \(m_{ij}\) to reach a macroscopic threshold \(\varepsilon > 0\) is approximately
\[
T_\varepsilon^{(ij)} \approx \frac{1- \left (\frac{\gamma_{ij}}{\varepsilon \sqrt{N}} \right)^{p-2}}{\lambda_i \lambda_j p (p-2) \gamma_{ij}^{p-2}} N^{\frac{p-2}{2}}.
\]
Consequently, the first correlation to become macroscopic is the one associated with the largest value of \(\lambda_i \lambda_j \gamma_{ij}^{p-2}\). Note that \((\lambda_i \lambda_j \gamma_{ij}^{p-2})_{1 \le i,j \le r}\) is precisely the initialization matrix \(\boldsymbol{I}_0\) introduced in~\eqref{eq: initialization matrix}, as here we have assumed that all initial correlations are positive. 

To simplify the discussion, we now assume \(r=2\). Without loss of generality, suppose that \(m_{11}\) is the first correlation to reach the macroscopic threshold \(\varepsilon\). Once \(m_{11}\) crosses a critical value, the remaining correlations remain close to their initialization scale. More precisely, as soon as \(m_{11}\) exceeds the microscopic threshold \(N^{-{\frac{p-2}{2p}}}\), the correction term in the population generator,
\[
\sum_{1 \le k, \ell \le 2} \lambda_k m_{kj} m_{k \ell} m_{i \ell} (\lambda_j m_{kj}^{p-2} + \lambda_\ell m_{k \ell}^{p-2}),
\]
becomes dominant in the evolution equations for \(m_{12}\) and \(m_{21}\), driving them to decrease. Similarly, this correction term may also dominate the dynamics of \(m_{22}\) once \(m_{11}\) exceeds a finer microscopic threshold of order \(N^{-{\frac{p-3}{2(p-1)}}}\), potentially inducing a decrease in \(m_{22}\) as well. A careful analysis shows that any such decrease in \(m_{22}\) is at most of order \(\frac{\log(N)}{N}\), so \(m_{22}\) remains stable at its initialization scale \(\mathcal{O}(N^{-1/2}\) during the growth of \(m_{11}\). Once \(m_{12}\) and \(m_{21}\) become sufficiently small, \(m_{22}\) evolves according to the same population ODE~\eqref{eq: simple sol p>3}, enabling the recovery of the second signal direction. 

This stepwise progression is referred to as the \emph{sequential elimination phenomenon}: when a correlation (e.g., \(m_{11}\)) crosses a critical threshold, the correlations in the same row or column (e.g., \(m_{12}, m_{21}\)) are suppressed, which in turn allows subsequent correlations (e.g., \(m_{22}\)) to grow. This behavior is illustrated in Figures~\ref{fig1} and~\ref{fig2}. Finally, if the SNRs are sufficiently separated, the algorithm achieves exact recovery of the unknown signal directions with high probability, as shown in~\cite{benarous2024langevin}. Otherwise, the result is a permutation of the signal components, determined by a greedy maximum selection (see Definition~\ref{def: greedy operation}) on the initialization matrix \(\boldsymbol{I}_0\).

\subsection{Overview}

An overview of the paper is as follows. 
Section~\ref{section: main results} presents the nonasymptotic formulations of the main results introduced in Subsection 1.3, stated under general initialization conditions. Section~\ref{preliminary} provides the necessary preliminary results for the proofs. These results are drawn from our companion paper~\cite[Section 4]{benarous2024langevin}, and their proofs are therefore deferred to that reference. Section~\ref{proofs} contains the proofs of our main results. Finally, Appendix~\ref{appendix: invariant measure} concludes the paper with concentration results for the uniform measure on the normalized Stiefel manifold \(\mathcal{S}_{N,r}\).\\

\textbf{Acknowledgements.}\ G.B.\ and C.G.\ acknowledge the support of the NSF grant DMS-2134216. V.P.\ acknowledges the support of the ERC Advanced Grant LDRaM No.\ 884584.

\section{Main results} \label{section: main results}

This section presents the nonasymptotic versions of our main results stated in Subsection~\ref{subsection: main results}. These nonasymptotic versions are stronger, as they explicitly provide all constants and convergence rates. Moreover, the asymptotic results from Subsection~\ref{subsection: main results} follow directly as corollaries of these nonasymptotic statements.

According to the definition of gradient flow dynamics given in Subsection~\ref{subsection: algorithm}, we consider \(\boldsymbol{X}_t \in \mathcal{S}_{N,r}\) as the solution to the ordinary differential equation
\begin{equation} \label{eq: GF 2}
\frac{\text{d} \boldsymbol{X}_t}{\textnormal{d}t} = - \nabla \mathcal{R}(\boldsymbol{X}_t),
\end{equation}
where the empirical risk \(\mathcal{R}\) is given in~\eqref{eq: empirical risk}. We observe that~\eqref{eq: GF 2} is equivalent to studying the solution \(\boldsymbol{X}_t \in \mathcal{S}_{N,r}\) of
\begin{equation} \label{eq: gradient flow 2}
\frac{\textnormal{d} \boldsymbol{X}_t}{\textnormal{d}t} = - \nabla H(\boldsymbol{X}_t),
\end{equation}
where the Hamiltonian \(H \colon \mathcal{S}_{N,r} \to \R\) is defined as \(H(\boldsymbol{X}) = \sqrt{M} \mathcal{R}(\boldsymbol{X})\). Indeed, multiplying by a factor of \(\sqrt{M}\) changes the timescale of the dynamics but not the nature of the dynamics itself. Specifically, the gradient flow dynamics~\eqref{eq: GF 2} results in
\[
\frac{\text{d} \boldsymbol{X}_t}{\text{d}t} = - \nabla \mathcal{R}(\boldsymbol{X}_t)  =  - \frac{1}{\sqrt{M}} \nabla H(\boldsymbol{X}_t),
\]
and introducing a new timescale \(\tau = \frac{t}{\sqrt{M}}\) yields
\[
\frac{\text{d} \boldsymbol{X}_{\tau \sqrt{M}}}{\text{d} \tau}  = -  \nabla H(\boldsymbol{X}_{\tau \sqrt{M}}).
\]
Thus, the only difference is that this rescaled dynamics speeds up the process, reducing the runtime by the factor \(\sqrt{M}\). The advantage of studying gradient dynamics with Hamiltonian \(H\) is that we can build on the results obtained with Langevin dynamics of our companion work~\cite{benarous2024langevin}. From this point onward, we consider the gradient flow \(\boldsymbol{X}_t\) as the solution to~\eqref{eq: gradient flow 2}.

\subsection{Initial conditions}

As discussed in Subsection~\ref{subsection: main results}, we consider the gradient flow dynamics initialized from a random point \(\boldsymbol{X}_0\) drawn according to the uniform measure \(\mu_{N \times r}\) on \(\mathcal{S}_{N,r}\). Our recovery guarantees extend beyond this uniform initialization to a broader class of random initial data, provided certain natural conditions are satisfied. Let \(\mathcal{M}_1 (\mathcal{S}_{N,r})\) denote the space of probability measures on \(\mathcal{S}_{N,r}\). Then, a choice of initialization corresponds to a choice of measure \(\mu_N \in \mathcal{M}_1 (\mathcal{S}_{N,r})\). We now introduce the conditions under which our guarantees continue to hold. The first condition ensures that the initial correlations are on the typical scale of order \(\Theta(N^{-\frac{1}{2}})\).
    
\begin{defn}[Condition 1]\label{def: condition 1 GF}
For every \(N \in \N\) and every \(\gamma_1 > \gamma_2 >0\), define
\[
\mathcal{C}_1^{(N)} (\gamma_1,\gamma_2) = \left \{ \boldsymbol{X} \in \mathcal{S}_{N,r} \colon \frac{\gamma_2}{\sqrt{N}} \leq m_{ij}^{(N)}(\boldsymbol{X}) < \frac{\gamma_1}{\sqrt{N}} \quad \text{for all} \enspace 1 \leq i,j \leq r\right \}.
\]
We say that a sequence of random probability measures \(\mu_N \in \mathcal{M}_1(\mathcal{S}_{N,r})\) satisfies \emph{Condition 1} if for every \(N \in \N\) and \(\gamma_1 > \gamma_2 > 0\),
\[ 
\mu_N \left ( \mathcal{C}_1^{(N)}(\gamma_1,\gamma_2) ^\textnormal{c} \right ) \leq C_1 e^{-c_1 \gamma_1^2} + C_2 e^{-c_2 \gamma_2 \sqrt{N}} + C_3 \gamma_2,
\]
where \(C_1, c_1, C_2, c_2, C_3 >0\) are absolute constants independent of \(N\).
\end{defn} 
    
The second condition ensures that the initial correlations, weighted by their associated SNRs, are sufficiently separated across index pairs.
    
\begin{defn}[Condition 2] \label{def: condition 2 GF}
For every \(N \in \N\) and every \(\gamma_1 > \gamma_3 >0\),define
\[
\mathcal{C}_2^{(N)}(\gamma_1, \gamma_3) = \left \{\boldsymbol{X} \in \mathcal{S}_{N,r} \colon \left | \frac{\lambda_i \lambda_j m_{ij}^{(N)}(\boldsymbol{X})^{p-2}}{\lambda_k \lambda_\ell m_{k \ell}^{(N)}(\boldsymbol{X})^{p-2}} - 1 \right | > \frac{\gamma_3}{\gamma_1} \enspace \text{for every} \enspace 1 \leq i,j,k,\ell \leq r, (i,j) \neq (k,\ell) \right \}.
\]
We say that a sequence of random probability measures \(\mu_N \in \mathcal{M}_1(\mathcal{S}_{N,r})\) satisfies \emph{Condition 2} if for every \(N \in \N\) and every \(\gamma_1 > \gamma_3 >0\), 
\[ 
\mu_N \left ( \mathcal{C}_2^{(N)}(\gamma_1,\gamma_3)^\textnormal{c} \right ) \leq C_1 e^{- c_1 \gamma_1^2} + C_2 e^{- c_2 \sqrt{N} \gamma_3} + C_3 \sup_{i,j,k,\ell} \left ( 1 + \left( \frac{\lambda_i \lambda_j}{\lambda_k \lambda_\ell} \right)^{\frac{2}{p-2}} \right)^{-\frac{1}{2}} \gamma_3,
\]
where \(C_1, c_1, C_2, c_2, C_3 >0\) are absolute constants independent of \(N\).
\end{defn}

We also need a further condition on the regularity of the noise generator \(L_0\).

\begin{defn}[Condition 0 at level \(n\)] \label{def: condition 0 GF}
For every \(N \in \N\), every \(\gamma_0 >0\), and every \(n\geq 1\), define
\begin{equation*} \label{eq: regularity noise generator}
\mathcal{C}_0^{(N)} (n,\gamma_0) = \bigcap_{k=0}^{n-1} \left \{ \boldsymbol{X} \in \mathcal{S}_{N,r} \colon |L_0 ^k m_{ij}^{(N)}(\boldsymbol{X})| \leq \frac{\gamma_0}{\sqrt{N}}  \enspace \text{for every} \enspace 1 \leq i,j \leq r\right \}.
\end{equation*}
We say that a sequence of random probability measures \(\mu_N \in \mathcal{M}_1(\mathcal{S}_{N,r})\) satisfies \emph{Condition 0 at level \(n\)} if for every \(N \in \N\) and every \(\gamma_0 > 0\),
\begin{equation*} 
\mu_N \left ( \mathcal{C}_0^{(N)} (n,\gamma_0)^\textnormal{c} \right) \leq C e^{-c\gamma_0^2},
\end{equation*}
where \(C,c >0\) are absolute constants independent of \(N\).
\end{defn}

\begin{defn}[Condition 0 at level \(\infty\)] \label{def: condition 0 infty GF}
For every \(N \in \N\), every \(\gamma_0 >0\), and every \(T >0\), define
\begin{equation*} \label{eq: regularity initial data infty}
\mathcal{C}_0^{(\infty,N)} (T,\gamma_0) = \left \{ \boldsymbol{X} \in \mathcal{S}_{N,r} \colon \sup_{t \le T} |e^{t L_0} L_0 m_{ij}^{(N)}(\boldsymbol{X})| \leq \frac{\gamma_0}{\sqrt{N}}  \enspace \text{for every} \enspace 1 \leq i,j \leq r\right \},
\end{equation*}
where \(e^{t L_0}\) denotes the semigroup generated by the operator \(L_0\). We say that a sequence of random probability measures \(\mu_N \in \mathcal{M}_1(\mathcal{S}_{N,r})\) \emph{weakly satisfies Condition 0 at level \(\infty\)} if for every \(N \in \N\), \(\gamma_0\), and \(T > 0\),
\begin{equation*} 
\mu_N \left (\mathcal{C}_0^{(\infty,N)}(T,\gamma_0)^\textnormal{c} \right) \leq C \sqrt{N}T e^{-c\gamma_0^2},
\end{equation*}
where \(C,c >0\) are absolute constants independent of \(N\).
\end{defn}

The most natural initialization is the uniform measure \(\mu_{N \times r}\) on \(\mathcal{S}_{N,r}\). We claim that
\begin{lem} \label{lem: invariant measure GF}
The uniform measure \(\mu_{N \times r}\) on \(\mathcal{S}_{N,r}\) weakly satisfies Condition 0 at level \(\infty\), and satisfies Condition 1 and Condition 2.  
\end{lem}

The proof of Lemma~\ref{lem: invariant measure GF} is deferred to Appendix~\ref{appendix: invariant measure}. 

\subsection{Main results in nonasymptotic form}

We are now ready to state our main results in nonasymptotic form under gradient flow dynamics with Hamiltonian \(H\). The corresponding nonasymptotic results for the original gradient flow dynamics given by~\eqref{eq: GF} remain the same in terms of sample complexity thresholds and SNR conditions. The only difference lies in the required runtime, which must be scaled by a factor of \(\sqrt{M}\) in the original dynamics, as explained above. Furthermore, in light of Remark~\ref{rmk: sign initialization}, we assume a positive initialization of the correlations. This allows us to drop the absolute values of the correlations in the subsequent statements and implies that \(r_\text{c} = r\). Finally, we denote by \((i_1^\ast, j_1^\ast), \ldots, (i_r^\ast, j_r^\ast)\) the greedy maximum selection of the initialization matrix \(\boldsymbol{I}_0\), defined in~\eqref{eq: initialization matrix} (see also Definition~\ref{def: greedy operation}).

We first present the recovery of the first spike. To enhance the clarity of our statement, we introduce the following definition.

\begin{defn}
We say that the $j$th column $(\boldsymbol{X}_{T_0})_{j}$ of the gradient flow process $(\boldsymbol{X}_t)_{t \geq 0}$, initialized at $\boldsymbol{X}_0 \sim \mu_0^{(N)} \in \mathcal{M}_1(\mathcal{S}_{N,r})$, \emph{recovers} the signal vector $\boldsymbol{v}_i$ at \emph{time} $T_0$ with \emph{precision} $\varepsilon >0$ and \emph{rate} $\xi>0$ if, for every $T \geq T_0$,
\[
\int_{\mathcal{S}_{N,r}} \mathbb{P}_{\boldsymbol{X}^+} \left(\inf_{t \in [T_0,T]} m_{ij}(\boldsymbol{X}_t) \geq 1-\varepsilon\right) \textnormal{d} \mu_0^{(N)} (\boldsymbol{X}) \geq \xi.
\]
Here, \(\boldsymbol{X}^+\) denotes the initialization conditioned on \(m_{ij}(\boldsymbol{X}_0) > 0\) for every \(i,j \in [r]\).
\end{defn}

Our first result determines the sample complexity required to efficiently recover the leading spike (up to a permutation). 

\begin{prop}[Recovery of the first spike] \label{thm: strong recovery first spike GF p>2}
Consider a sequence of initializations \(\mu_0^{(N)} \in \mathcal{M}_1(\mathcal{S}_{N,r})\) satisfying Condition \(0\) at level \(n\), Condition \(1\), and Condition \(2\). Then, the following holds: for every \(n \geq 1\), \(\gamma_0 > 0\), \(\gamma_1 > \gamma_2 \vee \gamma_3\), \(c_0 \ge 2 \left (1 + \frac{\gamma_1}{\gamma_3}\right)\), and \(\varepsilon > 0\), there exists \(C = C (p,\gamma_0,\gamma_2,c_0, \{\lambda_i\}_{i=1}^r)\) such that if \(\sqrt{M} \geq C (n+2)  N^{\frac{p-1}{2} - \frac{n}{2(n+1)}}\) and \(N\) is sufficiently large, then the column vector \((\boldsymbol{X}_{T_0})_{j_1^\ast}\) of the gradient flow process recovers \(\boldsymbol{v}_{i_1^\ast}\) at time \(T_0 \gtrsim \frac{1}{(n+2) \gamma_0}  N^{-\frac{1}{2(n+1)}}\) with precision \(\varepsilon\) and rate at least $1 - \frac{1}{C}$.
\end{prop}

\begin{rmk} \label{rmk1}
The constant \(C=C (p,\gamma_0,\gamma_2,c_0, \{\lambda_i\}_{i=1}^r)\) in Proposition~\ref{thm: strong recovery first spike GF p>2} takes the form 
\[
C = C'\frac{\gamma_0c_0}{p \lambda_r^2 \gamma_2^{p-1}},
\]
where \(C'\) is an absolute constant. Moreover, the convergence rate can be more precisely lower bounded by \(1-\eta\), where
\[
\eta = C_1 e^{-c_1 \gamma_0^2} + C_2 e^{-c_2 \gamma_1^2} + C_3 e^{- c_3 (\gamma_2 + \gamma_3) \sqrt{N}} + C_4 \gamma_2 + C_5  \gamma_3 + e^{-K N}.
\]
Here, the constants \(C_i, c_i\) depend only on those in Definitions~\ref{def: condition 1 GF},~\ref{def: condition 2 GF}, and~\ref{def: condition 0 GF}. The constant \(K\) depends only on $p,n$, and $\{\lambda_{i}\}_{i=1}^{r}$, and arises from the norm control of the noise Hamiltonian \(H_0\) (see Lemma~\ref{lem: regularity H0}). Lastly, the notation $\gtrsim$ in the expression for $T_0$ hides a constant that depends only on $\varepsilon$ and $\{\lambda_{i}\}_{i=1}^{r}$.
\end{rmk}

Proposition~\ref{thm: strong recovery first spike GF p>2} shows that the sample complexity required to efficiently recover the first spike (up to a permutation) matches the threshold in the single-spike case. Our next result determines the sample complexity needed for recovering a permutation of all spikes. To state it precisely, we first introduce the following definition.

\begin{defn}
For every subset $A \subset \mathcal{S}_{N, r}$, let $\mathcal{T}_A$ denote the first hitting time of \(S\) by the gradient flow $(\boldsymbol{X}_t)_{t \ge 0}$, that is,
\[
\mathcal{T}_A \coloneqq \inf\{t \ge 0\colon\boldsymbol{X}_t \in A\}.
\]
We say that the gradient flow $(\boldsymbol{X}_t)_{t \ge 0}$, initialized at $\boldsymbol{X}_0 \sim \mu_0^{(N)} \in \mathcal{M}_1(\mathcal{S}_{N,r})$, \emph{reaches} $A$ by time $T_0$ with rate $\xi>0$ if
\[
\int_{\mathcal{S}_{N,r}} \mathbb{P}_{\boldsymbol{X}^+} \left(\mathcal{T}_A \geq T_0 \right) \mathrm{d} \mu_0^{(N)} (\boldsymbol{X}) \leq \xi.
\]
Here, \(\boldsymbol{X}^+\) denotes the initialization conditioned on \(m_{ij}(\boldsymbol{X}_0) > 0\) for every \(i,j \in [r]\).
\end{defn}

\begin{prop}[Recovery of all spikes] \label{thm: strong recovery all spikes GF p>2}
For every \(\varepsilon > 0\), define the set
\begin{equation} \label{eq: set strong recovery GF p>2}
\begin{split}
R(\varepsilon) = \Big \{ \boldsymbol{X} & \colon m_{i_k^\ast j_k^\ast}^{(N)} (\boldsymbol{X}) \geq 1 - \varepsilon  \enspace \forall k \in [r] \enspace \textnormal{and} \enspace \\
& \quad m_{ij}^{(N)}(\boldsymbol{X}) \lesssim \log(N)^{-\frac{1}{2}}N^{-\frac{p-1}{4}} \enspace \forall (i,j) \in [r]^2 \backslash \cup_{k=1}^{r} (i_k^\ast, j_k^\ast) \Big \},
\end{split}
\end{equation}
where \(\lesssim\) hides an absolute constant. Consider a sequence of initializations \(\mu_0^{(N)} \in \mathcal{M}_1(\mathcal{S}_{N,r})\) satisfying Condition 1 and Condition 2. Then, the following holds: for every \(\gamma_1 > \gamma_2 \vee \gamma_3\), \(c_0 \ge 2 \left (1 + \frac{\gamma_1}{\gamma_3}\right)\), and \(\varepsilon > 0\), there exists a constant \(C = C(p,r,\gamma_2,c_0,\{\lambda_i\}_{i=1}^r)\) such that if \(\sqrt{M} \ge C N^{\frac{p-1}{2}}\), then for sufficiently large $N$, the gradient flow $(\boldsymbol{X}_t)_{t \geq 0}$ reaches $R(\varepsilon)$ at some time \(T_0 \gtrsim \frac{1}{\sqrt{N}}\), with rate at most $\frac{1}{C}$.
\end{prop}

\begin{rmk} \label{rmk2}
The constant \(C= C(p,r,\gamma_2, c_0, \{\lambda_i\}_{i=1}^r)\) in Proposition~\ref{thm: strong recovery all spikes GF p>2} is given by 
\[
C = C' \frac{\Lambda c_0}{p \lambda_r^2 \gamma_2^{p-1}}, 
\]
where \(C'\) is an absolute constant and \(\Lambda\) depends only on \(p,r\), and \(\{\lambda_i\}_{i=1}^r\). As in Proposition~\ref{thm: strong recovery first spike GF p>2}, our proofs establish a sharper lower bound on the convergence rate, given by 
\[
\eta = C_1 e^{- c_1 \gamma_1^2} + C_2 e^{- c_2 (\gamma_2+ \gamma_3) \sqrt{N}} + C_3 \gamma_2 + C_4  \gamma_3 + e^{-K N},
\]
where the constants \(C_i, c_i\) arise from Definitions~\ref{def: condition 1 GF} and ~\ref{def: condition 2 GF}, while the constant \(K\) depends only on \(p\) and \(\{\lambda_i\}_{i=1}^r\), and is derived from Lemma~\ref{lem: regularity H0}. Finally, note that the symbol $\gtrsim$, used for $T_0$, hides a constant that depends only on $\varepsilon$ and the eigenvalues $\{\lambda_{i}\}_{i=1}^{r}$.
\end{rmk}

As discussed in Subsection~\ref{subsection: main results}, the sample complexity required for recovery of a permutation of all spikes scales as \(N^{p-1}\), compared to \(N^{p-2}\) for the recovery of the first direction. This is because we are not able to exploit the advantageous scaling of the noise \(L_0 m_{ij}^{(N)}\), once \(m_{i_1^\ast j_1^\ast}^{(N)}\) becomes macroscopic, as explained in Subsection~\ref{subsection: outline proofs}.

\begin{rmk} \label{rmk: langevin SNRS}
In our companion paper~\cite{benarous2024langevin}, we show that under Langevin dynamics, the permutation of the recovered spikes correspond to the identity permutation, achieving thus exact recovery, provided the SNRs satisfy
\[
\lambda_i > \frac{c_0+1}{c_0-1} \left ( \frac{3 \gamma_1}{\gamma_2} \right)^{p-2} \lambda_{i+1},
\]
for every \(1 \le i \le r-1\). This also extends to gradient flow dynamics.
\end{rmk}

We now present the proof of Theorem~\ref{thm: strong recovery GF p>2 asymptotic}. The proofs of Propositions~\ref{thm: strong recovery first spike GF p>2} and~\ref{thm: strong recovery all spikes GF p>2} are deferred to Section~\ref{proofs}.

\begin{proof} [\textbf{Proof of Theorem~\ref{thm: strong recovery GF p>2 asymptotic}}]
According to Lemma~\ref{lem: invariant measure GF}, the uniform measure \(\mu_{N \times r}\) on \(\mathcal{S}_{N,r}\) satisfies Condition \(1\) and Condition \(2\), and weakly satisfies Condition \(0\) at level \(\infty\). To prove Theorem~\ref{thm: strong recovery GF p>2 asymptotic}, we must identify suitable sequences (in \(N\)) for the parameters \(\gamma_0, \gamma_1, \gamma_2\), and \(\gamma_3\) that govern the rates $\eta$ appearing in Remarks~\ref{rmk1} and~\ref{rmk2}, ensuring that $\eta$ vanishes in the large-\(N\) limit. Both Propositions~\ref{thm: strong recovery first spike GF p>2} and~\ref{thm: strong recovery all spikes GF p>2} depend on a control parameter \(c_0\) and require sufficiently large \(N\). Hence, we need to show that the assumptions of Theorem~\ref{thm: strong recovery GF p>2 asymptotic} are sufficient to guarantee the existence of such sequences for the parameters \(\gamma_0, \gamma_1, \gamma_2\), and \(\gamma_3\), while satisfying the constraints on $c_0$ and $N$. 

We begin by proving item (a). Let \(\alpha>p-2\), and define
\[
\nu \coloneqq \frac{\alpha - (p-2)}{2} > 0, \quad n_0 (\nu) \coloneqq \left \lfloor \frac{1}{2\nu} \right \rfloor.
\]
Then for every \(n \ge n_0(\nu)\), we have
\[
\frac{1}{2(n+1)} < \nu.
\]
Now, from Proposition~\ref{thm: strong recovery first spike GF p>2} and Remark~\ref{rmk1}, the required condition for $\sqrt{M}$ is given by
\[
\sqrt{M} \ge \omega(n) \coloneqq C'\frac{\gamma_0c_0}{p\lambda_{r}^{2}\gamma_{2}^{p-1}}(n+2)N^{\frac{p-2}{2}+\frac{1}{2(n+1)}}.
\]
Fix $n \ge n_0(\nu)$. By construction of \(n_0 (\nu)\), we have 
\[
\frac{p-2}{2} + \frac{1}{2(n+1)}  < \frac{\alpha}{2},
\]
so that for sufficiently large \(N\), the condition \(\sqrt{M} = N^{\alpha/2} > \omega(n)\) is satisfied. Applying Proposition~\ref{thm: strong recovery first spike GF p>2} with \(\mu_{N \times r}\) for the initialization mesure, we obtain that there exists 
\[
T_0 \gtrsim \frac{1}{\gamma_0 (n+2)} N^{-\frac{1}{2(n+1)}},
\]
such that for every $\varepsilon >0$,
\[
\int \mathbb{P}_{\boldsymbol{X}^+}\left(\inf_{t \in [T_0,T]} m_{i^{\ast}_{1}j^{\ast}_{1}}\left(\boldsymbol{X}_t\right) \geq 1-\varepsilon\right) \mathrm{d} \mu_{N \times r} \geq 1-\eta,
\]
where the error term $\eta$ is given by
\[
\eta = C_1 e^{-c_1 \gamma_0^2} + C_2 e^{-c_2 \gamma_1^2} + C_3 e^{- c_3 (\gamma_2 + \gamma_3) \sqrt{N}} + C_4 \gamma_2 + C_5  \gamma_3 + e^{-K N}.
\]
We must ensure that all assumptions of Proposition~\ref{thm: strong recovery first spike GF p>2} are satisfied. In particular, in the proof of Proposition~\ref{thm: strong recovery first spike GF p>2}, a necessary condition for controlling the generator correction term (see~\eqref{eq:comp_gen}) is given by~\eqref{eq: condition on N}, i.e., 
\begin{equation} \label{eq:first_cond_N}
N \geq \frac{r^{2}\lambda_{1}^{2}\tilde{\gamma}^{p+1}}{C_0\lambda_r^2 \gamma_2^{p-1}},
\end{equation}
where \(\tilde{\gamma} > \gamma_1\) is of the same order, and \(C_0 = 1 / c_0\) must satisfy
\[
C_0 \leq \frac{\gamma_3/\gamma_1}{2 (1 + \gamma_3 / \gamma_1)}.
\]
Since Proposition~\ref{thm: strong recovery first spike GF p>2} holds for all such $C_0$, we may take the largest admissible value. Substituting this into~\eqref{eq:first_cond_N} and replacing \(\tilde{\gamma} \sim \gamma_1\), we obtain
\[
N \geq 2 \frac{r^2 \lambda_1^2 \gamma_1^{p+2}}{\lambda_r^2 \gamma_2^{p-1} \gamma_3} \left( 1+\frac{\gamma_3}{\gamma_1}\right).
\]
Several similar conditions arise in our companion paper~\cite{benarous2024langevin}, typically with fractional powers of \(N\) on the left-hand side and slightly milder dependencies on the parameters $\gamma_0, \gamma_{1}, \gamma_{2}$, and $\gamma_3$ on the right-hand side. Thus, to ensure all such constraints, we focus on the condition 
\begin{equation} \label{eq:fourth_cond_N}
N^\kappa \geq 2 \frac{r^2 \lambda_1^2 \gamma_1^{p+2}}{\lambda_r^2 \gamma_2^{p-1} \gamma_3} \left(1+\frac{\gamma_3}{\gamma_1}\right).
\end{equation}
for some fixed $\kappa >0$, independent of all other parameters. We now choose the parameter sequences so that \(\gamma_0, \gamma_1 \to \infty\) and \(\gamma_2,\gamma_3 \to 0\), in a way that ensures condition~\eqref{eq:fourth_cond_N} is satisfied. A concrete admissible choice is
\[
\gamma_0(N) = \gamma_1 (N) = \log(N) \quad \text{and} \quad \gamma_{2}(N) = \gamma_3(N) = \frac{1}{\log(N)}.
\]
Substituting into the expression for \(\eta\), we obtain
\[
\eta = C_1 e^{-c_1 \log^2(N)} + C_2 e^{-c_2 \log^2(N)} + C_3 e^{-c_3 \sqrt{N} / \log(N)} + \mathcal{O} \left ( \frac{1}{\log(N)} \right) + e^{-KN},
\]
which implies that 
$$\lim_{N \to \infty} \eta = 0.$$
Moreover, under the same parameter choice, condition~\eqref{eq:fourth_cond_N} reduces to
\[
N^\kappa \ge C \log^{2p+2}(N),
\]
for some constant \(C>0 \), which clearly holds for sufficiently large $N$. This verifies that all required assumptions are met and completes the proof of part (a).

To prove item (b), we observe that, unlike in part (a), it is not necessary to introduce auxiliary quantities such as \(\nu\) and \(n_0(\nu)\), since the bounding flows method from Lemma~\ref{lem: bounding flows} does not apply in this case. We begin by defining the sequence  
$$a_N \coloneqq \frac{M}{N^{p-1}}.$$ 
Under the assumption \(M \gg N^{p-1}\), we have $\lim_{N \to \infty} a_N = \infty$. According to Proposition~\ref{thm: strong recovery all spikes GF p>2}, it suffices to verify that the error term \(\eta\) vanishes as \(N \to \infty\), and that the sample complexity condition \(\sqrt{M} \ge C N^{\frac{p-1}{2}}\) is satisfied, where \(C = C(p, r, \gamma_2, c_0, \{\lambda_i\})\) as given in Remark~\ref{rmk2}. To this end, we define the following parameter sequences:
$$\gamma_0(N)  = \gamma_{1}(N) = \log(a_N) \quad \text{and} \quad \gamma_{2}(N) = \gamma_3(N) = \frac{1}{\log(a_N)}.$$
Substituting these into the convergence rate expression from Remark~\ref{rmk2}, we obtain
\[
\lim_{N \to \infty} \eta =0,
\]
provided that 
$$\lim_{N \to \infty} \frac{\sqrt{N}}{\log(a_N)} = \infty,$$
which ensures that the term \(e^{-c_3 (\gamma_2 + \gamma_3) \sqrt{N}}\) in the error bound vanishes asymptotically. This condition clearly holds whenever \(a_N \to \infty\) grows at least polynomially in \(N\), as is the case here. It remains to verify that the sample complexity condition \(\sqrt{M} \ge C N^{\frac{p-1}{2}}\) is satisfied under our parameter choices. From Remark~\ref{rmk2}, we have
\[
C = C' \frac{\Lambda c_0}{p \lambda_r^2 \gamma_2^{p-1}}, 
\]
with \(c_0 \ge 2 (1 + \gamma_1/\gamma_3)\). We therefore find that \(C = \Theta (\log^{p+1}(a_N))\), and thus the sample complexity condition becomes
\[
\sqrt{M} = \sqrt{a_N} N^{\frac{p-1}{2}} \ge C N^{\frac{p-1}{2}},
\]
which holds for sufficiently large \(N\). Thus, all assumptions of Proposition~\ref{thm: strong recovery all spikes GF p>2} are satisfied in the large-\(N\) limit. This completes the proof of part (b).
\end{proof}

\begin{proof}[\textbf{Proof of Theorem~\ref{thm: strong recovery online p>2 asymptotic stronger}}]
The proof follows from the analysis in Section~\ref{proofs}, particularly from Lemmas~\ref{lem: E_1 GF} and~\ref{lem: E_2 GF}. The asymptotic formulation of the statement follows a similar approach to that used in the proof of Theorem~\ref{thm: strong recovery GF p>2 asymptotic}.
\end{proof}
\section{Preliminary results} \label{preliminary}

In this section, we present preliminary results that are crucial for proving the main results in Section~\ref{section: main results}. The proofs are deferred to our companion paper on Langevin dynamics~\cite{benarous2024langevin}, where these results are stated in greater generality for both Langevin and gradient flow dynamics.

\subsection{Ladder relations and bounding flows method}

Recall the Hamiltonian \(H_0 \colon \mathcal{S}_{N,r} \to \R\) defined by
\[
H_0 (\boldsymbol{X}) = N^{-\frac{p-1}{2}} \sum_{i=1}^r \lambda_i \langle \boldsymbol{W}, \boldsymbol{x}_i^{\otimes p}\rangle,
\]
where \(\boldsymbol{W} \in (\R^N)^{\otimes p}\) is an order-\(p\) tensor with i.i.d.\ entries \(W_{i_1, \ldots, i_p} \sim \mathcal{N}(0,1)\), and \(\mathcal{S}_{N,r}\) denotes the normalized Stiefel manifold defined in~\eqref{def: normalized stiefel manifold}. Following the approach in~\cite{ben2020bounding, arous2020algorithmic}, we work with the \(\mathcal{G}\)-norm, which is motivated by the homogeneous Sobolev norm and which we introduce as follows.

\begin{defn}[\(\mathcal{G}\)-norm on \(\mathcal{S}_{N,r}\)] \label{def: G norm}
For any integer \(k\), we say that a function \(F \colon \mathcal{S}_{N,r} \to \R\) is in the space \(\mathcal{G}^k(\mathcal{S}_{N,r})\) if
\[
\norm{F}_{\mathcal{G}^k} \coloneqq \sum_{\ell=0}^k N^{\ell/2} \norm{|\nabla^\ell F|_{\text{\textnormal{op}}}}_{L^\infty(\mathcal{S}_{N,r})} < \infty.
\]
Here, \(\nabla^\ell F\) denotes the \(\ell\)th Riemannian (covariant) derivative of \(F\), defined as a tensor field of order \(\ell\). For every \(\boldsymbol{X} \in \mathcal{S}_{N,r}\), it defines an \(\ell\)-linear map on the tangent space \(T_{\boldsymbol{X}}\mathcal{S}_{N,r}\):
\[
\nabla^\ell F (\boldsymbol{X}) \colon T_{\boldsymbol{X}}\mathcal{S}_{N,r} \times \cdots \times T_{\boldsymbol{X}}\mathcal{S}_{N,r}  \to \R.
\]
This map is defined recursively by
\[
\begin{split}
\nabla^\ell F(\boldsymbol{X}; \boldsymbol{U}_1, \ldots, \boldsymbol{U}_\ell)
&= \nabla_{\boldsymbol{U}_1} \nabla^{\ell-1} F(\boldsymbol{X}; \boldsymbol{U}_2, \ldots, \boldsymbol{U}_\ell) - \sum_{j=2}^\ell \nabla^{\ell-1} F(\boldsymbol{X}; \boldsymbol{U}_2, \ldots, \nabla_{\boldsymbol{U}_1} \boldsymbol{U}_j, \ldots, \boldsymbol{U}_\ell).
\end{split}
\]
for all \(\boldsymbol{U}_1, \ldots, \boldsymbol{U}_\ell \in T_{\boldsymbol{X}} \mathcal{S}_{N,r}\). The operator norm of \(\nabla^\ell F\) is given by
\[
|\nabla^\ell F|_{\textnormal{op}} (\boldsymbol{X}) = \sup_{\boldsymbol{U}_1, \ldots, \boldsymbol{U}_\ell \in T_{\boldsymbol{X}} \mathcal{S}_{N,r}, \norm{\boldsymbol{U}_i}_\textnormal{F} \le 1} \lvert \nabla^\ell F (\boldsymbol{X}; \boldsymbol{U}_1, \ldots, \boldsymbol{U}_\ell)\rvert,
\]
where \(\norm{\boldsymbol{U}}_\textnormal{F} = \sqrt{\Tr (\boldsymbol{U}^\top \boldsymbol{U})}\) is the Frobenius norm. For further details, see~\cite[Section 10.7]{boumal2023} and the references therein.
\end{defn}

We emphasize that this definition generalizes the \(\mathcal{G}\)-norm introduced by~\cite{ben2020bounding} for functions on the sphere \(\mathbb{S}^{N-1}(\sqrt{N})\). We now state the following key estimate for the \(\mathcal{G}\)-norm of \(H_0\).

\begin{lem}[Regularity of \(H_0\)] \label{lem: regularity H0}
For every \(n\), there exist \(C_1 = C_1(p,n)\) and \(C_2 = C_2(p,n) >0\) such that
\[
\mathbb{P}\left(\norm{H_0}_{\mathcal{G}^n} \geq C_1 \left (\sum_{i=1}^r \lambda_i \right ) N \right) \leq \exp \left (- C_2 \frac{(\sum_{i=1}^r \lambda_i)^2}{\sum_{i=1}^r \lambda_i^2} N\right ).
\]
\end{lem}

Lemma~\ref{lem: regularity H0} reduces to~\cite[Theorem 4.3]{ben2020bounding} in the special case \(r=1\). Its proof follows the same strategy as that of~\cite{ben2020bounding}, to which we refer the reader for details. 

We next present the ladder relations, which will be useful to bound \(\norm{L_0 m_{ij}^{(N)}}_\infty\), where we recall from~\eqref{eq: generator GF noise} that the generator \(L_0\) is given by \(L_0 =  - \langle \nabla H_0, \hat{\nabla} \cdot \rangle \). Since the Riemannian gradient at a point $\boldsymbol{X} \in \mathcal{S}_{N,r}$ is obtained by projecting the Euclidean gradient onto the tangent space \(T_{\boldsymbol{X}} \mathcal{S}_{N,r}\) at $\boldsymbol{X}$ (see~\eqref{eq: riemannian gradient on normalized stiefel}), and since this projection preserves inner products with the Euclidean gradient, it follows that 
\[
\langle \nabla H_0,\hat{\nabla} \cdot \rangle  = \langle \nabla H_0, \nabla \cdot \rangle.
\]
Here, we recall that \(\hat \nabla\) denotes the Euclidean gradient, while \(\nabla\) denotes the Riemmanian gradient.
 
\begin{lem}[Ladder relations] \label{lem: ladder relations}
Let \(L\) be any linear operator acting on the space of smooth functions \(F \colon \mathcal{S}_{N,r} \to \R\), and let \(n \geq m \geq 1\) be integers. Define 
\[
\norm{L}_{\mathcal{G}^n \to \mathcal{G}^m} \coloneqq \sup_{F \in  \mathcal{G}^n (\mathcal{S}_{N,r})} \frac{\norm{LF}_{\mathcal{G}^m}}{ \norm{F}_{\mathcal{G}^n}}.
\]
Then, for every \(n \geq 1\), there exists a constant \(c(n)\) such that for every \(N\), \(r\), and every $G \in \mathcal{G}^n(\mathcal{S}_{N,r})$,
\[
\norm{\langle \nabla G, \nabla \cdot \rangle}_{\mathcal{G}^n \to \mathcal{G}^{n-1}} \leq \frac{c(n)}{N} \norm{G}_{\mathcal{G}^n}. 
\]
\end{lem}

The proof of Lemma~\ref{lem: ladder relations} is provided in~\cite[Lemma 4.3]{benarous2024langevin}. Applying this result, we can estimate \(\norm{L_0 m_{ij}^{(N)}}_\infty\) for every \(1 \leq i,j \leq r\). In light of Lemma~\ref{lem: regularity H0}, for every \(n \geq 1\), there exist constants \(K = K (p,n, \{\lambda_i\}_{i=1}^r)\) and \(C = C (p,n, \{\lambda_i\}_{i=1}^r)\) such that 
\[
\norm{H_0}_{\mathcal{G}^n} \leq C N, 
\]
with \(\mathbb{P}\)-probability at least \(1 - \exp(- K N)\). Moreover, a direct computation shows that \(\norm{m_{ij}^{(N)}}_{\mathcal{G}^n} \leq c(n)\). Therefore, by Lemma~\ref{lem: ladder relations}, there exists a constant \(\Lambda = \Lambda(p,r, \{\lambda_i\}_{i=1}^r)\) such that 
\begin{equation} \label{eq: bound norm L_0m_ij}
\norm{L_0 m_{ij}}_\infty \leq \norm{ \langle \nabla H_0, \hat{\nabla} m_{ij}^{(N)}\rangle}_\infty \leq \frac{1}{N} \norm{H_0}_{\mathcal{G}^1} \norm{m_{ij}^{(N)}}_{\mathcal{G}^1} \leq \Lambda
\end{equation}
with \(\mathbb{P}\)-probability at least \(1 - \exp(-KN)\). \\

The \emph{bounding flows} method provides a sharper estimate of \(\norm{L_0 m_{ij}}_\infty \). This technique was introduced in~\cite{ben2020bounding} and later used in~\cite{arous2020algorithmic} to provide a precise control over the evolution of functions under Langevin and gradient flow dynamics on \(\mathbb{S}^{N-1}(\sqrt{N})\). Here, we extend the method in order to obtain more accurate bounds for the evolution of functions under gradient flow on the manifold \(\mathcal{S}_{N,r}\). In particular, the following result generalizes~\cite[Theorem 5.3]{arous2020algorithmic} and is extracted from~\cite[Lemma 4.4]{benarous2024langevin}.

\begin{lem}[Bounding flows on \(\mathcal{S}_{N,r}\)] \label{lem: bounding flows}
For every \(\gamma > 0\), define the interval \(I_\gamma = [-\frac{\gamma}{\sqrt{N}}, \frac{\gamma}{\sqrt{N}}]\). Let \(D \subset \mathcal{S}_{N,r}\), and consider a deterministic flow \((\boldsymbol{X}_t)_{t \ge 0}\) defined on \(D\) and evolving according to 
\[
\frac{\textnormal{d} \boldsymbol{X}_t}{\textnormal{d}t} = V(\boldsymbol{X}_t),
\]
where \(V\) is a smooth vector field satisfying \(V(\boldsymbol{X}_t) \in T_{\boldsymbol{X}_t} \mathcal{S}_{N,r}\) for all \(t\). Let \(L\) denote the first-order differential operator associated with the flow, defined as the Lie derivative along \(V\), i.e., 
\[
L  = \langle V, \hat{\nabla} \cdot \rangle.
\]
Suppose that \(\boldsymbol{X}_0 \in D\), and let the exit time be
\[
\mathcal{T}_{D^\text{c}} = \inf \{t \ge 0 \colon \boldsymbol{X}_t \notin D\}. 
\]
Let \(F \colon D \to \R\) be a smooth function. Suppose that the following conditions are satisfied for some integer \(n \geq 1\):
\begin{enumerate}
\item[(1)] The operator \(L\) has the form \(L = L_0 + \sum_{1 \leq i,j \leq r} a_{ij}(\boldsymbol{X}) A_{ij}\), where
\begin{enumerate}
\item[(a)] \(A_{ij} = \langle \nabla \psi_{ij}, \hat{\nabla} \cdot \rangle\) for some function \(\psi_{ij} \in C^\infty (\mathcal{S}_{N,r})\) with \(\norm{\psi_{ij}}_{\mathcal{G}^1} \leq c_1 N\),
\item[(b)] \(a_{ij} \in C^0 (\mathcal{S}_{N,r})\),
\item[(c)] \(L_0 = \langle \nabla U, \hat{\nabla} \cdot \rangle \) for some \(U \in C^{\infty}  (\mathcal{S}_{N,r})\) with \(\norm{U}_{\mathcal{G}^{2n}} \leq c_2(n) N\).
\end{enumerate}
\item[(2)] \(F\) is smooth with \(\norm{F}_{\mathcal{G}^{2n}} \leq c_3(n) \).
\item[(3)] There exists \(\gamma > 0\) such that \(L_0^k F(\boldsymbol{X}_0) \in I_{\gamma}\) for every \(0 \leq k \leq n-1\).
\item[(4)] There exist \(\varepsilon \in (0,1)\) and \(T_0^{(ij)}>0\), possibly depending on \(\varepsilon\), such that for every \(t \leq \mathcal{T}_{D^\text{c}} \wedge T_0^{(ij)}\), 
\[
\int_0^t |a_{ij}(\boldsymbol{X}_s)| \mathrm{d}s  \leq \varepsilon |a_{ij}(\boldsymbol{X}_t)|.
\]
\end{enumerate}
Then, there exists a constant \(K_1>0\), depending only on \(c_1,c_2,c_3\), and \(\gamma\), such that for every \(T_0 > 0\), 
\begin{equation} \label{eq: bounding flows}
| F(\boldsymbol{X}_t)| \leq K_1 \left (  \frac{\gamma}{\sqrt{N}} \sum_{k=0}^{n-1} t^k + t^n + \frac{1}{1-\varepsilon} \sum_{1 \leq i,j \leq r}\int_0^t |a_{ij}(\boldsymbol{X}_s)| \mathrm{d} s\right )
\end{equation}
for every \(t \leq \mathcal{T}_{D^\text{c}} \wedge \min_{1 \leq i,j \leq r} T_0^{(ij)} \wedge T_0\). 

If instead of item (3), the following holds:
\begin{enumerate}
\item[(3')] There exist \(T_1,\gamma >0\) such that \(e^{tL_0}F(\boldsymbol{X}_0) \in I_\gamma\) for every \(t< T_1\),
\end{enumerate}
then the bound~\eqref{eq: bounding flows} holds for every \(t \leq \mathcal{T}_{D^\text{c}} \wedge \min_{1 \leq i,j \leq r} T_0^{(ij)} \wedge T_0 \wedge T_1 \wedge 1\).
\end{lem}

\subsection{Evolution equations for the correlations}

For simplicity of notation, we omit the dependence on \(N\) in \(m_{ij}^{(N)}(\boldsymbol{X})\) and write \(m_{ij}(\boldsymbol{X})\) instead. For every \(i,j \in [r]\), the correlations \(m_{ij}\) are smooth functions from \(\mathcal{S}_{N,r} \subset \R^{N \times r}\) to \(\R\), and they satisfy the integral identity
\[
m_{ij}(\boldsymbol{X}_t) = m_{ij}(\boldsymbol{X}_0) + \int_0^t L m_{ij}(\boldsymbol{X}_s) \mathrm{d} s,
\]
where \(Lm_{ij}(\boldsymbol{X}_t) = - \langle \nabla H_{N,r}(\boldsymbol{X}_t), \hat{\nabla} m_{ij} (\boldsymbol{X}_t)\rangle\). An explicit computation of the generator yields the following evolution equations for the correlation functions \(\{m_{ij}(\boldsymbol{X}_t)\}_{1 \le i,j \le r}\), as established in our companion paper (see~\cite[Lemma 4.6]{benarous2024langevin}).

\begin{lem}[Evolution equation for \(m_{ij}\)] \label{lem: evolution equation m_ij}
For every \(1 \leq i,j \leq r\), 
\[
Lm_{ij} = L_0 m_{ij} + \sqrt{M} p \lambda_i \lambda_j m_{ij}^{p-1}  -  \sqrt{M}\frac{p}{2}\sum_{1 \leq k, \ell \leq r} \lambda_k m_{kj} m_{k \ell} m_{i \ell}\left (\lambda_j m_{kj}^{p-2} + \lambda_\ell m_{k \ell}^{p-2} \right) ,
\]
and
\[
L_0 m_{ij}  =  - \langle \nabla H_0, \hat{\nabla} m_{ij} \rangle.
\]
\end{lem}

We refer to~\cite[Lemma 4.6]{benarous2024langevin} for a proof.

\subsection{Comparison inequalities}

We finally report Lemma 5.1 of~\cite{arous2020algorithmic} that provides simple comparison inequalities for functions. 

\begin{lem}[Bounds on functions] \label{lem: Gronwall}
Let \(\gamma >0\) with \(\gamma \neq 1\), \(c >0\), and \(f \in C_{\text{loc}}([0,T))\) with \(f(0)>0\). 
\begin{itemize}
\item[(a)] Suppose that there exists \(T\) such that \(f\) satisfies the integral inequality
\begin{equation} \label{eq: gronwall}
f(t) \geq a + \int_0^t c f^\gamma(s) \mathrm{d} s,
\end{equation}
for every \(t \leq T\) and some \(a >0\). Then, for \(t \geq 0\) satisfying \((\gamma-1)ca^{\gamma-1}t < 1\), we have that
\[
f(t) \geq a \left ( 1 - (\gamma-1) c a^{\gamma-1} t\right )^{-\frac{1}{\gamma-1}}.
\]
\item[(b)] If the integral inequality~\eqref{eq: gronwall} holds in reverse, i.e., if \(f(t) \leq a + \int_0^t c f^\gamma(s) ds\), then the corresponding upper bound holds.
\item[(c)] If \(\gamma >1\), then \(T \leq t_{\ast}\), where \(t_{\ast} = \left ( (\gamma-1)ca^{\gamma-1} \right )^{-1}\) is called the blow-up time.
\item[(d)] If~\eqref{eq: gronwall} holds with \(\gamma = 1\), then the Grönwall's inequality gives \(f(t) \geq a\exp(ct)\).
\end{itemize}
\end{lem}
\section{Proof of main results} \label{proofs}

In this section, we present the proofs of Propositions~\ref{thm: strong recovery first spike GF p>2} and~\ref{thm: strong recovery all spikes GF p>2}. To simplify notation, we write the correlation functions as \(m_{ij}(\boldsymbol{X})\) instead of \(m_{ij}^{(N)} (\boldsymbol{X})\), and define the time-dependent quantities \(m_{ij}(t) \coloneqq m_{ij}(\boldsymbol{X}_t)\). Moreover, for any \(\varepsilon \in (0,1)\), we denote by \(\mathcal{T}_\varepsilon^{(ij)}\) the hitting time
\[
\mathcal{T}_\varepsilon^{(ij)} \coloneqq \min \{ t \ge 0 \colon m_{ij}(t) \ge \varepsilon \}.
\]

\subsection{Recovery of the first spike (up to a permutation)}

We begin by establishing weak recovery of the leading spike, up to a permutation. By weak recovery, we mean that with high probability, the estimator $\boldsymbol{X}_t$ achieves a nontrivial correlation with one of the columns of the ground truth matrix \(\boldsymbol{V}\) within a given time.

\begin{lem}[Weak recovery of the first spike] \label{lem: weak recovery first spike GF p>2}
Consider a sequence of initializations \(\mu_0^{(N)} \in \mathcal{M}_1(\mathcal{S}_{N,r})\) and let \(\varepsilon_N = C N^{-\frac{p-2}{2(p-1)}}\) for some constant \(C>0\). Then, for every \(n \geq 1, \gamma_0 > 0\), \(\gamma_1 > \gamma_2 \vee \gamma_3 \), and \(C_0 \in \left (0, \frac{\gamma_3/\gamma_1}{2 (1 + \gamma_3/\gamma_1)}\right)\), there exist constants \(K,C>0\) such that if \(\sqrt{M} \geq C \frac{(n+2) \gamma_0}{p \lambda_r^2 C_0 \gamma_2^{p-1}} N^{\frac{p-1}{2}-\frac{n}{2(n+1)}}\) and \(N\) is sufficiently large, 
\[
\int_{\mathcal{S}_{N,r}} \mathbb{P}_{\boldsymbol{X}^+} \left (\mathcal{T}_{\varepsilon_N}^{(i_1^\ast j_1^\ast)} \gtrsim \frac{1}{(n+2) \gamma_0}N^{-\frac{1}{2(n+1)}} \right ) \mathbf{1} \{\mathcal{C}_0^{(N)} (n,\gamma_0) \cap \mathcal{C}_1^{(N)} (\gamma_1,\gamma_2) \cap \mathcal{C}_2^{(N)}(\gamma_1,\gamma_3)\} \mathrm{d} \mu_0^{(N)}(\boldsymbol{X}) \leq e^{-KN},
\]
where the notation \(\gtrsim\) hides only absolute constants, and \((i_1^\ast, j_1^\ast)\) is the first pair in the greedy maximum selection of \(\boldsymbol{I}_0\). 
\end{lem}

Strong recovery of the first spike follows directly from Lemma~\ref{lem: weak recovery first spike GF p>2}, as stated below.

\begin{lem}[Strong recovery from weak recovery]  \label{lem: weak implies strong recovery p>2}
Let \(\varepsilon_N = C N^{-\frac{p-2}{2(p-1)}}\) for some constant \(C>0\). Then, for every \(n \ge 1\), \(\varepsilon > 0\), and \(\sqrt{M} \gtrsim N^{\frac{p-1}{2}-\frac{n}{2(n+1)}}\), there exists \(T_0 > \frac{1}{(n+2)\gamma_0} N^{-\frac{1}{2(n+1)}}\) such that for all \(T \geq T_0\) and sufficiently large \(N\),
\[
\inf_{\boldsymbol{X} \colon m_{i_1^\ast j_1^\ast}(\boldsymbol{X}) \geq \varepsilon_N} \mathbb{P}_{\boldsymbol{X}} \left (\inf_{t \in [T_0,T]} m_{i_1^\ast j_1^\ast}(\boldsymbol{X}_t) \geq 1-\varepsilon \right ) \ge 1 - \exp(-K N).
\]
\end{lem}

The proof of Lemma~\ref{lem: weak implies strong recovery p>2} follows the same strategy as~\cite[Lemma 5.2]{benarous2024langevin}, where a similar result is established for Langevin dynamics. Proposition~\ref{thm: strong recovery first spike GF p>2} then follows by combining Lemmas~\ref{lem: weak recovery first spike GF p>2} and~\ref{lem: weak implies strong recovery p>2}, using the semigroup property of the flow. This mirrors the approach taken in~\cite[Proposition 3.5]{benarous2024langevin}, where the strong Markov property is applied in the presence of Brownian noise.

We now proceed to the proof of Lemma~\ref{lem: weak recovery first spike GF p>2}.

\begin{proof}[\textbf{Proof of Lemma~\ref{lem: weak recovery first spike GF p>2}}]
Let \(\mathcal{A} = \mathcal{A}(n, \gamma_0, \gamma_1, \gamma_2, \gamma_3)\) denote the event 
\[
\mathcal{A}(n, \gamma_0,\gamma_1, \gamma_2, \gamma_3) = \left \{\boldsymbol{X}_0 \sim \mu_0^{(N)} \colon \boldsymbol{X}_0 \in \mathcal{C}_0^{(N)}(n,\gamma_0) \cap \mathcal{C}_1^{(N)}(\gamma_1,\gamma_2) \cap  \mathcal{C}_2^{(N)}(\gamma_1,\gamma_3) \right \}.
\]
On the event \(\mathcal{C}_1^{(N)}(\gamma_1,\gamma_2)\), for every \(i,j \in [r]\), there exists \(\gamma_{ij} \in (\gamma_2,\gamma_1)\) such that 
\[
m_{ij}(\boldsymbol{X}_0) = \gamma_{ij}N^{-\frac{1}{2}}. 
\]
According to Definition~\ref{def: greedy operation}, we can write 
\[
\lambda_{i_1^\ast} \lambda_{j_1^\ast} \gamma_{i_1^\ast j_1^\ast}^{p-2} = \max_{1 \leq i,j \leq r} \{\lambda_i \lambda_j \gamma_{ij}^{p-2}\}.
\]
Furthermore, under the event \(\mathcal{C}_2^{(N)}(\gamma_1, \gamma_3)\), we obtain the strict inequality
\[
\lambda_{i_1^\ast} \lambda_{j_1^\ast} \gamma_{i_1^\ast j_1^\ast}^{p-2} > \left (1 + \frac{\gamma_3}{\gamma_1} \right ) \lambda_i \lambda_j \gamma_{ij}^{p-2},
\]
for all \((i,j) \neq (i_1^\ast,j_1^\ast)\). We now introduce constants \(\delta_{ij} \in (0,1)\) such that
\begin{equation} \label{eq:def_delta_ij}
\lambda_{i_1^\ast} \lambda_{j_1^\ast} \gamma_{i_1^\ast j_1^\ast}^{p-2}
= \frac{1}{\delta_{ij}} \left(1 + \frac{\gamma_3}{\gamma_1} \right) \lambda_i \lambda_j \gamma_{ij}^{p-2}.
\end{equation}

Next, for every \(i,j \in [r]\), let \(\mathcal{T}_{L_0}^{(ij)}\) denote the hitting time of the set
\[
\left \{\boldsymbol{X} \colon |L_0  m_{ij}(\boldsymbol{X})| > C_0 \sqrt{M} p\lambda_i \lambda_j m_{ij}^{p-1}(\boldsymbol{X}) \right \},
\]
where \(C_0 \in (0,\frac{1}{2})\) is a constant independent of \(N\). Note that on the event \(\mathcal{A}\)---and in particular on \(\mathcal{C}_0^{(N)}(n, \gamma_0)\)---we have
\[
|L_0  m_{ij}(\boldsymbol{X}_0)| \leq \frac{\gamma_0}{\sqrt{N}} \leq C_0 \sqrt{M} p \lambda_i \lambda_j \left (\frac{\gamma_2}{\sqrt{N}} \right)^{p-1} \le C_0 \sqrt{M} p \lambda_i \lambda_j m_{ij}^{p-1}(\boldsymbol{X}_0),
\]
provided that \(\sqrt{M} \geq \frac{\gamma_0}{C_0 p \lambda_i \lambda_j \gamma_2^{p-1}} N^{\frac{p-2}{2}}\), which holds by assumption. Therefore, by continuity of the flow \(\boldsymbol{X}_t\), we conclude that \(\mathcal{T}_{L_0}^{(ij)} > 0\) on the event \(\mathcal{A}\). We also define the hitting time \(\mathcal{T}_{L_0}\) of the set 
\[
\left \{ \boldsymbol{X} \colon \sup_{1 \leq k, \ell \leq r}\left | L_0  m_{k \ell}(\boldsymbol{X})\right | > C_0 \sqrt{M} p \lambda_{i_1^\ast}  \lambda_{j_1^\ast} m_{i_1^\ast j_1^\ast}^{p-1}(\boldsymbol{X})\right \}.
\]
It follows again by continuity that \(\mathcal{T}_{L_0} >0\), and by construction we have \(\mathcal{T}_{L_0} \leq \mathcal{T}_{L_0}^{(i_1^\ast j_1^\ast)}\).

We now fix \(i,j \in [r]\) and work under the event \(\mathcal{A}\). We introduce a first microscopic threshold $\tilde{\varepsilon}_N = \tilde \gamma N^{-\frac{1}{2}}$, where $\tilde{\gamma} >  \gamma_1$ is a constant to be determined later. Let $\mathcal{T}^{(ij)}_{\tilde{\varepsilon}_N}$ denote the hitting time of the set \(\{\boldsymbol{X} \colon m_{ij}(\boldsymbol{X}) \ge \tilde{\varepsilon}_N\}\). Since $\tilde{\gamma} >  \gamma_1$, it follows immediately that $\min_{1 \leq i,j \leq r} \mathcal{T}_{\tilde{\varepsilon}_N}^{(ij)}>0$. From Lemma~\ref{lem: evolution equation m_ij}, we have
\[
L m_{ij} = L_0  m_{ij} + \sqrt{M} p \lambda_i \lambda_j m_{ij}^{p-1} - \sqrt{M}  \frac{p}{2} \sum_{1 \leq k,\ell \leq r} \lambda_k m_{i \ell}m_{kj} m_{k \ell} (\lambda_j m_{kj}^{p-2} + \lambda_\ell m_{k \ell}^{p-2}).
\]
As a consequence, for every \(t \leq \mathcal{T}_{L_0}^{(ij)} \wedge \mathcal{T}_{L_0} \wedge \min_{1 \leq k, \ell \leq r} \mathcal{T}_{\tilde{\varepsilon}_N}^{(k \ell)}\), we obtain the comparison bounds
\begin{equation} \label{eq:comp_gen}
(1-C_0) \sqrt{M} p \lambda_i \lambda_j m_{ij}^{p-1}(t) \leq L m_{ij}(t) \leq (1+ C_0) \sqrt{M} p \lambda_i \lambda_j m_{ij}^{p-1}(t),
\end{equation}
provided that
\begin{equation} \label{eq: condition on N}
N \ge \frac{r^2\lambda_1^2 \tilde{\gamma}^{p+1}}{C_0 \lambda_r^2 \gamma_2^{p-1}}. 
\end{equation}
Since the evolution of \(m_{ij}\) under gradient flow satisfies
\[
m_{ij}(t) = m_{ij}(0) + \int_0^t Lm_{ij}(s) \mathrm{d}s,
\]
we obtain the integral inequality 
\begin{equation} \label{eq: integral inequality GF 1}
\frac{\gamma_{ij}}{\sqrt{N}} +  (1-C_0) \sqrt{M} p \lambda_i \lambda_j \int_0^t m_{ij}^{p-1}(s) \mathrm{d}s \leq m_{ij}(t) \leq \frac{\gamma_{ij}}{\sqrt{N}} +  (1+ C_0)  \sqrt{M} p \lambda_i \lambda_j \int_0^t m_{ij}^{p-1}(s) \mathrm{d}s,
\end{equation}
for every \(t \leq \mathcal{T}_{L_0}^{(ij)} \wedge \mathcal{T}_{L_0} \wedge \min_{1 \leq k, \ell \leq r} \mathcal{T}_{\tilde{\varepsilon}_N}^{(k \ell)}\). Applying items (a) and (b) of Lemma~\ref{lem: Gronwall}, we obtain the comparison inequality 
\begin{equation} \label{eq: comparison inequality GF 1}
\ell_{ij}(t) \leq m_{ij}(t) \leq u_{ij}(t),
\end{equation}
for all \(t\) in the same time interval, where the lower and upper envelope functions are given by
\begin{equation} \label{eq: function ell GF 1}
\ell_{ij}(t) = \frac{\gamma_{ij}}{\sqrt{N}} \left ( 1 -(1 - C_0) \sqrt{M} p(p-2)  \lambda_i \lambda_j \left ( \frac{\gamma_{ij}}{\sqrt{N}} \right )^{p-2} t \right )^{-\frac{1}{p-2}},
\end{equation}
and 
\begin{equation} \label{eq: function u GF 1}
u_{ij}(t) = \frac{\gamma_{ij}}{\sqrt{N}} \left (1 - (1 + C_0) \sqrt{M} p(p-2)  \lambda_i \lambda_j \ \left ( \frac{\gamma_{ij}}{\sqrt{N}} \right )^{p-2} t \right )^{-\frac{1}{p-2}},
\end{equation}
respectively. We now define \(T_{\ell,\tilde{\varepsilon}_N}^{(ij)}\) as the time at which the lower bound \(\ell_{ij}(t)\) reaches the threshold \(\tilde{\varepsilon}_N\), i.e., 
\begin{equation}  \label{eq: lower time GF 1}
T_{\ell, \tilde{\varepsilon}_N}^{(ij)} = \frac{1 - \left(\frac{\gamma_{ij}}{\tilde{\gamma}}\right)^{p-2}}{(1 - C_0) \sqrt{M} p(p-2)\lambda_i \lambda_j \left ( \frac{\gamma_{ij}}{\sqrt{N}}\right)^{p-2} } .
\end{equation}
Similarly, define \(T_{u, \tilde{\varepsilon}_N}^{(ij)}\) by the condition \(u_{ij}(T_{u, \tilde{\varepsilon}_N}^{(ij)}) = \tilde{\varepsilon}_N\), i.e.,
\begin{equation} \label{eq: upper time GF 1}
T_{u, \tilde{\varepsilon}_N}^{(ij)} =
\frac{1 - \left( \frac{\gamma_{ij}}{\tilde{\gamma}} \right)^{p-2}}
{(1 + C_0) \sqrt{M} p(p - 2) \lambda_i \lambda_j \left( \frac{\gamma_{ij}}{\sqrt{N}} \right)^{p - 2}}.
\end{equation}
Due to the scaling of \(\sqrt{M}\), both \(T_{\ell, \tilde{\varepsilon}_N}^{(ij)}\) and \(T_{u, \tilde{\varepsilon}_N}^{(ij)}\) are strictly less than one. Moreover, on the event \(\mathcal{A}\), the hitting time \(\mathcal{T}_{\tilde{\varepsilon}_N}^{(ij)}\) satisfies
\[
T_{u, \tilde{\varepsilon}_N}^{(ij)} \le \mathcal{T}_{\tilde{\varepsilon}_N}^{(ij)} \le T_{\ell, \tilde{\varepsilon}_N}^{(ij)}.
\]

Our goal is thus to show that \(\min_{1 \le i,j \le r} \mathcal{T}_{\tilde \varepsilon_N}^{(ij)} \leq \mathcal{T}_{L_0}\) and that \(\min_{1 \le i,j \le r} \mathcal{T}_{\tilde \varepsilon_N}^{(ij)} = \mathcal{T}^{(i^\ast_{1}j^\ast_{1})}_{\tilde{\varepsilon}_N}\), noting that \(\mathcal{T}_{L_0} \le \mathcal{T}_{L_0}^{(i_1^\ast j_1^\ast)}\) by definition. Choose \(\tilde{\gamma} >0\) such that
\begin{equation} \label{eq: tilde gamma GF}
\frac{1}{\delta} \geq \frac{\tilde{\gamma}^{p-2}}{\tilde{\gamma}^{p-2} - \gamma_1^{p-2}} \enspace \iff \enspace \tilde{\gamma} \geq  \left (\frac{1}{1-\delta} \right )^{\frac{1}{p-2}} \gamma_1,
\end{equation}
where \(\delta = \max_{(i,j) \neq (i_1^\ast, j_1^\ast)} \delta_{ij} \in (0,1)\), and \(\delta_{ij}\) is defined in~\eqref{eq:def_delta_ij}. Then, for every \((i,j) \ne (i_1^\ast, j_1^\ast)\), we compare the respective hitting times:
\[
\begin{split}
T_{\ell, \tilde{\varepsilon}_N}^{(i_1^\ast j_1^\ast)} & = \frac{1 - \left (\frac{\gamma_{i_1^\ast j_1^\ast}}{\tilde{\gamma}} \right )^{p-2}}{(1-C_0) \sqrt{M} p(p-2) \lambda_{i_1^\ast} \lambda_{j_1^\ast} \left ( \frac{\gamma_{i_1^\ast j_1^\ast}}{\sqrt{N}} \right)^{p-2}}   \\
& \le \frac{1}{\frac{1}{\delta_{ij}} (1+C_0) \sqrt{M} p(p-2) \lambda_i \lambda_j \left ( \frac{\gamma_{ij}}{\sqrt{N}} \right)^{p-2}}\\
& \leq \frac{1 - \left (\frac{\gamma_{ij}}{\tilde{\gamma}} \right )^{p-2}}{(1 + C_0) \sqrt{M} p(p-2) \lambda_i \lambda_j \left ( \frac{\gamma_{ij}}{\sqrt{N}} \right)^{p-2}}  = T_{u, \tilde{\varepsilon}_N}^{(ij)},
\end{split}
\]
where the first inequality follows from~\eqref{eq:def_delta_ij}, provided
\[
C_0 \le \frac{\gamma_3 / \gamma_1}{2 + \gamma_3 / \gamma_1} ,
\]
which holds by assumption since \(C_0 \le \frac{\gamma_3 / \gamma_1}{2 (1+ \gamma_3 / \gamma_1)}\), 
and the second inequality follows from~\eqref{eq: tilde gamma GF}. This shows that
\[
T_{\ell, \tilde{\varepsilon}_N}^{(i_1^\ast j_1^\ast)} \le T_{u, \tilde{\varepsilon}_N}^{(ij)}
\quad \text{for all } (i,j) \ne (i_1^\ast, j_1^\ast),
\]
and by monotonicity of the dynamics, it follows that
\[
\mathcal{T}_{\tilde{\varepsilon}_N}^{(i_1^\ast j_1^\ast)}
= \min_{1 \le k,\ell \le r} \mathcal{T}_{\tilde{\varepsilon}_N}^{(k\ell)},
\]
as soon as we can show
\[
\min_{1 \le k,\ell \le r} \mathcal{T}_{\tilde{\varepsilon}_N}^{(k\ell)}
\le \min_{1 \le k,\ell \le r} \mathcal{T}_{L_0}^{(k\ell)} \wedge \mathcal{T}_{L_0}.
\]

To achieve this, we seek an estimate for \(L_0 m_{ij}\) for every \(i,j \in [r]\). To this end, we apply Lemma~\ref{lem: bounding flows} to the function \(F_{ij}(\boldsymbol{X}) = L_0  m_{ij} (\boldsymbol{X})\). We see that if we let \(\psi_{k \ell}(\boldsymbol{X}) = \langle \boldsymbol{v}_k, \boldsymbol{x}_\ell \rangle\), \(a_{k \ell}(\boldsymbol{X}) = \sqrt{M} p\lambda_k \lambda_\ell m_{k \ell}^{p-1}(\boldsymbol{X})\) and \(U = H_0\), then condition (1) is satisfied with \(\mathbb{P}\)-probability at least \(1 - \exp(-KN)\) for every \(n \geq 1\) according to Lemma~\ref{lem: regularity H0}. The function \(F_{ij}\) is smooth and for every \(n \geq 1\) satisfies \(\norm{F_{ij}}_{\mathcal{G}^{2n}} \leq \Lambda\) with \(\mathbb{P}\)-probability at least \(1 - \exp(-K N)\) according to~\eqref{eq: bound norm L_0m_ij}, thus condition (2) is verified. Condition (3) follows by assumption on the initial data, i.e., the event \(\mathcal{C}_0^{(N)}(n,\gamma_0)\). We now verify condition (4). Fix \(k, \ell \in [r]\). Using the lower bound from the integral inequality~\eqref{eq: integral inequality GF 1}, we have
\begin{equation} \label{eq: bound a_ij GF}
\int_0^t |a_{k \ell}(s)| \mathrm{d} s \leq \frac{1}{1 - C_0} \left (m_{k \ell}(t)-\frac{\gamma_{k \ell}}{\sqrt{N}}\right) \leq \frac{1}{1-C_0} m_{k \ell}(t),
\end{equation}
for every \(t \leq \mathcal{T}_{L_0}^{(k \ell)} \wedge \mathcal{T}_{L_0} \wedge \min_{1 \le i,j \le r} \mathcal{T}_{\tilde{\varepsilon}_N}^{(ij)}\). We observe that at time \(t=0\), for every \(\xi > 0\), we have 
\[
\xi \sqrt{M} p \lambda_k \lambda_\ell \left (\ell_{k \ell}(0)\right)^{p-1} = \xi \sqrt{M} p \lambda_k \lambda_\ell \left ( \frac{\gamma_{k \ell}}{\sqrt{N}} \right)^{p-1} \geq C \xi (n+2) \gamma_0 N^{- \frac{n}{2(n+1)}} \geq \ell_{k \ell}(0),
\]
where we used the assumption \(\sqrt{M} \ge C\frac{(n+2) \gamma_0}{p \lambda_r^2 C_0 \gamma_2^{p-1}} N^{\frac{p-1}{2} - \frac{n}{2(n+1)}}\). Now, from~\eqref{eq: comparison inequality GF 1}, we know that \(m_{k\ell}(t) \ge \ell_{k\ell}(t)\) over the time interval of interest. Since \(\ell_{k \ell}(t)\) is increasing and satisfies the above inequality at \(t=0\), it follows that 
\[
m_{k \ell}(t) \leq \xi \sqrt{M} p \lambda_k \lambda_\ell m_{k \ell}^{p-1}(t),
\]
and therefore, combining with~\eqref{eq: bound a_ij GF}, we obtain
\[
\int_0^t |a_{k \ell}(s)| \mathrm{d} s \leq \frac{1}{1- C_0} m_{k \ell}(t) \leq \frac{\xi}{1-C_0} \sqrt{M} p \lambda_k \lambda_\ell m_{k \ell}^{p-1}(t) ,
\]
for every \(t \leq \mathcal{T}_{L_0}^{(k \ell)} \wedge \mathcal{T}_{L_0} \wedge \min_{1 \le i,j \le r} \mathcal{T}_{\tilde{\varepsilon}_N}^{(ij)}\). Choosing \(\xi = (1 - C_0)/2\) yields condition (4) with \(\epsilon = 1/2\). Thus, by Lemma~\ref{lem: bounding flows}, there exists a constant \(K_1 > 0\) such that on the event \(\mathcal{A}\),
\begin{equation} \label{eq: bounding flows GF}
|L_0 m_{ij}(t)| \leq K_1 \left(\frac{\gamma_0}{\sqrt{N}} \sum_{k=0}^{n-1} t^k  + t^n + 2 \sum_{1 \leq k, \ell \leq r} \int_0^t |a_{k \ell}(s)| \mathrm{d}s \right),
\end{equation}
for every \(t \leq \min_{1 \le k, \ell \le r} \mathcal{T}_{L_0}^{(k \ell)} \wedge \mathcal{T}_{L_0} \wedge \min_{1 \le k, \ell \le r} \mathcal{T}_{\tilde{\varepsilon}_N}^{(k \ell)}\), with \(\mathbb{P}\)-probability at least \(1 - \exp(-KN)\). To conclude this step, we will show that, over the same time interval,
\[
\sup_{1 \leq i,j \leq r} \vert L_0 m_{ij}(t) \vert \leq C_0 \sqrt{M} p \inf_{1 \leq i, j \leq r} \lambda_{i} \lambda_{j} m_{i j}^{p-1} (t).
\]
A sufficient condition for this is to show that each term on the right-hand side of~\eqref{eq: bounding flows GF} is bounded above by \(\frac{C_0 \sqrt{M} p}{n+2} \inf_{1 \leq k, \ell \leq r} \lambda_k \lambda_\ell m_{k \ell}^{p-1} (t)\) for all \(t \leq \min_{1 \le k, \ell \le r} \mathcal{T}_{L_0}^{(k \ell)} \wedge \mathcal{T}_{L_0} \wedge \min_{1 \le k, \ell \le r} \mathcal{T}_{\tilde{\varepsilon}_N}^{(k \ell)} \wedge 1\). We verify this term by term:
\begin{itemize}
\item[(i)] For all \(1 \leq i,j \leq r\), the lower bound in~\eqref{eq: comparison inequality GF 1} implies
\[
\frac{C_0 \sqrt{M} p \lambda_i \lambda_j}{n+2} m_{ij}^{p-1}(t) \geq \frac{C_0 \sqrt{M}  p \lambda_i \lambda_j}{n+2} \ell_{ij}^{p-1}(t) \geq \frac{C_0 \sqrt{M}  p \lambda_i \lambda_j}{n+2} \ell_{ij}^{p-1}(0).
\]
Hence, for every \(0 \leq k \leq n-1\), 
\[
\frac{C_0 \sqrt{M} p \lambda_i \lambda_j}{n+2} \left (\frac{\gamma_{ij}}{2 \sqrt{N}} \right)^{p-1} \geq C  \frac{\gamma_0}{N^{\frac{n}{2(n+1)}}} \geq K \frac{\gamma_0}{\sqrt{N}} t^k,
\]
for \(t \leq \min_{1 \le k,\ell \le r}\mathcal{T}_{L_0}^{(k \ell)}  \wedge \mathcal{T}_{L_0} \wedge \min_{1 \leq k, \ell \leq r} \mathcal{T}_{\tilde{\varepsilon}_N}^{(k \ell)} \wedge 1\). 

\item[(ii)] A sufficient condition to control the second term is given by \(F(t) \leq G(t)\), where \(F(t) = Kt^n\) and \(G(t)= \frac{C_0 \sqrt{M} p \lambda_i \lambda_j}{n+2} \ell_{ij}^{p-1}(t)\). To compare these, compute the derivatives: for any \(k \leq n\),
\[
F^{(k)}(t) = K n(n-1) \cdots (n-k+1)t^{n-k},
\]
and
\[
G^{(k)}(t) = \frac{C_0 \sqrt{M} p \lambda_i \lambda_j \prod_{i=1}^k \left ( \frac{p-1}{p-2} + (i-1) \right)}{n+2}\left(\frac{\gamma_{ij}}{2 \sqrt{N}}\right)^{p-1} \left(\frac{1}{t_\ast^{(ij)}}\right)^k \left(1-\frac{t}{t_\ast^{(ij)}}\right)^{- \left(\frac{p-1}{p-2} + k\right)},
\]
where \(t_\ast^{(ij)}\) denotes the blow-up time of \(\ell_{ij}\) which is given by
\[
t_\ast^{(ij)} \coloneqq \left [ (1 - C_0) \sqrt{M} p(p-2)  \lambda_i \lambda_j \left ( \frac{\gamma_{ij}}{2\sqrt{N}} \right )^{p-2} \right]^{-1}.
\]
For \(k \leq n-1\), we have \(G^{(k)}(0) \geq 0 = F^{(k)}(0)\). For \(k=n\), we obtain the lower bound
\[
\begin{split}
G^{(n)}(t) & \geq \frac{(\sqrt{M} p \lambda_i \lambda_j)^{n+1} C_0 (1-C_0)^n}{n+2} \left ( \frac{\gamma_{ij}}{2 \sqrt{N}}\right )^{p-1+n(p-2)} \left(1-\frac{t}{t_\ast^{(ij)}}\right)^{- \left(\frac{p-1}{p-2} + n\right)} \\
& \gtrsim C_0 (1-C_0)^n (n+2)^n \gamma_0^{n+1}\\
& \geq K n! = F^{(n)}(t),
\end{split}
\]
which holds for all $t \leq \min_{1 \le k,\ell \le r}\mathcal{T}_{L_0}^{(k \ell)}  \wedge \mathcal{T}_{L_0} \wedge \min_{1 \leq k, \ell \leq r} \mathcal{T}_{\tilde{\varepsilon}_N}^{(k \ell)} \wedge 1$. 

\item[(iii)] We control the last term as follows. According to the integral inequality~\eqref{eq: integral inequality GF 1}, on the event \(\mathcal{A}\), we have
\[
2 \sum_{1 \le k ,\ell \le r} \int_0^t |a_{k \ell}(s)| \mathrm{d} s \leq \frac{2r^2}{1-C_0} \max_{1 \le k, \ell \le r} m_{k \ell}(t) \leq \frac{2 r^2}{1-C_0} \tilde{\varepsilon}_N =  \frac{2 r^2}{1-C_0}  \frac{\tilde{\gamma}}{\sqrt{N}},
\]
for all \(t \leq \min_{1 \le k, \ell \le r} \mathcal{T}_{L_0}^{(k \ell)} \wedge \mathcal{T}_{L_0}  \wedge \min_{1 \leq k, \ell \leq r} \mathcal{T}_{\tilde{\varepsilon}_N}^{(k \ell)}\). From the lower bound in~\eqref{eq: comparison inequality GF 1}, we also have
\[
\frac{C_0  \sqrt{M} p \lambda_i \lambda_j}{n+2} m_{ij}^{p-1}(t) \geq \frac{C_0 \sqrt{M} p \lambda_i \lambda_j}{n+2} \ell_{ij}^{p-1}(t) \geq \frac{C_0 \sqrt{M} p \lambda_i \lambda_j}{n+2} \ell_{ij}^{p-1}(0),
\]
for all \(t \leq \min_{1 \le k, \ell \le r} \mathcal{T}_{L_0}^{(k \ell)}  \wedge \mathcal{T}_{L_0} \wedge \min_{1 \le k, \ell \le r} \mathcal{T}_{\tilde{\varepsilon}_N}^{(k \ell)} \wedge 1\). Using the assumption on \(\sqrt{M}\), it follows that
\[
\frac{C_0 \sqrt{M} p \lambda_i \lambda_j}{n+2} \left( \frac{\gamma_{ij}}{2\sqrt{N}} \right)^{p-1}
\ge C \frac{\gamma_0}{(1 - C_0) N^{\frac{n}{2(n+1)}}}
\ge K \frac{2r^2 \tilde{\gamma}}{(1 - C_0)\sqrt{N}},
\]
and thus the integral term is also bounded appropriately.
\end{itemize}
On the event \(\mathcal{A}\), all terms in~\eqref{eq: bounding flows GF} are controlled as desired. Hence,
\[
\min_{1 \le k,\ell \le r} \mathcal{T}_{\tilde{\varepsilon}_N}^{(k\ell)} \le \min_{k,\ell} \mathcal{T}_{L_0}^{(k\ell)} \wedge \mathcal{T}_{L_0},
\]
which implies that
\[
\mathcal{T}_{\tilde{\varepsilon}_N}^{(i_1^\ast j_1^\ast)} = \min_{1 \le k, \ell \le r} \mathcal{T}_{\tilde{\varepsilon}_N}^{(k\ell)},
\]
with \(\mathbb{P}\)-probability at least \(1 - \exp(-K N)\). That is, the correlation \(m_{i_1^\ast j_1^\ast}\) is the first to reach the microscopic threshold \(\tilde{\varepsilon}_N\).

We now show that \(m_{i_1^\ast j_1^\ast}\) remains the dominant correlation and, in particular, reaches the second threshold \(\varepsilon_N\) before any other correlation. From Lemma~\ref{lem: evolution equation m_ij}, we observe that at time \(t = \mathcal{T}_{\tilde{\varepsilon}_N}^{(i_1^\ast j_1^\ast)}\), 
\[
L m_{i_1^\ast j_1^\ast}(t) \geq (1-C_0) \sqrt{M} p\lambda_{i_1^\ast} \lambda_{j_1\ast} m_{i_1^\ast j_1^\ast}^{p-1}(t) = (1-C_0) \sqrt{M} p\lambda_{i_1^\ast} \lambda_{j_1\ast}  \left (\frac{\tilde \gamma}{\sqrt{N}} \right)^{p-1},
\]
and for every \((i,j) \neq (i_1^\ast,j_1^\ast)\),
\[
\begin{split}
L m_{ij}(t) & \leq C_0 \sqrt{M} p \lambda_{i_1^\ast} \lambda_{j_1^\ast} m_{i_1^\ast j_1^\ast}^{p-1}(t) + \sqrt{M} p \lambda_i \lambda_j m_{ij}^{p-1}(t) \\
& = C_0 \sqrt{M} p \lambda_{i_1^\ast} \lambda_{j_1^\ast}  \left (\frac{\tilde \gamma}{\sqrt{N}} \right)^{p-1}+ \sqrt{M} p \lambda_i \lambda_j m_{ij}^{p-1}(t).
\end{split}
\]
For \((i,j) \ne (i_1^\ast, j_1^\ast)\), we upper bound
\[
m_{ij}(t) \le u_{ij}(T_{\ell, \tilde{\varepsilon}_N}^{(i_1^\ast j_1^\ast)})
= \frac{\gamma_{ij}}{\sqrt{N}}  \frac{1}{(1 - \delta_{ij})^{\frac{1}{p-2}}},
\]
so that
\begin{align*}
L m_{ij}(t)
&\le C_0 \sqrt{M} p \lambda_{i_1^\ast} \lambda_{j_1^\ast} \left( \frac{\tilde{\gamma}}{\sqrt{N}} \right)^{p-1}
+ \sqrt{M} p \lambda_i \lambda_j \left( \frac{\gamma_{ij}}{\sqrt{N}} \right)^{p-1}  \frac{1}{(1 - \delta_{ij})^{\frac{p-1}{p-2}}} \\
& < C_0 \sqrt{M} p \lambda_{i_1^\ast} \lambda_{j_1^\ast} \left( \frac{\tilde{\gamma}}{\sqrt{N}} \right)^{p-1}
+ \frac{1}{1 + \gamma_3 / \gamma_1} \sqrt{M} p \lambda_{i_1^\ast} \lambda_{j_1^\ast}  \frac{\gamma_{i_1^\ast j_1^\ast}^{p-2} \gamma_{ij}}{(1 - \delta_{ij})^{\frac{p-1}{p-2}}} \left( \frac{1}{\sqrt{N}} \right)^{p-1}.
\end{align*}
We now recall~\eqref{eq: tilde gamma GF}, which ensures that \(\frac{1}{1 - \delta_{ij}} \le \frac{\tilde \gamma^{p-2}}{\gamma_1^{p-2}}\), and so
\[
\frac{\gamma_{i_1^\ast j_1^\ast}^{p-2} \gamma_{ij}}{(1 - \delta_{ij})^{\frac{p-1}{p-2}}} \le \frac{\gamma_{i_1^\ast j_1^\ast}^{p-2} \gamma_{ij}}{\gamma_1^{p-1}} \tilde{\gamma}^{p-1} < \tilde{\gamma}^{p-1}.
\]
Combining all bounds, we obtain
\[
L m_{ij}(t) < \left( C_0 + \frac{1}{1 + \gamma_3/\gamma_1} \right) \sqrt{M} p \lambda_{i_1^\ast} \lambda_{j_1^\ast} \left( \frac{\tilde{\gamma}}{\sqrt{N}} \right)^{p-1}
< L m_{i_1^\ast j_1^\ast}(t),
\]
where the last inequality uses that
\[
\frac{1}{1 + \gamma_3 / \gamma_1}  < 1 - 2C_0 \quad \iff \quad  C_0 < \frac{\gamma_3/\gamma_1}{2(1 + \gamma_3 / \gamma_1)}. 
\]
Therefore, since \(m_{i_1^\ast j_1^\ast}(t) > m_{ij}(t)\) and \(L m_{i_1^\ast j_1^\ast}(t) \geq L m_{ij}(t)\) at \(t=\mathcal{T}_{\tilde{\varepsilon}_N}^{(i_1^\ast j_1^\ast)}\), we obtain that \(m_{i_1^\ast j_1^\ast}(t) > m_{ij}(t)\) for all \(\mathcal{T}_{\tilde{\varepsilon}_N}^{(i_1^\ast j_1^\ast)} \leq t \leq \min_{1 \le k, \ell \leq r} \mathcal{T}_{L_0}^{(k \ell)} \wedge \mathcal{T}_{L_0} \wedge \min_{1 \le k,\ell \le r} \mathcal{T}_{\varepsilon_N}^{(k \ell)}\), ensuring that 
\[
\mathcal{T}_{\varepsilon_N}^{(i_1^\ast j_1^\ast)} = \min_{1 \le i,j \le r} \mathcal{T}_{\varepsilon_N}^{(i j)},
\]
with \(\mathbb{P}\)-probability at least \(1 - \exp(-KN)\), and the correlation \(m_{i_1^\ast j_1^\ast}\) is the first to reach the threshold \(\varepsilon_N\).

The last step is to show that \(\mathcal{T}_{\varepsilon_N}^{(i_1^\ast j_1^\ast)} \leq \mathcal{T}_{L_0}\) with high \(\mathbb{P}\)-probability. We first note that the bound~\eqref{eq:comp_gen} holds for \(L m_{i_1^\ast j_1^\ast}(t)\) over the time interval 
\[
t \leq \mathcal{T}_{L_0}^{(i_1^\ast j_1^\ast)} \wedge \mathcal{T}_{L_0} \wedge \min_{1 \leq k, \ell \leq r }\mathcal{T}_{\varepsilon_N}^{(k \ell)}, 
\]
provided that \(N^{\frac{p-3}{2(p-1)}} \ge \frac{r^2 C^{p+1}}{C_0}\). As a result, both the integral inequality~\eqref{eq: integral inequality GF 1} and the comparison inequality~\eqref{eq: comparison inequality GF 1} apply to \(m_{i_1^\ast j_1^\ast}(t)\) for all \( t \leq  \min_{1 \le k,\ell \le r} \mathcal{T}_{L_0}^{(k \ell)}  \wedge \mathcal{T}_{L_0} \wedge \mathcal{T}_{\varepsilon_N}^{(i_1^\ast j_1^\ast)}\). Moreover, \(a_{ij}(t)\leq a_{i_1^\ast j_1^\ast}(t)\) for every \((i,j) \neq (i_1^\ast, j_1^\ast)\) and every \(\mathcal{T}_{\tilde{\varepsilon}_N}^{(i_1^\ast j_1^\ast)} \leq t \leq  \min_{1 \le k,\ell \le r} \mathcal{T}_{L_0}^{(k \ell)}  \wedge \mathcal{T}_{L_0} \wedge \mathcal{T}_{\varepsilon_N}^{(i_1^\ast j_1^\ast)}\). Using similar computations as before, condition (4) of Lemma~\ref{lem: bounding flows} is satisfied in a slightly modified form:
\[
\int_0^t |a_{ij}(s)| \mathrm{d} s  \leq \frac{1}{2} |a_{i_1^\ast j_1^\ast}(t)|,
\]
for every \(\mathcal{T}_{\tilde{\varepsilon}_N}^{(i_1^\ast j_1^\ast)} \leq t \leq  \min_{1 \le k,\ell \le r} \mathcal{T}_{L_0}^{(k \ell)}  \wedge \mathcal{T}_{L_0} \wedge \mathcal{T}_{\varepsilon_N}^{(i_1^\ast j_1^\ast)}\). This implies that the estimate~\eqref{eq: bounding flows GF} holds in the following way: on the event \(\mathcal{A}\), 
\begin{equation} \label{eq: bounding flows GF 2}
|L_0 m_{ij}(t)| \leq K \left(\frac{\gamma_0}{\sqrt{N}} \sum_{k=0}^{n-1} t^k +t^n + 2r^2 \int_0^t |a_{i_1^\ast j_1^\ast}(s)| \mathrm{d} s \right),
\end{equation}
for \(\mathcal{T}_{\tilde{\varepsilon}_N}^{(i_1^\ast j_1^\ast)} \leq t \leq  \min_{1 \le k,\ell \le r} \mathcal{T}_{L_0}^{(k \ell)}  \wedge \mathcal{T}_{L_0} \wedge \mathcal{T}_{\varepsilon_N}^{(i_1^\ast j_1^\ast)}\), with \(\mathbb{P}\)-probability at least \(1 - \exp(-K N)\). As before, by the assumption on \(\sqrt{M}\), each term on the right-hand side of~\eqref{eq: bounding flows GF 2} can be bounded above by \(\frac{C_0  \sqrt{M} p \lambda_{i_1^\ast} \lambda_{j_1^\ast}}{n+2} m_{i_1^\ast j_1^\ast}^{p-1}(t)\) for every \(\mathcal{T}_{\tilde{\varepsilon}_N}^{(i_1^\ast j_1^\ast)} \leq t \leq  \min_{1 \le k,\ell \le r} \mathcal{T}_{L_0}^{(k \ell)}  \wedge \mathcal{T}_{L_0} \wedge \mathcal{T}_{\varepsilon_N}^{(i_1^\ast j_1^\ast)}\). This ensures that
\[
\mathcal{T}_{\varepsilon_N}^{(i_1^\ast j_1^\ast)} \le \mathcal{T}_{L_0},
\]
with high \(\mathbb{P}\)-probability. Therefore, on the event \(\mathcal{A}\), we have that \(\mathcal{T}_{\varepsilon_N}^{(i_1^\ast j_1^\ast)}  \le \mathcal{T}_{L_0}\), and we find that
\[
\mathcal{T}_{\varepsilon_N}^{(i_1^\ast j_1^\ast)} \le T_{\ell,\varepsilon_N}^{(i_1^\ast j_1^\ast)} \lesssim \frac{1}{(n+2) \gamma_0 N^{\frac{1}{2(n+1)}}},
\]
with \(\mathbb{P}\)-probability at least \(1 - \exp(-KN)\). This completes the proof of Lemma~\ref{lem: weak recovery first spike GF p>2}.
\end{proof}

\subsection{Recovery of all spikes (up to a permutation)}

We now prove Proposition~\ref{thm: strong recovery all spikes GF p>2} on the recovery of a permutation of all spikes. The argument follows the proof of~\cite[Proposition 3.6]{benarous2024langevin} in the Langevin dynamics setting. Accordingly, we highlight only the key elements that differ from the Langevin case and refer the reader to~\cite{benarous2024langevin} for the overlapping parts.

The proof proceeds through \(r\) steps, each focusing on the strong recovery of a new correlation \(m_{i_k^\ast j_k^\ast}\). For every \(\varepsilon > 0\), define the following sequence of events:
\[
\begin{split}
E_1 (\varepsilon) & = R_1 (\varepsilon) \cap \left \{ \boldsymbol{X} \colon m_{ij}(\boldsymbol{X}) \in \Theta(N^{-\frac{1}{2}}) \enspace \forall \, i \neq i_1^\ast , j\neq j_1^\ast \right \},\\
E_2 (\varepsilon) & = R_1(\varepsilon) \cap R_2(\varepsilon) \cap \left \{ \boldsymbol{X} \colon m_{ij}(\boldsymbol{X}) \in \Theta(N^{-\frac{1}{2}}) \enspace \text{for} \enspace i \neq i_1^\ast, i_2^\ast j\neq j_1^\ast, j_2^\ast \right \},\\
& \vdots \\
E_{r-1} (\varepsilon) & = \cap_{1 \leq i \leq r -1} R_i(\varepsilon) \cap \left \{ \boldsymbol{X} \colon m_{i_r^\ast j_r^\ast}(\boldsymbol{X}) \in \Theta(N^{-\frac{1}{2}})\right \},\\
E_r (\varepsilon) & = \cap_{1 \leq i \leq r -1} R_i(\varepsilon) \cap \left \{ \boldsymbol{X} \colon m_{i_r^\ast j_r^\ast}(\boldsymbol{X}) \ge 1-\varepsilon \right \},
\end{split}
\]
where \(R_k(\varepsilon)\) denotes the event of strong recovery of the \(k\)th spike in the permutation, i.e.,
\[
R_k(\varepsilon) = \left \{ \boldsymbol{X} \colon m_{i_k^\ast j_k^\ast}(\boldsymbol{X}) \geq 1-\varepsilon \enspace \textnormal{and} \enspace m_{i_k^\ast j}(\boldsymbol{X}),m_{i j_k^\ast}(\boldsymbol{X}) \lesssim \log(N)^{-\frac{1}{2}}N^{-\frac{p-1}{4}} \enspace \forall \, i\neq i_k^\ast, j \neq j_k^\ast \right \}.
\]
Here, the symbol \(\lesssim\) hides an absolute constant. We note that the final event \(E_r(\varepsilon)\) coincides with \(R(\varepsilon)\), as defined in~\eqref{eq: set strong recovery GF p>2}. Moreover, we note that, once a correlation \(m_{i_k^\ast j_k^\ast}\) reaches a macroscopic threshold \(\varepsilon\), all correlations \(m_{i_k^\ast j}\) and \(m_{i j_k^\ast}\) for \(i \neq i_k^\ast, j \neq j_k^\ast\) decrease below \(\log(N)^{-\frac{1}{2}}N^{-\frac{p-1}{4}}\). This is crucial to ensure the recovery of the subsequent correlation \(m_{i_{k+1}^\ast j_{k+1}^\ast}\).

The next lemma quantifies the sample complexity and time required to attain the event \(E_1(\varepsilon)\) from a random initialization that satisfies Condition \(1\) and Condition \(2\). 

\begin{lem} [Recovery of the first spike] \label{lem: E_1 GF}
Consider a sequence of initializations \(\mu_0^{(N)} \in \mathcal{M}_1(\mathcal{S}_{N,r})\). Then, the following holds: For every \(\gamma_1 > \gamma_2 \vee \gamma_3\), \(C_0 \in (0,\frac{\gamma_3/\gamma_1}{2(1+\gamma_3/\gamma_1)})\), and \(\varepsilon > 0\), there exist \(\Lambda = \Lambda(p, \{\lambda_i\}_{i=1}^r)>0\), \(C>0\), and \(K>0\) such that if \(\sqrt{M} \geq C\frac{\Lambda}{p \lambda_r^2 C_0 \gamma_2^{p-1}} N^{\frac{p-1}{2}}\), then for \(N\) sufficiently large,
\[
\int_{\mathcal{S}_{N,r}} \mathbb{P}_{\boldsymbol{X}^+} \left (\mathcal{T}_{E_1} \gtrsim \frac{1}{\sqrt{N}} \right ) \boldsymbol{1}\{\mathcal{C}_1^{(N)}(\gamma_1,\gamma_2) \cap \mathcal{C}_2^{(N)}(\gamma_1,\gamma_3)\} \mathrm{d} \mu_0^{(N)} (\boldsymbol{X}) \leq e ^{-KN}.
\]
\end{lem}

Compared to Lemma~\ref{lem: weak recovery first spike GF p>2}, this result ensures not only the recovery of the leading spike direction, but also suppression of all entries sharing the same row or column index, along with the stability of all other correlations---thereby preparing the system for the next step in the recovery sequence. Once the set \(E_1\) is attained, reaching \(E_2\) follows directly. More generally, assuming that the \((k-1)\)st event \(E_{k-1}(\varepsilon)\) holds, we now show that the system reaches \(E_k (\varepsilon)\) with high probability.

\begin{lem}[Inductive recovery step] \label{lem: E_2 GF}
For every \(\gamma_1 >\gamma_2 \vee \gamma_3\), \(C_0 \in (0,\frac{1}{2})\), and \(\varepsilon > 0\), there exist \(\Lambda = \Lambda(p, \{\lambda_i\}_{i=1}^r)>0\), \(C>0\), and \(K>0\) such that if \(\sqrt{M} \gtrsim \frac{\Lambda}{p \lambda_r^2 C_0 \gamma_2^{p-1}} N^{\frac{p-1}{2}}\), then there exists \(T_k > T_{k-1}\) (with \(T_0 = \mathcal{T}_{E_1}\)) such that for every \(T > T_k\) and \(N\) sufficiently large,
\[
\inf_{\boldsymbol{X}_0 \in E_{k-1}(\varepsilon)} \mathbb{P}_{\boldsymbol{X}_0} \left ( \inf_{t \in [T_k, T]} \boldsymbol{X}_t \in E_k(\varepsilon) \right ) \geq 1 - e^{-KN}.
\]
\end{lem}

Proposition~\ref{thm: strong recovery all spikes GF p>2} then follows by iteratively applying Lemmas~\ref{lem: E_1 GF} and~\ref{lem: E_2 GF}, using the semi-group property of the flow. We direct the reader to the proof of Proposition 3.5 in~\cite{benarous2024langevin} for a proof. 
It remains to prove Lemmas~\ref{lem: E_1 GF} and~\ref{lem: E_2 GF}.

\begin{proof}[\textbf{Proof of Lemma~\ref{lem: E_1 GF}}]
Let \(\mathcal{A} = \mathcal{A}(\gamma_1,\gamma_2,\gamma_3)\) denote the event 
\[
\mathcal{A} (\gamma_1,\gamma_2,\gamma_3) = \left \{ \boldsymbol{X}_0 \sim \mu_0^{(N)} \colon \boldsymbol{X}_0 \in \mathcal{C}_1^{(N)} (\gamma_1,\gamma_2)  \cap \mathcal{C}_2^{(N)} (\gamma_1,\gamma_3)\right \}.
\]
We note that on \(\mathcal{C}_1^{(N)}(\gamma_1,\gamma_2)\), for every \(i,j \in [r]\) there exists \(\gamma_{ij} \in (\gamma_2,\gamma_1)\) such that \(m_{ij}(0) = \gamma_{ij}N^{-\frac{1}{2}}\). In particular, according to Definition~\ref{def: greedy operation}, we have 
\[
\lambda_{i_1^\ast} \lambda_{j_1^\ast} \gamma_{i_1^\ast j_1^\ast}^{p-2} \ge \lambda_{i_2^\ast }\lambda_{j_2^\ast} \gamma_{i_2^\ast j_2^\ast}^{p-2} \geq \cdots \geq \lambda_{i_r^\ast} \lambda_{j_r^\ast} \gamma_{i_r^\ast j_r^\ast}^{p-2}.
\]
Moreover, on the event \(\mathcal{C}_2^{(N)}(\gamma_1,\gamma_3)\), 
\begin{equation} \label{eq: separation cond GF}
\lambda_{i_1^\ast} \lambda_{j_1^\ast} \gamma_{i_1^\ast j_1^\ast}^{p-2} > \left (1 + \frac{\gamma_3}{\gamma_1} \right) \lambda_i \lambda_j \gamma_{ij}^{p-2},
\end{equation}
for every \((i,j) \neq (i_1^\ast,j_1^\ast)\). 

In the following, we fix \(i,j \in [r]\) and place ourselves on the event \(\mathcal{A}\). In a similar fashion as in the proof of Lemma~\ref{lem: weak recovery first spike GF p>2}, we first consider a microscopic threshold $\tilde{\varepsilon}_N = \frac{\tilde{\gamma}}{\sqrt{N}}$ with $\tilde{\gamma} >  \gamma_1$ and show that $m_{i^{\ast}_{1}j^{\ast}_{1}}$ is the first correlation to reach this threshold under the chosen scaling for $\sqrt{M}$. The only difference lies in the fact that, for this threshold value of $\sqrt{M}$, there is no need to use the bounding flow from Lemma~\ref{lem: bounding flows}, and the uniform bound from Lemma~\ref{lem: regularity H0} is sufficient. As this uniform bound will be repeated below, we do not write this first part of the proof explicitly. Thus, we directly move to the threshold \(\varepsilon_N =C N^{-\frac{p-2}{2(p-1)}}\) with \(C>0\). Let \(\mathcal{T}_{\varepsilon_N}^{(ij)}\) denote the hitting time of the set \(\{\boldsymbol{X} \colon m_{ij}(\boldsymbol{X}) \geq \varepsilon_N\}\). According to the generator expansion by Lemma~\ref{lem: evolution equation m_ij}, i.e., 
\[
L m_{ij} = L_0 m_{ij} +  \sqrt{M} p \lambda_i \lambda_j m_{ij}^{p-1} - \sqrt{M} \frac{p}{2} \sum_{1 \leq k,\ell \leq r} \lambda_k m_{i \ell}m_{kj} m_{k \ell} (\lambda_j m_{kj}^{p-2} + \lambda_\ell m_{k \ell}^{p-2}),
\]
we have 
\[
- \norm{L_0 m_{i^{\ast}_{1}j^{\ast}_{1}}}_\infty + \sqrt{M} p \lambda_i \lambda_j m_{i^{\ast}_{1}j^{\ast}_{1}}^{p-1}(t) \leq L m_{i^{\ast}_{1}j^{\ast}_{1}}(t) \leq \norm{L_0 m_{i^{\ast}_{1}j^{\ast}_{1}}}_\infty + \sqrt{M}  p \lambda_i \lambda_j m_{i^{\ast}_{1}j^{\ast}_{1}}^{p-1}(t),
\]
for \(t \leq \min_{1 \leq k, \ell \leq r} \mathcal{T}_{\varepsilon_N}^{(k \ell)}\). Furthermore, for the other correlations, i.e., for $(i,j) \neq (i^{\ast}_{1},j^{\ast}_{1})$, the following upper bound holds:
\[ 
L m_{ij}(t) \leq \norm{L_0 m_{i_{1}j_{1}}}_\infty + \sqrt{M}  p \lambda_i \lambda_j m_{ij}^{p-1}(t).
\]
According to Lemma~\ref{lem: regularity H0} and especially to~\eqref{eq: bound norm L_0m_ij}, we have that \(\norm{L_0 m_{ij}}_\infty \leq \Lambda \) for some constant \(\Lambda = \Lambda(p,n, \{\lambda_i\}_{i=1}^r)\), with \(\mathbb{P}\)-probability at least \(1 - \exp(-K N)\). This implies that for \(t \leq \min_{1 \leq k, \ell \leq r} \mathcal{T}_{\varepsilon_N}^{(k \ell)}\),
\[
C_0 \sqrt{M} p \lambda_i \lambda_j m_{ij}^{p-1}(t) \gtrsim \frac{\Lambda}{\gamma_2^{p-1}} N^{\frac{p-1}{2}} m_{ij}^{p-1}(t) \geq \Lambda \geq \norm{L_0 m_{ij}}_\infty,
\]
for some constant \(C_0 \in (0,1)\), where we used the facts that \(\sqrt{M} \gtrsim \frac{\Lambda}{  p \lambda_r^2 C_0 \gamma_2^{p-1}}N^{\frac{p-1}{2}}\) and that \(m_{ij}(t) \geq \gamma_2N^{-\frac{1}{2}}\). We obtain the integral inequality given by
\[
\begin{split}
m_{ij}(t) & \leq \frac{\gamma_{ij}}{\sqrt{N}} + (1+C_0)   \sqrt{M} p \lambda_i \lambda_j \int_0^t m_{ij}^{p-1}(s)  \mathrm{d} s,\\
m_{i_1^\ast j_1^\ast}(t)& \ge \frac{\gamma_{i_1^\ast j_1^\ast}}{\sqrt{N}} + (1 - C_0)   \sqrt{M} p \lambda_{i_1^\ast} \lambda_{j_1^\ast} \int_0^t m_{i_1^\ast j_1^\ast}^{p-1}(s) \mathrm{d}s ,
\end{split}
\]
for every \(t \leq \min_{1 \leq k, \ell \leq r} \mathcal{T}_{\varepsilon_N}^{(k \ell)}\), with \(\mathbb{P}\)-probability at least \(1 - \exp(-K N)\). Lemma~\ref{lem: Gronwall} then yields the comparison inequality:
\[
\begin{split}
m_{ij}(t) & \leq u_{ij}(t),\\
m_{i_1^\ast j_1^\ast}(t)& \ge \ell_{i_1^\ast j_1^\ast}(t),
\end{split}
\]
for every \(t \leq \min_{1 \leq k, \ell \leq r} \mathcal{T}_{\varepsilon_N}^{(k \ell)}\), where 
\[
u_{ij}(t) = \frac{\gamma_{ij}}{\sqrt{N}} \left ( 1-  (1+C_0)   \sqrt{M} p(p-2) \lambda_i \lambda_j\left ( \frac{\gamma_{ij}}{\sqrt{N}}\right )^{p-2} t \right )^{-\frac{1}{p-2}},
\]
and
\[
\ell_{i_1^\ast j_1^\ast}(t) = \frac{\gamma_{i_1^\ast j_1^\ast}}{\sqrt{N}} \left ( 1- (1-C_0)   \sqrt{M} p(p-2) \lambda_{i_1^\ast} \lambda_{j_1^\ast} \left ( \frac{\gamma_{i_1^\ast j_1^\ast}}{\sqrt{N}}\right )^{p-2} t \right )^{-\frac{1}{p-2}}.
\]
We define \(T_{\ell,\varepsilon_N}^{(i_1^\ast j_1^\ast)}\) to solve \(\ell_{i_1^\ast j_1^\ast}(T_{\ell,\varepsilon_N}^{(i_1^\ast j_1^\ast)}) =\varepsilon_N \), i.e., 
\[ 
T_{\ell,\varepsilon_N}^{(i_1^\ast j_1^\ast)} = \frac{1 - \gamma_{i_1^\ast j_1^\ast}^{p-2} N^{-\frac{p-2}{2(p-1)}}}{ (1-C_0)   \sqrt{M} p(p-2) \lambda_{i_1^\ast} \lambda_{j_1^\ast} \left (\frac{ \gamma_{i_1^\ast j_1^\ast}}{\sqrt{N}}\right )^{p-2}}.
\]
Similarly, for every \(i,j \in [r]\), we let \(T_{u,\varepsilon_N}^{(ij)}\) denote the time such that \(u_{ij}(T_{u,\varepsilon_N}^{(ij)}) = \varepsilon_N\), i.e., 
\begin{equation*} 
T_{u,\varepsilon_N}^{(ij)} = \frac{1 - \gamma_{ij}^{p-2} N^{-\frac{p-2}{2(p-1)}}}{ (1+C_0)   \sqrt{M} p(p-2) \lambda_i \lambda_j \left (\frac{\gamma_{ij}}{\sqrt{N}}\right )^{p-2}}.
\end{equation*}
We observe that for every \(i,j \in [r], (i,j) \neq (i_1^\ast, j_1^\ast)\), 
\[
T_{\ell, \varepsilon_N}^{(i_1^\ast j_1^\ast)} \leq T_{u, \varepsilon_N}^{(ij)},
\]
provided \(N\) sufficiently large and \(C_0 < \frac{\gamma_3/\gamma_1}{2 + \gamma_3/\gamma_1}\). In fact, together with~\eqref{eq: separation cond GF} yields 
\[
(1-C_0) \lambda_{i_1^\ast} \lambda_{j_1^\ast} \gamma_{i_1^\ast j_1^\ast}^{p-2} > \left (1 - \frac{\gamma_3/\gamma_1}{2 + \gamma_3/\gamma_1} \right ) \left (1 + \frac{\gamma_3}{\gamma_1} \right ) \lambda_i \lambda_j \gamma_{ij}^{p-2} > (1+C_0) \lambda_i \lambda_j \gamma_{ij}^{p-2}.
\]
As a consequence, on the event \(\mathcal{A}\), we have that 
\[
\mathcal{T}_{\varepsilon_N}^{(i_1^\ast j_1^\ast)} = \min_{1 \leq k, \ell \leq r} \mathcal{T}_{\varepsilon_N}^{(k \ell)}
\]
with \(\mathbb{P}\)-probability at least \(1-\exp(- K N)\), that is, \(m_{i_1^\ast j_1^\ast}\) is the first correlation that reaches the threshold \(\varepsilon_N\). We therefore have that on the event \(\mathcal{A}\), 
\[
\mathcal{T}_{\varepsilon_N}^{(i_1^\ast j_1^\ast)} \leq T_{\ell, \varepsilon_N}^{(i_1^\ast j_1^\ast)} \lesssim \frac{1}{\sqrt{N}} ,
\]
with \(\mathbb{P}\)-probability at least \(1 -\exp(-K N)\). Furthermore, we observe that as \(m_{i_1^\ast j_1^\ast}(t)\) exceeds \(\varepsilon_N\), the other correlations are still on the scale \(\Theta(N^{-\frac{1}{2}})\). Indeed, since \(\mathcal{T}_{\varepsilon_N}^{(i_1^\ast j_1^\ast)} \leq T_{\ell, \varepsilon_N}^{(i_1^\ast j_1^\ast)}\) and \(u_{ij}\) is a monotone increasing function, on the event \(\mathcal{A}\) we can upper bound \(m_{ij}(\mathcal{T}_{\varepsilon_N}^{(i_1^\ast j_1^\ast)})\) by \(u_{ij}(\mathcal{T}_{\varepsilon_N}^{(i_1^\ast j_1^\ast)}) \leq u_{ij}(T_{\ell, \varepsilon_N}^{(i_1^\ast j_1^\ast)})\) and we find that
\[
\begin{split}
u_{ij}(T_{\ell, \varepsilon_N}^{(i_1^\ast j_1^\ast)}) & = \frac{\gamma_{ij}}{\sqrt{N}} \left ( 1 -  \frac{(1+C_0) \lambda_i \lambda_j \gamma_{ij}^{p-2}}{(1-C_0)\lambda_{i_1^\ast} \lambda_{j_1^\ast} \gamma_{i_1^\ast j_1^\ast}^{p-2}}  \left (1 - \gamma_{i_1^\ast j_1^\ast}^{p-2} N^{-\frac{p-2}{2(p-1)}}\right)  \right )^{-\frac{1}{p-2}}\\
& \leq \frac{\gamma_{ij}}{\sqrt{N}} \left ( 1 - \frac{1+C_0}{1-C_0} \frac{\delta_{ij}}{1+\gamma_3/\gamma_1} \right )^{-\frac{1}{p-2}} = \frac{\gamma_{ij}}{\sqrt{N}} \frac{1}{(1-\delta_{ij})^{\frac{1}{p-2}}},
\end{split}
\]
where \(\delta_{ij} \in (0,1)\) is defined as \(\lambda_{i_1^\ast} \lambda_{j_1^\ast} \gamma_{i_1^\ast j_1^\ast}^{p-2} = \frac{1}{\delta_{ij}} (1+\gamma_3/\gamma_1) \lambda_i \lambda_j \gamma_{ij}^{p-2}\). Therefore, on the event \(\mathcal{A}\), we have that \(m_{ij}(\mathcal{T}_{\varepsilon_N}^{(i_1^\ast j_1^\ast)}) = \gamma_{ij}' N^{-\frac{1}{2}}\), where \(\gamma_{ij}'>0\) is a constant of order one. \\

From this point onward, the proof of Lemma~\ref{lem: E_1 GF} is identical to the proof of~\cite[Lemma 5.3]{benarous2024langevin}. In particular, we first prove that \(m_{i_1^\ast j_1^\ast}\) attains \(1 - \varepsilon\) for \(\varepsilon \in (0,1)\) with high \(\mathbb{P}\)-probability. Next, we show that the correlation \(m_{i_1^\ast j}\) and \(m_{i j_1^\ast}\) for \(i \neq i_1^\ast, j\neq j_1^\ast\) begin to decrease as \(m_{i_1^\ast j_1^\ast}\) exceeds the threshold \(N^{-\frac{p-2}{2p}}\) and they decrease below \(\frac{1}{\sqrt{\log(N)} N^{\frac{p-1}{4}}}\) with high \(\mathbb{P}\)-probability. Finally, we study the evolution of \(m_{ij}(t)\) for \(i \neq i_1^\ast, j \neq j_1^\ast\) as \(t \geq \mathcal{T}_{\varepsilon_N}^{(i_1^\ast j_1^\ast)}\). We show that as \(m_{i_1^\ast j_1^\ast}\) crosses \(N^{-\frac{p-3}{2(p-1)}}\), the correlations are decreasing until \(m_{i_1^\ast j}\) and \(m_{i j_1^\ast}\) are sufficiently small, ensuring that the decrease is at most by a constant and that \(m_{ij}\) scale as \(N^{-\frac{1}{2}}\), as strong recovery of the first spike is achieved. We therefore obtain that on the initial event \(\mathcal{A}\),
\[
\mathcal{T}_{E_1(\varepsilon)} \lesssim \frac{1}{\sqrt{N}},
\]
with \(\mathbb{P}\)-probability at least \(1 - \exp(-K N)\), thus completing the proof.
\end{proof}

It remains to show Lemma~\ref{lem: E_2 GF}.

\begin{proof}[\textbf{Proof of Lemma~\ref{lem: E_2 GF}}]
We prove the statement for \(k=2\). Let \(\varepsilon>0\) and assume that \(\boldsymbol{X}_0 \in E_1(\varepsilon)\). We show that the evolution of the correlations \(m_{i_1^\ast j_1^\ast}\) and \(m_{i_1^\ast j}, m_{i j_1^\ast}\) for \(i \neq i_1^\ast, j \neq j_1^\ast\) are stable for all \(t \ge 0\), similarly as what done with Langevin dynamics in~\cite{benarous2024langevin} for the proof of Lemma 5.3. 

We therefore look at the evolution of the correlations \(m_{ij}\) for \(i \neq i_1^\ast, j \neq j_1^\ast\). Since \(\boldsymbol{X}_0 \in E_1(\varepsilon)\) we have that \(m_{ij}(0) = \gamma_{ij} N^{-\frac{1}{2}}\) for some order-\(1\) constant \(\gamma_{ij}>0\). Let \(\varepsilon_N = C N^{-\frac{p-2}{2(p-1)}}\) with \(C>0\). By the generator expansion from Lemma~\ref{lem: evolution equation m_ij}, i.e., 
\[
L m_{ij} = L_0 m_{ij} + \sqrt{M} p \lambda_i \lambda_j m_{ij}^{p-1} - \sqrt{M} \frac{p}{2} \sum_{1 \leq k,\ell \leq r} \lambda_k m_{i \ell}m_{kj} m_{k \ell} (\lambda_j m_{kj}^{p-2} + \lambda_\ell m_{k \ell}^{p-2}),
\]
we see that for every \(i \neq i_1^\ast,j \neq j_1^\ast\), 
\[
-\norm{L_0 m_{ij}}_\infty +\sqrt{M} p \lambda_i \lambda_j m_{ij}^{p-1}(t) \le L m_{ij}(t) \leq \norm{L_0 m_{ij}}_\infty +\sqrt{M} p \lambda_i \lambda_j m_{ij}^{p-1}(t),
\]
for all \(t \le \min_{i \neq i_1^\ast, j_1^\ast} \mathcal{T}_{\varepsilon_N}^{(ij)}\). Indeed, the terms associated with \(m_{i_1^\ast j_1^\ast}\) in the generator expansion are also accompanied by \(m_{i_1^\ast j}\) and \(m_{i j_1^\ast}\) which make that globally they are small compared to the term \(\sqrt{M} p \lambda_i \lambda_j m_{ij}^{p-1}\), for \(N\) sufficiently large. We can therefore proceed exactly as done in the proof of Lemma~\ref{lem: E_1 GF}. In particular, the greedy maximum selection gives that \(\lambda_{i_2^\ast} \lambda_{j_2^\ast} \gamma_{i_2^\ast j_2^\ast}^{p-2} > C \lambda_i \lambda_j \gamma_{ij}^{p-2}\) for every \(i,j \in [r], i \neq i_1^\ast, j \neq j_1^\ast\) and some constant \(C >1\). This shows that there exists \(T_2 > \mathcal{T}_{E_1(\varepsilon)}\) such that for all \(T > T_2\),
\[
\inf_{\boldsymbol{X}_0 \in E_1(\varepsilon)} \mathbb{P}_{\boldsymbol{X}_0} \left ( \inf_{t \in [T_2, T]} \boldsymbol{X}_t \in E_2(\varepsilon) \right ) \geq 1 - \exp(-K N)
\]
with \(\mathbb{P}\)-probability at least \(1 - \exp(-K N)\), provided \(N\) is sufficiently large.
\end{proof}

\appendix

\section{Concentration properties of the uniform measure on the Stiefel manifold} \label{appendix: invariant measure}

In this section, we study the concentration and anti-concentration properties of the uniform measure \(\mu_{N \times r}\) on the normalized Stiefel manifold \(\mathcal{S}_{N,r}\). Recall that the correlations are defined by \(m_{ij}^{(N)}(\boldsymbol{X}) = \frac{1}{N}(\boldsymbol{V}^\top \boldsymbol{X})_{ij} =  \frac{1}{N} \langle \boldsymbol{v}_i, \boldsymbol{x}_j \rangle\).

\begin{lem} \label{lem: concentration GF}
Let \(\boldsymbol{X} \sim \mu_{N \times r}\). Then, there exist constants \(C(r), c(r) > 0\), depending only on \(r\), such that for every \(t > 0\) and every \(i,j \in [r]\),
\[
\mu_{N \times r} \left( \left| m_{ij}^{(N)}(\boldsymbol{X}) \right| > t \right) \leq C(r) \exp\left(-c(r) N t^2\right).
\]
\end{lem}

\begin{proof}
From e.g.~\cite[Theorem 2.2.1]{chikuse2012statistics}, a random matrix \(\boldsymbol{X} \sim \mu_{N \times r}\) admits the representation
\[
\boldsymbol{X} = \boldsymbol{Z} \left(\frac{1}{N }\boldsymbol{Z}^\top \boldsymbol{Z}\right)^{-1/2},
\]
where \(\boldsymbol{Z} \in \mathbb{R}^{N \times r}\) has i.i.d.\ standard Gaussian entries. Therefore for every \(t >0\), we obtain
\[ 
\mu_{N \times r} \left( | m_{ij}^{(N)}(\boldsymbol{X}) | > t \right) = \mu_{N \times r} \left( \left | \left (  \frac{1}{N} \boldsymbol{V}^\top \boldsymbol{Z} \left( \frac{1}{N} \boldsymbol{Z}^\top \boldsymbol{Z} \right)^{-1/2} \right )_{ij} \right | > t \right) .
\]
We decompose the right-hand side as
\[
\begin{split}  
& \mu_{N \times r} \left( \left | \left (\frac{1}{N}\boldsymbol{V}^\top\boldsymbol{Z} \left(\frac{1}{N}\boldsymbol{Z}^\top \boldsymbol{Z}\right)^{-1/2} \right )_{ij} \right | > t\right)  \\
&= \mu_{N \times r} \left( \left | \frac{1}{N} \left (\boldsymbol{V}^\top \boldsymbol{Z} \right )_{ij} + \frac{1}{N} \left (\boldsymbol{V}^\top \boldsymbol{Z} \left(\left(\frac{1}{N}\boldsymbol{Z}^\top\boldsymbol{Z}\right)^{-1/2} -\boldsymbol{I}_r \right) \right )_{ij} \right | > t\right)  \\
& \leq \mu_{N \times r} \left(\left | \frac{1}{N} \left ( \boldsymbol{V}^\top\boldsymbol{Z} \right )_{ij} \right | + \left | \frac{1}{N} \left (\boldsymbol{V}^\top \boldsymbol{Z} \left(\left(\frac{1}{N}\boldsymbol{Z}^\top\boldsymbol{Z}\right)^{-1/2}-\boldsymbol{I}_r\right ) \right )_{ij} \right | > t \right).
\end{split} 
\]
We now look at the second summand. Using the submultiplicativity and norm bounds, we obtain
\[
\left | \frac{1}{N} \left (\boldsymbol{V}^\top \boldsymbol{Z} \left(\left(\frac{1}{N}\boldsymbol{Z}^\top\boldsymbol{Z}\right)^{-1/2} - \boldsymbol{I}_r \right )\right )_{ij} \right | \leq \left \lVert \frac{1}{N} \boldsymbol{V}^\top \boldsymbol{Z} \right \rVert_{\textnormal{op}} \left \lVert \left(\frac{1}{N} \boldsymbol{Z}^\top \boldsymbol{Z}\right)^{-1/2} - \boldsymbol{I}_r \right \rVert_{\textnormal{op}}.
\]
We then use the identity 
\[
\boldsymbol{A}^{-1/2}  - \boldsymbol{I}_r=  \boldsymbol{A}^{-1/2} ( \boldsymbol{A}^{1/2} - \boldsymbol{I}_r) = \boldsymbol{A}^{-1/2} ( \boldsymbol{A} - \boldsymbol{I}_r) (\boldsymbol{I}_r + \boldsymbol{A}^{1/2})^{-1},
\]
with \(\boldsymbol{A} = \frac{1}{N} \boldsymbol{Z}^\top \boldsymbol{Z}\), to obtain
\[
\left | \frac{1}{N} \left (\boldsymbol{V}^\top \boldsymbol{Z} \left(\left(\frac{1}{N}\boldsymbol{Z}^\top\boldsymbol{Z}\right)^{-1/2} - \boldsymbol{I}_r \right )\right )_{ij} \right | \leq \left \lVert \frac{1}{N} \boldsymbol{V}^\top \boldsymbol{Z} \right \rVert_{\textnormal{op}} \left \lVert \left(\frac{1}{N} \boldsymbol{Z}^\top \boldsymbol{Z}\right)^{-1/2} \right \rVert_{\textnormal{op}} \left \lVert \frac{1}{N} \boldsymbol{Z}^\top \boldsymbol{Z} - \boldsymbol{I}_r \right \rVert_{\textnormal{op}},
\]
where we bounded \(\norm{(\boldsymbol{I}_r + \boldsymbol{A}^{1/2})^{-1}}_{\textnormal{op}}\) above by \(1\). Standard results on the concentration of sub-Gaussian random matrices (see e.g.~\cite[Theorem 4.6.1]{vershynin2018high}) show that there exists an absolute constant \(C >0\) such that for every \(t >0\),
\begin{equation} \label{eq: op_norm_bound 1}
\mu_{N \times r} \left( \left \lVert \frac{1}{N}\boldsymbol{Z}^\top \boldsymbol{Z}-\boldsymbol{I}_r \right \rVert_{\textnormal{op}} > \max\left(C \left(\frac{\sqrt{r}+t}{\sqrt{N}}\right) ,C^{2} \left(\frac{\sqrt{r}+t}{\sqrt{N}}\right)^{2}\right)\right) \leq 2 \exp(-t^2).
\end{equation}
We note that \(\norm{\frac{1}{N}\boldsymbol{Z}^\top \boldsymbol{Z}-\boldsymbol{I}_r }_{\textnormal{op}} = |\lambda_{\min}(\frac{1}{N}\boldsymbol{Z}^\top \boldsymbol{Z}) - 1| \vee |\lambda_{\max}(\frac{1}{N}\boldsymbol{Z}^\top \boldsymbol{Z}) - 1|\). In particular, since \(\norm{ \left (\frac{1}{N}\boldsymbol{Z}^\top\boldsymbol{Z} \right )^{-1/2}}_{\textnormal{op}} = \left (\lambda_{\min}(\frac{1}{N}\boldsymbol{Z}^\top \boldsymbol{Z}) \right )^{-1/2}\), we can also deduce the bound:
\begin{equation} \label{eq: op_norm_bound 2}
\mu_{N \times r} \left(\left(1+\frac{C\left(\sqrt{r}+t\right)}{\sqrt{N}}\right)^{-1/2} \leq \norm{ \left (\frac{1}{N} \boldsymbol{Z}^\top \boldsymbol{Z} \right )^{-1/2}}_{\textnormal{op}} \leq \left(1-\frac{C \left( \sqrt{r} + t \right)}{\sqrt{N}} \right)^{-1/2} \right) \geq 1 - 2 \exp(-t^2).
\end{equation}
Finally, since the entries of the \(r \times r\) random matrix \(\frac{1}{N} \boldsymbol{V}^\top{\boldsymbol{Z}}\) are i.i.d.\ Gaussian with zero mean and variance \(\frac{1}{N}\), we have the estimate from~\cite[Theorem 4.4.5]{vershynin2018high}:
\begin{equation} \label{eq: op_norm_bound 3}
\mu_{N \times r} \left( \left \lVert \frac{1}{N} \boldsymbol{V}^\top \boldsymbol{Z} \right \rVert_{\textnormal{op}} > \frac{C}{\sqrt{N}} (2 \sqrt{r} + t) \right) \leq 2\exp(-t^2).
\end{equation}
We combine the above estimates~\eqref{eq: op_norm_bound 1}-~\eqref{eq: op_norm_bound 3} to conclude the proof. We split the event:
\[
\begin{split}
\mu_{N \times r}(\vert m_{ij}^{(N)}(\boldsymbol{X}) \vert > t) & \leq \mu_{N \times r} \left( \left | \frac{1}{N} \left (\boldsymbol{V}^\top \boldsymbol{Z} \right)_{ij} \right | > \frac{t}{2}\right)\\ 
& \quad + \mu_{N \times r} \left(\norm{\frac{1}{N}\boldsymbol{V}^\top \boldsymbol{Z}}_{\textnormal{op}} \norm{\left( \frac{1}{N} \boldsymbol{Z}^\top \boldsymbol{Z}\right)^{-1/2}}_{\textnormal{op}} \norm{\frac{1}{N} \boldsymbol{Z}^\top \boldsymbol{Z} - \boldsymbol{I}_r}_{\textnormal{op}} > \frac{t}{2}\right).
\end{split}
\]
Since \(\frac{1}{N}(\boldsymbol{V}^\top \boldsymbol{Z})_{ij} \sim \mathcal{N}(0,1/N)\), we have that the first term is bounded by
\[
\mu_{N \times r} \left( \left | \frac{1}{N} (\boldsymbol{V}^\top \boldsymbol{Z})_{ij} \right | > \frac{t}{2}\right) = \mu_{N \times r} \left( \left | \frac{1}{\sqrt{N}} (\boldsymbol{V}^\top \boldsymbol{Z})_{ij} \right | > \frac{t\sqrt{N}}{2}\right) \leq 2 \exp \left(- Nt^2/8 \right).
\]
Decomposing the second term on the intersection with the event \(\left\{\norm{\frac{1}{\sqrt{N}}(\boldsymbol{Z}^\top\boldsymbol{Z})^{-1/2}}_{\textnormal{op}} \leq \frac{1}{2t}\right\}\) gives 
\[
\begin{split}
& \mu_{N\times r} \left(\norm{\frac{1}{N} \boldsymbol{V}^\top \boldsymbol{Z}}_{\textnormal{op}} \norm{ \left(\frac{1}{N} \boldsymbol{Z}^\top\boldsymbol{Z}\right)^{-1/2}}_{\textnormal{op}}\norm{\frac{1}{N}\boldsymbol{Z}^\top\boldsymbol{Z} - \boldsymbol{I}_r}_{\textnormal{op}} > \frac{t}{2}\right) \\
&\leq \mu_{N \times r} \left(\norm{\frac{1}{N} \boldsymbol{V}^\top \boldsymbol{Z}}_{\textnormal{op}} \norm{\frac{1}{N} \boldsymbol{Z}^\top \boldsymbol{Z}- \boldsymbol{I}_r}_{\textnormal{op}} > t^2\right) + \mu_{N \times r} \left(\norm{ \left( \frac{1}{N}\boldsymbol{Z}^{\top }\boldsymbol{Z} \right)^{-1/2} }_{\textnormal{op}} > \frac{1}{2t}\right).
\end{split}
\]
Finally, decomposing again using the event \(\left\{\norm{\frac{1}{N} \boldsymbol{Z}^\top\boldsymbol{Z} - \boldsymbol{I}_r}_{\textnormal{op}} \leq t\right\}\), we obtain that  
\[
\begin{split} 
& \mu_{N \times r} \left(\norm{\frac{1}{N} \boldsymbol{V}^\top \boldsymbol{Z} }_{\textnormal{op}} \norm{
\left( \frac{1}{N} \boldsymbol{Z}^\top \boldsymbol{Z} \right)^{-1/2}}_{\textnormal{op}} \norm{\frac{1}{N}  \boldsymbol{Z}^\top \boldsymbol{Z} - \boldsymbol{I}_r}_{\textnormal{op}} > \frac{t}{2}\right) \\
& \leq \mu_{N \times r} \left(\norm{\frac{1}{N} \boldsymbol{V}^\top \boldsymbol{Z}}_{\textnormal{op}} > t\right) + \mu_{N \times r} \left ( \norm{\frac{1}{N} \boldsymbol{Z}^\top \boldsymbol{Z} - \boldsymbol{I}_r}_{\textnormal{op}} > t \right) + \mu_{N \times r} \left(\norm{ \left( \frac{1}{N}\boldsymbol{Z}^{\top }\boldsymbol{Z} \right)^{-1/2} }_{\textnormal{op}} > \frac{1}{2t}\right).
\end{split}
\]
Using~\eqref{eq: op_norm_bound 1},~\eqref{eq: op_norm_bound 2}, and~\eqref{eq: op_norm_bound 3} we then find that 
\[
\begin{split}
& \mu_{N \times r} \left(\norm{\frac{1}{N} \boldsymbol{V}^\top \boldsymbol{Z}}_{\textnormal{op}} \norm{ \left( \frac{1}{N} \boldsymbol{Z}^\top \boldsymbol{Z}\right)^{-1/2}}_{\textnormal{op}} \norm{\frac{1}{N} \boldsymbol{Z}^\top \boldsymbol{Z} - \boldsymbol{I}_r}_{\textnormal{op}} > \frac{t}{2}\right) \\
&\leq 2 \exp\left(-\left( \frac{t \sqrt{N}}{C} - 2 \sqrt{r} \right)^2 \right) + 2 \exp\left(-\left(\frac{t\sqrt{N}}{C}-\sqrt{r}\right)^2 \right) + 2\exp\left(-\left(\frac{\sqrt{N}}{C}(1-2t)-\sqrt{r}\right)^2\right).
\end{split}
\]
Combining all bounds, we obtain the desired result.
\end{proof}

\begin{lem} \label{lem: anti-concentration GF}
Let \(\boldsymbol{X} \sim \mu_{N \times r}\). Then, there exist constants \(C(r),c(r) >0\), depending only on \(r\), such that for every \(t >0\) and every \(i,j \in [r]\),
\[
\mu_{N \times r} \left(\vert m_{ij}^{(N)}(\boldsymbol{X}) \vert < \frac{t}{\sqrt{N}}\right) \leq \frac{4}{\sqrt{2 \pi}} t + C(r) \exp \left( -c(r)t \sqrt{N} \right).
\]
\end{lem}

\begin{proof}
Using a similar argument as in Lemma~\ref{lem: concentration GF} and the fact that \(|a + b| \geq ||a| - |b||\), we have 
\[
\begin{split}
\mu_{N \times r} \left(\vert m_{ij}^{(N)}(\boldsymbol{X})  \vert < t \right) &\leq \mu_{N \times r} \left( \vert \frac{1}{N}(\boldsymbol{V}^\top\boldsymbol{Z})_{ij} \vert < 2t\right) \\
& \quad + \mu_{N \times r} \left(\norm{\frac{1}{N} \boldsymbol{V}^\top\boldsymbol{Z}}_{\textnormal{op}} \norm{ \left( \frac{1}{N} \boldsymbol{Z}^\top \boldsymbol{Z}\right)^{-1/2}}_{\textnormal{op}} \norm{ \frac{1}{N} \boldsymbol{Z}^\top \boldsymbol{Z} - \boldsymbol{I}_r}_{\textnormal{op}} > t\right).
\end{split}
\]
We bound the first term as
\[
\begin{split}
\mu_{N \times r}\left(\vert \frac{1}{N}(\boldsymbol{V}^\top\boldsymbol{Z})_{ij} \vert \leq 2t \right) & = 2 \mu_{N \times r} \left ( 0 \le \left |\frac{1}{\sqrt{N}} (\boldsymbol{V}^\top \boldsymbol{Z})_{ij} \right |  \leq 2 t\sqrt{N} \right ) \\
& = \frac{2}{\sqrt{2 \pi}} \int_0^{2t \sqrt{N}} e^{-x^2/2} \mathrm{d} x\\ 
&\leq \frac{4t \sqrt{N}}{\sqrt{2 \pi}},
\end{split}
\]
where we used \(e^{-x^2/2} \le 1\) in the last line. For the second term, we use a similar argument as in the proof of Lemma~\ref{lem: concentration GF} with different thresholds. For every \(\eta > 0\), we successively decompose on the events \(\norm{\left ( \frac{1}{N} \boldsymbol{Z}^\top \boldsymbol{Z} \right )^{-1/2}}_{\textnormal{op}} \leq \frac{1}{\eta}\) and \(\norm{\boldsymbol{I}_r - \frac{1}{N} \boldsymbol{Z}^\top \boldsymbol{Z}}_{\textnormal{op}} \leq \sqrt{t\eta}\). It then follows that
\[
\begin{split}
& \mu_{N \times r} \left(\norm{\frac{1}{N} \boldsymbol{V}^\top\boldsymbol{Z}}_{\textnormal{op}} \norm{ \left( \frac{1}{N} \boldsymbol{Z}^\top \boldsymbol{Z}\right)^{-1/2}}_{\textnormal{op}} \norm{ \frac{1}{N} \boldsymbol{Z}^\top \boldsymbol{Z} - \boldsymbol{I}_r}_{\textnormal{op}} > t\right)     \\
& \leq \mu_{N \times r} \left ( \norm{\frac{1}{N} \boldsymbol{V}^\top\boldsymbol{Z}}_{\textnormal{op}}  > \sqrt{t \eta} \right ) + \mu_{N \times r} \left ( \norm{ \left( \frac{1}{N} \boldsymbol{Z}^\top \boldsymbol{Z}\right)^{-1/2}}_{\textnormal{op}} > \frac{1}{\eta} \right )  \\
& \quad + \mu_{N \times r} \left ( \norm{ \frac{1}{N} \boldsymbol{Z}^\top \boldsymbol{Z} - \boldsymbol{I}_r}_{\textnormal{op}}> \sqrt{t \eta} \right ) \\
& \leq 2 \exp\left(-\left( \frac{\sqrt{t \eta N}}{C} - 2 \sqrt{r} \right)^2 \right) + 2 \exp \left(-\left(\frac{\sqrt{N}}{C}(1 - \eta)-\sqrt{r}\right)^2 \right) \\
& \quad + 2 \exp\left(-\left(\frac{\sqrt{t \eta N} }{C} -\sqrt{r} \right)^2 \right).
\end{split}
\]
Choosing \(\eta = \frac{1}{2}\) and replacing \(t\) by \(\frac{t}{\sqrt{N}}\) completes the proof.
\end{proof}

From Lemmas~\ref{lem: concentration GF} and~\ref{lem: anti-concentration GF}, it follows that \(\mu_{N \times r}\) satisfies Condition 1. We now proceed to verify that the invariant measure \(\mu_{N \times r}\) satisfies Condition 2.

\begin{lem} \label{lem: separation initial data GF}
Let \(p \ge 3\) and \(\lambda_1 \ge \cdots \ge \lambda_r \ge 0\). Let \(\boldsymbol{X} \sim \mu_{N \times r}\). Then, there exist constants \(C(r) >0\) and \(c(r,\{\lambda_i\}_{i=1}^r)>0\) such that for every \(0 < t < \gamma_1 \) and every \(1 \leq i,j, k, \ell \leq r\), \((i,j) \neq (k,\ell)\), 
\[
\begin{split}
& \mu_{N \times r} \left( \left | \frac{ \lambda_i \lambda_j \left (m_{ij}^{(N)}(\boldsymbol{X}) \right)^{p-2}}{\lambda_k \lambda_\ell \left (m_{k \ell}^{(N)}(\boldsymbol{X}) \right)^{p-2}}-1 \right | \leq \frac{t}{\gamma_1} \right) \\
& \leq C_1 e^{- c_1 \gamma_1^2} + C_2 e^{- c_2 (\lambda_k \lambda_\ell)^{\frac{1}{p-2}} \sqrt{N} t} + \frac{4}{\sqrt{2\pi} \sqrt{ 1 + \left( \frac{\lambda_i \lambda_j}{\lambda_k \lambda_\ell} \right)^{\frac{2}{p-2}}}} t.
\end{split}
\]
\end{lem}
\begin{proof}
To simplify notation slightly, we let \(\alpha_{ij} = \lambda_i \lambda_j\) for every \(i,j \in [r]\). For every \(\delta >0\), we denote by \(\mathcal{A}(\delta)\) the desired event, i.e., 
\[
A(\delta) = \left \{ \left | \frac{ \alpha_{ij} \left (m_{ij}(\boldsymbol{X}) \right)^{p-2}}{\alpha_{k \ell} \left (m_{k \ell}(\boldsymbol{X}) \right)^{p-2}} - 1 \right | \leq \delta \right \} = \left \{ 1 - \delta \le \frac{ \alpha_{ij} \left (m_{ij}(\boldsymbol{X}) \right)^{p-2}}{\alpha_{k \ell} \left (m_{k \ell}(\boldsymbol{X}) \right)^{p-2}} \le 1 + \delta \right \} .
\]
We then introduce the event \(\mathcal{B}(\delta)\) given by
\[
\mathcal{B}(\delta) = \left \{ (1 - \delta)^{p-2} \le \frac{ \alpha_{ij} \left (m_{ij}(\boldsymbol{X}) \right)^{p-2}}{\alpha_{k \ell} \left (m_{k \ell}(\boldsymbol{X}) \right)^{p-2}} \le (1 + \delta)^{p-2} \right \} 
\]
so that \(\mathcal{A}(\delta) \subseteq \mathcal{B}(\delta)\), with equality when \(p=3\). It therefore suffices to estimate the event \(\mathcal{B}(\delta)\). We note that controlling \(\mathcal{B}(\delta)\) is equivalent to controlling 
\[
\bar{\mathcal{B}}(\delta) = \left\{ \left | \frac{\beta_{ij} m_{ij}(\boldsymbol{X}) - \beta_{k \ell} m_{k \ell}(\boldsymbol{X})}{\beta_{k \ell} m_{k \ell}(\boldsymbol{X})} \right | \leq \delta \right\},
\]
where \(\beta_{ij} = \alpha_{ij}^{\frac{1}{p-2}}\). In light of Lemma~\ref{lem: concentration GF}, since \(\boldsymbol{X} \sim \mu_{N \times r}\), the event 
\[
\mathcal{E}(\gamma_1) = \left \{ \boldsymbol{X} \colon \left |  m_{ij}(\boldsymbol{X}) \right| \le \frac{\gamma_1}{\sqrt{N}} \right \}
\]
occurs with probability at least \(1 - C e^{- c \gamma_1^2}\). We introduce a further event: for every \(t \ge 0\), we consider the event \(\tilde{\mathcal{B}}(t)\) given by
\[
\tilde{\mathcal{B}}(t) = \left\{ \left | \beta_{ij} m_{ij}(\boldsymbol{X}) - \beta_{k \ell} m_{k \ell}(\boldsymbol{X}) \right | < \frac{t}{\sqrt{N}} \right\}.
\]
Then, we note that \(\tilde{\mathcal{B}}^\textnormal{c}(t) \cap \mathcal{E}(\gamma_1) \subset \bar{\mathcal{B}}^{\textnormal{c}} \left ( \frac{t}{\beta_{k \ell} \gamma_1} \right)\), so that 
\[
\mu_{N \times r} \left( \bar{\mathcal{B}}\left ( \frac{t}{\beta_{k \ell} \gamma_1} \right) \right ) \leq \mu_{N \times r} \left( \tilde{\mathcal{B}}(t) \cup \mathcal{E}^{\textnormal{c}}(\gamma_1) \right) \le \mu_{N \times r} \left ( \tilde{\mathcal{B}}(t) \right) + \mu_{N \times r} \left (\mathcal{E}^{\textnormal{c}}(\gamma_1) \right).
\]
It remains to estimate \(\mu_{N \times r} \left ( \tilde{\mathcal{B}}(t) \right)\). We will proceed in a similar way as done for the proofs of Lemmas~\ref{lem: concentration GF} and~\ref{lem: anti-concentration GF} by using the representation \(\boldsymbol{X} = \boldsymbol{Z} \left(\frac{1}{N} \boldsymbol{Z}^\top \boldsymbol{Z}\right)^{-1/2}\) for \(\boldsymbol{X} \sim \mu_{N \times r}\) (see e.g.~\cite{chikuse2012statistics}). In particular, if we write 
\[
\begin{split}
m_{ij}^{(N)}(\boldsymbol{X}) & =  \left( \frac{1}{N}\boldsymbol{V}^\top  \boldsymbol{Z} \left(\frac{1}{N} \boldsymbol{Z}^\top \boldsymbol{Z}\right)^{-1/2} \right)_{ij} = \frac{1}{N} (\boldsymbol{V}^\top  \boldsymbol{Z})_{ij} + \frac{1}{N}  \left (\boldsymbol{V}^\top  \boldsymbol{Z} \left ( \left ( \frac{1}{N} \boldsymbol{Z}^\top \boldsymbol{Z} \right )^{-1/2} - \boldsymbol{I}_r \right) \right)_{ij},
\end{split}
\]
we can upper bound \(\mu_{N \times r} \left ( \tilde{\mathcal{B}}(t) \right)\) by 
\[
\begin{split}
\mu_{N \times r} \left ( \tilde{\mathcal{B}}(t) \right) &  \leq \mu_{N \times r} \left ( \left | \beta_{ij} \frac{1}{N} (\boldsymbol{V}^\top  \boldsymbol{Z})_{ij}  - \beta_{k \ell} \frac{1}{N} (\boldsymbol{V}^\top  \boldsymbol{Z})_{k \ell}\right | < \frac{2t}{\sqrt{N}} \right ) \\
& \quad + \mu_{N \times r} \left( \left | \beta_{ij} \frac{1}{N}  \left (\boldsymbol{V}^\top  \boldsymbol{Z} \left( \left ( \frac{1}{N} \boldsymbol{Z}^\top \boldsymbol{Z} \right )^{-1/2} - \boldsymbol{I}_r \right) \right)_{ij} \right |> \frac{t}{2\sqrt{N}} \right ) \\
& \quad + \mu_{N \times r} \left( \left | \beta_{k \ell} \frac{1}{N}  \left (\boldsymbol{V}^\top  \boldsymbol{Z} \left( \left ( \frac{1}{N} \boldsymbol{Z}^\top \boldsymbol{Z} \right )^{-1/2} - \boldsymbol{I}_r \right) \right)_{k \ell} \right |> \frac{t}{2\sqrt{N}} \right ).
\end{split}
\]
The second and third concentration estimates can be bounded as done in the proof of Lemma~\ref{lem: anti-concentration GF}, so that there exist constants \(C,c(r)\) such that 
they are bounded by \(C\exp\left(-c(r)\frac{t}{\beta_{ij} \vee \beta_{k \ell}} \sqrt{N} \right)\). To bound the first term, we note that \(\beta_{ij} \frac{1}{\sqrt{N}} (\boldsymbol{V}^\top \boldsymbol{Z})_{ij}\) is a Gaussian random variable with zero mean and variance \(\beta_{ij}^2\). We easily note that the random variable \(\beta_{ij} \frac{1}{\sqrt{N}} (\boldsymbol{V}^\top \boldsymbol{Z})_{ij} - \beta_{k \ell} \frac{1}{\sqrt{N}} (\boldsymbol{V}^\top \boldsymbol{Z})_{k \ell}\) follows a normal distribution with zero mean and variance  \(\beta_{ij}^2 + \beta_{k \ell}^2\), ensuring that 
\[
\begin{split}
\mu_{N \times r} \left ( \left | \beta_{ij} \frac{1}{N} (\boldsymbol{V}^\top  \boldsymbol{Z})_{ij}  - \beta_{k \ell} \frac{1}{N} (\boldsymbol{V}^\top  \boldsymbol{Z})_{k \ell}\right | < \frac{2t}{\sqrt{N}} \right ) &= \mu_{N \times r} \left ( |\mathcal{N}(0,\beta_{ij}^2 + \beta_{k \ell}^2)| < 2t \right ) \\
& \leq \frac{4}{\sqrt{2 \pi} \sqrt{\beta_{ij}^2 + \beta_{k \ell}^2}} t.
\end{split} 
\]
Finally, we find that
\[
\begin{split}
\mu_{N \times r} \left( \bar{\mathcal{B}} \left ( \frac{t}{\beta_{k \ell} \gamma_1} \right) \right) & \leq C_1  e^{- c_1 \gamma_1^2} + C_2 \exp \left(-c_2 (\{\lambda_i \}_{i=1}^r ) t\sqrt{N}\right) +\frac{4}{\sqrt{2\pi} \sqrt{\beta_{ij}^2 + \beta_{k \ell}^2}} t,
\end{split}
\]
which completes the proof since 
\[
\begin{split}
\mu_{N \times r} \left( \mathcal{A} \left( \frac{t}{\gamma} \right)\right) & \leq C_1 e^{-c_1\gamma_1^2} + C_2 \exp \left(- c_2 \beta_{k \ell} t \sqrt{N}  \right) + \frac{4 \beta_{k \ell}}{\sqrt{2\pi} \sqrt{\beta_{ij}^2 + \beta_{k \ell}^2}} t.
\end{split}
\]
\end{proof} 

It remains to prove the following concentration estimate which ensures that \(\mu_{N \times r}\) weakly satisfies Condition 1 at level \(\infty\). 

\begin{lem} \label{lem: concentration cond infinity}
For every \(T > 0\) and every \(1 \leq i,j \leq r\), there exist \(C_1, C_2 > 0\), depending only on \(p,r,\{\lambda_{i}\}_{i=1}^r\), such that for every \(\gamma>0\),
\[
\mu_{N \times r} \left(\sup_{t \leq T} \vert e^{t L_0} L_0 m_{ij}^{(N)} (\boldsymbol{X}) \vert \geq \gamma \right) \leq C_{1} N T \exp\left(-C_{2}\gamma^2N\right),
\]
with \(\mathbb{P}\)-probability at least \(1-\mathcal{O}(e^{-KN})\).
\end{lem}

We prove Lemma~\ref{lem: concentration cond infinity} following the same ideas to those used to prove Theorem 6.2 of~\cite{arous2020algorithmic}. In the following, we let \(\hat{\boldsymbol{X}}_t\) denote the gradient flow process generated by \(L_0\) (see~\eqref{eq: generator GF noise}). The first step is to establish the rotational invariance properties of this dynamics. Compared with what was done in Section 6 of~\cite{arous2020algorithmic}, here we need to introduce an intermediate quantity to study the gradient \(\nabla H_0 (\hat{\boldsymbol{X}}_t)\), which is necessary to avoid having to control quantities that vanish exponentially over time. In addition, we note that the invariant measure on the normalized Stiefel manifold is characterized by left and right invariance under rotations, whereas the invariant measure on the sphere requires only verification of invariance under rotations.  

In the remainder of this subsection, for every \(\boldsymbol{X} \in \mathcal{S}_{N,r}\) we let \(R^N_{\boldsymbol{X}} \colon T_{\boldsymbol{X}} \mathcal{S}_{N,r} \to \mathcal{S}_{N,r}\) denote the polar retraction defined by 
\[
R^N_{\boldsymbol{X}}(\boldsymbol{U}) = \left (\boldsymbol{X} +\boldsymbol{U} \right) \left(\boldsymbol{I}_r + \frac{1}{N} \boldsymbol{U}^\top \boldsymbol{U}\right)^{-1/2}.
\]
which verifies \(\left(R^N_{\boldsymbol{X}} (\boldsymbol{U})\right)^\top R^N_{\boldsymbol{X}} (\boldsymbol{U}) = N \boldsymbol{I}_r\). 

\begin{lem} \label{lem: noise invariance}
For every \(\boldsymbol{X} \sim \mu_{N \times r}\) and every \(\boldsymbol{U} \in T_{\boldsymbol{X}} \mathcal{S}_{N,r}\), we let \(R^N_{\boldsymbol{X}}(\boldsymbol{U})\) denote the polar retraction at the point \((\boldsymbol{X},\boldsymbol{U})\). Then, for every \(t \geq 0\), if \(\hat{\boldsymbol{X}}_0 \sim \mu_{N \times r}\), \(\hat{\boldsymbol{X}}_t\) and \(R^N_{\boldsymbol{X}} \left(\nabla H_0(\hat{\boldsymbol{X}}_t)\right)\) are elements of \(\mathcal{S}_{N,r}\) that are invariant under left rotations. 
\end{lem}
\begin{proof}
We let \(\boldsymbol{O}_N\) denote the elements of the orthogonal groups \(O(N)\). The initial condition \(\hat{\boldsymbol{X}}_0 \sim \mu_{N \times r}\) satisfies 
\[
\boldsymbol{O}_N \hat{\boldsymbol{X}}_0 \sim \hat{\boldsymbol{X}}_0.
\]
In the following, we let 
\[
\tilde{\boldsymbol{X}}_0 = \boldsymbol{O}_N \hat{\boldsymbol{X}}_0 \quad \textnormal{and} \quad \tilde{H}_0 (\boldsymbol{X}) = H_0 \left( \boldsymbol{O}_N^{-1} \boldsymbol{X} \right),
\]
and \(\tilde{\boldsymbol{X}}_t\) denote the gradient flow on \(\tilde{H}_0\) started from \(\tilde{\boldsymbol{X}}_0\). Since the Hamiltonian \(H_0\) is a centered Gaussian process with covariance function given by
\[
\E \left[H_0(\boldsymbol{X})H_0(\boldsymbol{Y})\right] = N \sum_{1 \leq i,j \leq r}\lambda_{i}\lambda_{j}\left(\frac{\langle \boldsymbol{x}_i ,\boldsymbol{y}_j \rangle}{N}\right)^p, 
\]
we see that \(H_0\) is invariant under right rotations, i.e., \(H_0(\boldsymbol{O}_N\boldsymbol{X})\) is equidistributed with \(H_0(\boldsymbol{X})\). Since \(\tilde{\boldsymbol{X}}_0\) is equidistributed with \(\hat{\boldsymbol{X}}_0\), and \(\tilde{H}_0(\boldsymbol{X})\) is equidistributed with \(H_0(\boldsymbol{X})\) for every \(\boldsymbol{X} \in \mathcal{S}_{N,r}\), the gradient \(\nabla\tilde{H}_0(\tilde{\boldsymbol{X}}_t)\) is equal in distribution as \(\nabla H_0(\hat{\boldsymbol{X}}_t)\). Since \(\tilde{\boldsymbol{X}}_t = \boldsymbol{O}_N \hat{\boldsymbol{X}}_t\), we deduce that \(\hat{\boldsymbol{X}}_t\) is invariant under right rotations for every \(t \geq 0\). We also have that  
\[
\nabla\tilde{H}_0(\tilde{\boldsymbol{X}}_t) = \boldsymbol{O}_N\nabla H_0 (\hat{\boldsymbol{X}}_t ).
\]
Since \(\nabla H_0(\hat{\boldsymbol{X}}_t)\) is equidistributed with \(\nabla \tilde{H}_0(\tilde{\boldsymbol{X}}_t)\), we have that \(\nabla H_0(\hat{\boldsymbol{X}}_t)\) is also invariant under rotations from the right. Finally, we have that 
\[
R^N_{\boldsymbol{X}}(\nabla H_0(\hat{\boldsymbol{X}}_t)) = \left(\boldsymbol{X}+\nabla H_0(\hat{\boldsymbol{X}}_t)\right)\left( \boldsymbol{I}_r + \frac{1}{N} \left( \nabla H_0(\hat{\boldsymbol{X}}_t)\right)^\top\nabla H_0(\hat{\boldsymbol{X}}_t)\right)^{-1/2}
\]
is well defined for every \(t \geq 0\) and, in particular, for every value of \( \norm{\nabla H_0(\hat{\boldsymbol{X}}_t)}_2\). Since \(\boldsymbol{X} \sim \mu_{N \times r}\), we have that \(R^N_{\boldsymbol{X}}(\nabla H_0(\hat{\boldsymbol{X}}_t))\) is an element of $\mathcal{S}_{N,r}$ and is invariant under rotations from the right.
\end{proof}

\begin{rmk}
We remark that one could use the matrix \(R^N_{\boldsymbol{x}} (\nabla_{\mathcal{S}^N} H_0(\hat{\boldsymbol{x}}_t))\) also in the spherical case studied in~\cite{arous2020algorithmic}. In this case, \(R^N_{\boldsymbol{x}} (\nabla_{\mathcal{S}^N} H_0(\hat{\boldsymbol{x}}_t))\) reduces to the vector 
\[
R^N_{\boldsymbol{x}} (\nabla_{\mathcal{S}^N} H_0(\hat{\boldsymbol{x}}_t))  = \frac{\boldsymbol{x}_0 + \nabla_{\mathcal{S}^N} H_0 (\hat{\boldsymbol{x}}_t)}{\norm{\boldsymbol{x}_0 + \nabla_{ \mathcal{S}^N} H_0(\hat{\boldsymbol{x}}_t)}_2},
\]
with \(\boldsymbol{x}_0\) being distributed according to the invariant measure on the sphere \(\mathcal{S}^N = \mathbb{S}^{N-1}(\sqrt{N})\). The orthogonality between the sphere and its tangent space ensures that the normalizing factor \(\norm{\boldsymbol{x}_0 + \nabla_{\mathcal{S}^N} H_0(\hat{\boldsymbol{x}}_t)}_2\) is always strictly greater than one.
\end{rmk}

Having Lemma~\ref{lem: noise invariance} at hand, we are now able to prove that the invariant measure \(\mu_{N \times r}\) weakly satisfies Condition \(1\) at level \(\infty\). 

\begin{proof}[\textbf{Proof of Lemma~\ref{lem: concentration cond infinity}}]
By definition of the semigroup of the noise process, it holds for every \(1 \leq i,j \leq r\),
\[
e^{tL_0} L_0 m_{ij}^{(N)}(\hat{\boldsymbol{X}}_0) = L_0 m_{ij}^{(N)}(\hat{\boldsymbol{X}}_t)= -\frac{1}{N} \langle  \nabla H_0(\hat{\boldsymbol{X}}_t), [\boldsymbol{v}_{i}]_{j} \rangle,
\]
where \([\boldsymbol{v}_i]_j = [\boldsymbol{0}, \ldots, \boldsymbol{0}, \boldsymbol{v}_i, \boldsymbol{0}, \ldots, \boldsymbol{0}] \in \R^{N \times r}\) denotes the matrix with all zero columns except for the \(j\)th column, which is \(\boldsymbol{v}_i\). Therefore, it suffices to study \(L_0 m_{ij}^{(N)}(\hat{\boldsymbol{X}}_t)\). We let \(\boldsymbol{H} \in \R^{r \times r}\) be a matrix sampled from the Haar measure on \(O(r)\). For every \(t \geq 0\), we define \(\boldsymbol{Z}_t = R^N_{\hat{\boldsymbol{X}}_0} (\nabla H_0(\hat{\boldsymbol{X}}_t)) \boldsymbol{H}\). According to Lemma~\ref{lem: noise invariance} and by definition of the Haar measure, we have that \(\boldsymbol{Z}_t\) belongs to \(\mathcal{S}_{N,r}\) and is invariant under left and right rotations. Since this property uniquely characterizes the invariant measure on \(\mathcal{S}_{N,r}\), we deduce that \(\boldsymbol{Z}_t\) is distributed according to \(\mu_{N \times r}\). Combining this with Lemma~\ref{lem: concentration GF}, we obtain that for every \(t \geq 0\) there exist \(C(r),c(r) > 0\) such that for every \(\gamma >0\),
\[
\mu_{N \times r} \otimes \mathbb{P} \left(\frac{1}{N}\vert \langle \boldsymbol{Z}_t,[\boldsymbol{v}_i]_j \rangle \vert \geq \gamma\right) \leq C(r)\exp\left(-c(r)\gamma^2N\right).
\]
The rest of the proof follows a similar argument as done for Theorem 6.2 of~\cite{arous2020algorithmic}, based on a discretization of the trajectory. In light of Lemma~\ref{lem: regularity H0}, for constants \(\Gamma = \Gamma (p,\{\lambda_i\}_{i=1}^r)\) and \(K = K(p,\{\lambda_i\}_{i=1}^r)\) we have that the event
\[
\mathcal{E} = \left\{\norm{H_0}_{\mathcal{G}^2} \geq \Gamma N \right\}
\]
holds with \(\mathbb{P}\)-probability at most \(\exp\left(- K N\right)\). We direct the reader to Definition~\ref{def: G norm} for a definition of the \(\mathcal{G}^n\)-norm on \(\mathcal{S}_{N,r}\). According to Definition~\ref{def: G norm}, we easily notice that, under the event \(\mathcal{E}^\textnormal{c}\), 
\[
\norm{\vert \nabla^2 H_0(\boldsymbol{X}) \vert_{\textnormal{op}}}_{\infty} \leq \Gamma.
\]
Then, for every \(0 \leq s \leq t\) we have that
\[
\begin{split}
\norm{ \boldsymbol{Z}_t - \boldsymbol{Z}_s}_{\textnormal{F}} &= \norm{R^N_{\hat{\boldsymbol{X}}_0}(\nabla H_0(\hat{\boldsymbol{X}}_t))\boldsymbol{H} - R^N_{\hat{\boldsymbol{X}}_0}(\nabla H_0(\hat{\boldsymbol{X}}_s))\boldsymbol{H}}_{\textnormal{F}} \\
& \leq r \norm{R^N_{\hat{\boldsymbol{X}}_0}(\nabla H_0(\hat{\boldsymbol{X}}_t))-R^N_{\hat{\boldsymbol{X}}_0}(\nabla H_0(\hat{\boldsymbol{X}}_s))}_{\textnormal{F}},
\end{split}
\]
where we used \(\norm{\boldsymbol{H}}_{\text{F}} \le r\). Recall that by definition,
\[
R^N_{\hat{\boldsymbol{X}}_0}(\nabla H_0(\hat{\boldsymbol{X}}_t)) = \left( \hat{\boldsymbol{X}}_0 + \nabla H_0(\hat{\boldsymbol{X}}_t)\right) \left( \boldsymbol{I}_r + \frac{1}{N} \nabla H_0(\hat{\boldsymbol{X}}_t)^\top \nabla H_0 (\hat{\boldsymbol{X}}_t) \right)^{-1/2}.
\]
In the following, we let \(\boldsymbol{U}_t\) denote the Riemannian gradient \(\boldsymbol{U}_t = \nabla H_0(\hat{\boldsymbol{X}}_t) \in \R^{N \times r}\) for every \(t \geq 0\). We therefore write the difference \(R^N_{\hat{\boldsymbol{X}}_0}(\nabla H_0(\hat{\boldsymbol{X}}_t))-R^N_{\hat{\boldsymbol{X}}_0}(\nabla H_0(\hat{\boldsymbol{X}}_s))\) as
\[
\begin{split}
& R^N_{\hat{\boldsymbol{X}}_0}(\nabla H_0(\hat{\boldsymbol{X}}_t))-R^N_{\hat{\boldsymbol{X}}_0}(\nabla H_0(\hat{\boldsymbol{X}}_s)) \\
& = \left ( \hat{\boldsymbol{X}}_0 + \boldsymbol{U}_t \right) \left(\left(\boldsymbol{I}_r+\frac{1}{N}\boldsymbol{U}_t^\top\boldsymbol{U}_t\right)^{-1/2}-\left(\boldsymbol{I}_r+\frac{1}{N}\boldsymbol{U}_s^\top\boldsymbol{U}_s\right)^{-1/2}\right)  \\
& \quad + \left(\boldsymbol{U}_t-\boldsymbol{U}_s\right)\left(\boldsymbol{I}_r+\frac{1}{N}\boldsymbol{U}_s^\top\boldsymbol{U}_s\right)^{-1/2},
\end{split}
\]
so that for every \(0 \leq s \leq t\),
\[
\begin{split}
\norm{ \boldsymbol{Z}_t - \boldsymbol{Z}_s}_{\textnormal{F}} & \leq r \norm{\hat{\boldsymbol{X}}_0 + \boldsymbol{U}_t}_{\textnormal{F}} \norm{\left(\boldsymbol{I}_r + \frac{1}{N}\boldsymbol{U}_t^\top\boldsymbol{U}_t\right)^{-1/2}-\left(\boldsymbol{I}_r+\frac{1}{N}\boldsymbol{U}_s^\top\boldsymbol{U}_s\right)^{-1/2}}_{\textnormal{F}} + r \norm{\boldsymbol{U}_t - \boldsymbol{U}_s}_{\textnormal{F}} \\
& \leq r \norm{\hat{\boldsymbol{X}}_0 + \boldsymbol{U}_t}_{\textnormal{F}} \norm{\left(\boldsymbol{I}_r + \frac{1}{N}\boldsymbol{U}_t^\top\boldsymbol{U}_t\right)^{1/2}-\left(\boldsymbol{I}_r +\frac{1}{N}\boldsymbol{U}_s^\top\boldsymbol{U}_s\right)^{1/2}}_{\textnormal{F}} + r \norm{\boldsymbol{U}_t - \boldsymbol{U}_s}_{\textnormal{F}},
\end{split}
\]
where we used the fact that \(\norm{\left(\boldsymbol{I}_r + \frac{1}{N} \nabla H_0(\hat{\boldsymbol{X}}_t)^\top \nabla H_0(\hat{\boldsymbol{X}}_t)\right)^{-1/2}}_{\textnormal{F}} \leq 1\) and that \(\boldsymbol{A}^{-1} - \boldsymbol{B}^{-1} = - \boldsymbol{A}^{-1} \left (\boldsymbol{A} - \boldsymbol{B} \right) \boldsymbol{B}^{-1} \) for invertible matrices \(\boldsymbol{A}, \boldsymbol{B} \in \R^{r \times r}\). Hölder continuity for the matrix square-root (see e.g.~\cite[Theorem X.1.1]{Bhatia}) then implies that
\begin{equation} \label{eq: bound difference Z}
\norm{\boldsymbol{Z}_t-\boldsymbol{Z}_s}_{\textnormal{F}} \leq \frac{r}{\sqrt{N}} \norm{\hat{\boldsymbol{X}}_0+\boldsymbol{U}_t}_{\textnormal{F}} \norm{\boldsymbol{U}_t^\top\boldsymbol{U}_t -\boldsymbol{U}_s^\top\boldsymbol{U}_s}^{\frac{1}{2}}_{\textnormal{F}} + r \norm{\boldsymbol{U}_t - \boldsymbol{U}_s}_{\textnormal{F}}.
\end{equation}
We now observe that 
\[
\frac{\mathrm{d}}{\mathrm{d}t} \nabla H_0(\hat{\boldsymbol{X}}_t) = \nabla^2 H_0 \left(\nabla H_0(\hat{\boldsymbol{X}}_t), \cdot\right),
\]
so that, under the event \(\mathcal{E}^{\textnormal{c}}\), we have that
\[
\begin{split}
\norm{\boldsymbol{U}_t - \boldsymbol{U}_s}_{\textnormal{F}} & = \norm{\nabla H_0(\hat{\boldsymbol{X}}_t) -\nabla H_0(\hat{\boldsymbol{X}}_s)}_{\textnormal{F}} \\
& \leq \int_s^t \norm{\nabla^2 H_0 \left(\nabla H_0(\hat{\boldsymbol{X}}_u), \cdot \right)}_{\textnormal{F}} \mathrm{d} u \\
&\leq \Gamma \sqrt{N} (t-s).
\end{split}
\]
Similarly, 
\[
\frac{\mathrm{d}}{\mathrm{d}t}  \nabla H_0(\hat{\boldsymbol{X}}_t)^\top\nabla H_0(\hat{\boldsymbol{X}}_t)  = 2\nabla^2  H_0 \left(\nabla H_0(\hat{\boldsymbol{X}}_t), \nabla H_0(\hat{\boldsymbol{X}}_t)\right),
\]
so that under \(\mathcal{E}^{\textnormal{c}}\),
\[
\norm{\boldsymbol{U}_t^\top\boldsymbol{U}_t-\boldsymbol{U}_s^\top\boldsymbol{U}_s}_{\textnormal{F}} \leq 2 \Gamma^2 N (t-s).
\]
Finally, since \(\norm{\boldsymbol{U}_t}_{\textnormal{F}}\) is a decreasing function of time, under \(\mathcal{E}^{\textnormal{c}}\), it holds that 
\[
\norm{\hat{\boldsymbol{X}}_0 +\boldsymbol{U}_t}_{\textnormal{F}} \leq (1+\Gamma)\sqrt{N}.
\]
According to~\eqref{eq: bound difference Z}, we therefore obtain on the event \(\mathcal{E}^{\textnormal{c}}\),
\begin{equation} \label{eq: bound difference Z 2}
\norm{\boldsymbol{Z}_t-\boldsymbol{Z}_s}_{\textnormal{F}} \leq r \Gamma (1+\Gamma) \sqrt{N} \sqrt{t-s} + \sqrt{2} r \Gamma \sqrt{N} (t-s).
\end{equation} 

Now, for a constant \(a \in (0,1)\) we let \(N_{\frac{a}{N}}([0,T])\) be a \(\frac{a}{N}\)-net of the interval \([0,T]\). According to~\eqref{eq: bound difference Z 2}, for every \(t \in [0,T]\), there exists \(\tilde{t} \in N_{\frac{a}{N}}([0,T])\) such that 
\[
\norm{\boldsymbol{Z}_t - \boldsymbol{Z}_{\tilde{t}}}_{\textnormal{F}} \leq r \sqrt{a} \Gamma (1 + \Gamma). 
\]
Combining Lemma~\ref{lem: concentration GF} with a union bound, for every \(1 \leq i,j \leq r\) we obtain that
\[
\mu_{N \times r} \otimes \mathbb{P} \left(\sup_{t \in N_{\frac{a}{N}}([0,T])} \frac{1}{N} \left | \langle \boldsymbol{Z}_t, [\boldsymbol{v}_i]_j \rangle \right | \geq \gamma\right) \leq \frac{N T}{a} C(r) \exp(-c(r) N\gamma^2). 
\]
Then, for every \(t \in [0,T]\), there exists \(\tilde{t} \in N_{\frac{a}{N}}([0,T])\) such that 
\[
\frac{1}{N} \left | \langle \boldsymbol{Z}_t,[\boldsymbol{v}_i]_j \rangle \right | = \frac{1}{N} \left | \langle \boldsymbol{Z}_{\tilde{t}}, [\boldsymbol{v}_i]_j \rangle +\langle \boldsymbol{Z}_t -\boldsymbol{Z}_{\tilde{t}},[\boldsymbol{v}_i]_j \rangle \right | \leq \frac{1}{N} \left |\langle \boldsymbol{Z}_{\tilde{t}}, [\boldsymbol{v}_i]_j \rangle \right | + \frac{\sqrt{a} \Gamma (1+\Gamma)}{N},
\]
where we used Cauchy-Schwarz to bound \(\left |\langle \boldsymbol{Z}_t -\boldsymbol{Z}_{\tilde{t}},[\boldsymbol{v}_i]_j \rangle\right | \le \norm{\boldsymbol{Z}_t -\boldsymbol{Z}_{\tilde{t}}}_{\textnormal{F}} \norm{[\boldsymbol{v}_i]_j}_{\textnormal{F}} \leq \sqrt{a} \Gamma (1 + \Gamma)\). This then implies that 
\[
\mu_{N \times r} \otimes \mathbb{P} \left(\sup_{t \in [0,T]} \frac{1}{N} \left | \langle \boldsymbol{Z}_t, [\boldsymbol{v}_i]_j \rangle \right | \geq \gamma \right) \leq \frac{NT}{a} C(r)\exp(-c(r,a,\Gamma) N \gamma^2). 
\]
We note that \(R^N_{\boldsymbol{X}}(\nabla H_0(\hat{\boldsymbol{X}}_t)) = \boldsymbol{Z}_t\boldsymbol{H}^\top\). By the Cauchy-Schwarz inequality and orthonormality of \(\boldsymbol{H}\) we have that
\[
\frac{1}{N} \left | \langle R^N_{\hat{\boldsymbol{X}}_0}(\nabla H_0(\hat{\boldsymbol{X}}_t)),[\boldsymbol{v}_i]_j \rangle \right | \leq \frac{r}{N} \left |\langle \boldsymbol{Z}_t, [\boldsymbol{v}_i]_{j} \rangle \right |,
\] 
so that
\[
\mu_{N \times r} \otimes \mathbb{P} \left(\sup_{t \in [0,T]} \frac{1}{N} \left | \langle R^N_{\hat{\boldsymbol{X}}_0}(\nabla H_0(\hat{\boldsymbol{X}}_t)),[\boldsymbol{v}_i]_j \rangle \right | \geq \gamma \right) \leq \frac{NT}{a} C(r)\exp(-c(r,a,\Gamma) N\gamma^2 / r^2). 
\]
Additionally, we note that \(\nabla H_0(\hat{\boldsymbol{X}}_t)\) can be written as
\[
\nabla H_0(\hat{\boldsymbol{X}}_t) = R^N_{\boldsymbol{X}} (\nabla H_0(\hat{\boldsymbol{X}}_t)) \left(\boldsymbol{I}_r + \frac{1}{N} \nabla H_0(\hat{\boldsymbol{X}}_t)^\top \nabla H_0(\hat{\boldsymbol{X}}_t)\right)^{1/2} - \hat{\boldsymbol{X}}_0. 
\] 
Since \(\norm{\nabla H_0(\hat{\boldsymbol{X}}_t)}_{\textnormal{F}}^2\) is decreasing and on the event \(\mathcal{E}^{\textnormal{c}}\) it holds that 
\[
\norm{\nabla H_0(\hat{\boldsymbol{X}}_0)}_{\textnormal{F}}^2 \leq N \norm{\nabla H_0(\hat{\boldsymbol{X}}_0)}_\infty^2 \le \Gamma^2 N, 
\]
the matrix \(\left(\boldsymbol{I}_r + \frac{1}{N} \nabla H_0(\hat{\boldsymbol{X}}_t)^\top \nabla H_0(\hat{\boldsymbol{X}}_t)\right)^{1/2}\) has bounded spectral norm for all \(t \ge 0\). This implies that 
\[
\frac{1}{N}\langle  \nabla H_0(\hat{\boldsymbol{X}}_t), [\boldsymbol{v}_i]_j \rangle \leq \sqrt{1 + \Gamma^2} \frac{1}{N} \langle R^N_{\hat{\boldsymbol{X}}_0} (\nabla H_0(\hat{\boldsymbol{X}}_t)),[\boldsymbol{v}_i]_j \rangle - \frac{1}{N}\langle \hat{\boldsymbol{X}}_0,[\boldsymbol{v}_i]_j \rangle.
\]
We thus reach
\[
\begin{split}
& \mu_{N \times r} \otimes \mathbb{P} \left(\sup_{t \in [0,T]} \frac{1}{N} \vert \langle  \nabla H_0(\hat{\boldsymbol{X}}_t), [\boldsymbol{v}_i]_j \rangle \vert \geq \gamma \right) \\
& \leq \mu_{N \times r} \otimes \mathbb{P} \left(\frac{1}{N} \left | \langle \hat{\boldsymbol{X}}_0,[\boldsymbol{v}_i]_j \rangle \right | \geq \frac{\gamma}{2}\right) + \mu_{N \times r} \otimes \mathbb{P} \left(\sup_{t \in [0,T]} \frac{1}{N} \left | \langle R^N_{\hat{\boldsymbol{X}}_0} (\nabla H_0(\hat{\boldsymbol{X}}_t)),[\boldsymbol{v}_i]_j \rangle \right | \geq \frac{\gamma}{2 \sqrt{1 + \Gamma^2}} \right).
\end{split}
\]
This completes the proof of Lemma~\ref{lem: concentration cond infinity} upon combining the fact that $\hat{\boldsymbol{X}}_0 \sim \mu_{N \times r}$ with the deviation inequality obtained above for the second term of the right hand side of the previous line.
\end{proof}

\printbibliography
\end{document}